\documentclass[11pt,reqno,a4paper]{amsart}

\usepackage{amsmath,amsfonts,amsthm,mathrsfs,amssymb,cite}
\usepackage[usenames]{color}
\usepackage{graphicx}
\usepackage{latexsym} 
\usepackage{enumerate}
\newtheorem{thm}{Theorem}[section]
\newtheorem{lem}[thm]{Lemma}
\newtheorem{cor}[thm]{Corollary}
\newtheorem{prop}[thm]{Proposition}

\theoremstyle{remark}
\newtheorem{defn}[thm]{Definition}
\newtheorem{rem}[thm]{Remark}

\renewcommand{\eqref}[1]{\textnormal{(\ref{#1})}}

\numberwithin{equation}{section}

\newcommand{\rmi}{\mathrm{i}}
\newcommand{\rme}{\mathrm{e}}
\newcommand{\rmd}{\mathrm{d}}

\title[Mosco convergence for the Maxwell system]{Mosco convergence for $H(\mathrm{curl})$ spaces, higher integrability for Maxwell's equations, 
and 
stability in direct and inverse EM scattering problems}

\author{Hongyu Liu}
\address{Department of Mathematics, Hong Kong Baptist University, Kowloon, Hong Kong SAR.}
\email{hongyuliu@hkbu.edu.hk}

\author{Luca Rondi}
\address{Dipartimento di Matematica e Geoscienze, Universit\`a degli Studi di Trieste, Trieste, Italy.}
\email{rondi@units.it}

\author{Jingni Xiao}
\address{Department of Mathematics, Hong Kong Baptist University, Kowloon, Hong Kong SAR. }
\email{xiaojn@live.com}

\begin{document}

\begin{abstract}
This paper is concerned with the scattering problem for time-harmonic electromagnetic waves, due to the presence of scatterers and of inhomogeneities in the medium.

We prove a sharp stability result for the solutions to the direct electromagnetic scattering problem, with respect to variations of the scatterer and of the inhomogeneity, under minimal regularity assumptions for both of them.
The stability result leads to bounds on solutions to the scattering problems
which are uniform for an extremely general class of admissible scatterers and inhomogeneities.

These uniform bounds are a key step to tackle the challenging stability issue for the corresponding inverse electromagnetic scattering problem. In this paper we establish two optimal stability results of logarithmic type for the determination of
polyhedral scatterers by a minimal number of electromagnetic scattering measurements.

In order to prove the stability result for the direct electromagnetic scattering problem, we  study two fundamental issues in the theory of Maxwell equations:
Mosco convergence for $H(\mathrm{curl})$ spaces and higher integrability properties of solutions to Maxwell equations in nonsmooth domains.

\medskip


\medskip

\end{abstract}




\maketitle

\section{Introduction}\label{sec:Intro}

We are concerned with the electromagnetic scattering problem, in the time-harmonic case, which is governed by the Maxwell system as follows. Given an incident time-harmonic electromagnetic wave, characterised by the incident electric and magnetic fields $(\mathbf{E}^i,\mathbf{H}^i)$, we seek a pair of functions $(\mathbf{E},\mathbf{H})$ solving the 
following exterior boundary value problem
\begin{equation}\label{directscattering00}
\left\{\begin{array}{ll}
\nabla\wedge \mathbf{E}-\rmi k\mu \mathbf{H}=0
&\text{in }G=\mathbb{R}^3\backslash \Sigma\\
\nabla\wedge \mathbf{H}+\rmi k \epsilon \mathbf{E}=0
&\text{in }G=\mathbb{R}^3\backslash \Sigma\\
(\mathbf{E},\mathbf{H})=(\mathbf{E}^i,\mathbf{H}^i)+(\mathbf{E}^s,\mathbf{H}^s)
&\text{in }G=\mathbb{R}^3\backslash \Sigma\\
\nu\wedge \mathbf{E}=0&\text{on }\partial G=\partial\Sigma\\
\lim_{r\to+\infty}r\left(\frac{\mathbf{x}}{\|\mathbf{x}\|}\wedge \mathbf{H}^s(\mathbf{x})+
\mathbf{E}^s(\mathbf{x})\right)=0 & r=\|\mathbf{x}\|.
\end{array}
\right.
\end{equation}
Here $k>0$ is the wavenumber, $\Sigma$ is a scatterer, that is, a compact subset with connected complement, present in the medium, $\epsilon$ and $\mu$ are the electric permittivity and magnetic permeability of the medium, respectively. The first two equations satisfied by $(\mathbf{E},\mathbf{H})$ outside $\Sigma$ are the time-harmonic Maxwell equations.
The total electric and magnetic fields $(\mathbf{E},\mathbf{H})$ are the juxtaposition of the incident ones and the reflected, or scattered, ones $(\mathbf{E}^s,\mathbf{H}^s)$.
The incident fields $\mathbf{E}^i$ and $\mathbf{H}^i$ are an entire solution of the Maxwell equations with $\epsilon=\mu=I_3$. Here, and also in what follows, $I_3$ denotes the $3\times 3$ identity matrix. For instance, one can take $(\mathbf{E}^i,\mathbf{H}^i)$ to be a normalised electromagnetic plane wave with given polarisation vector and direction of propagation. The presence of the scatterer $\Sigma$, and of an inhomogeneity in $(\epsilon,\mu)$ around it, perturbs the incident wave through the creation of the scattered wave, characterised by the fact that it is radiating, that is, its corresponding fields satisfy the so-called Silver-M\"uller radiation condition, which is the last limit in \eqref{directscattering00}. The boundary condition on the boundary of $\Sigma$ depends on the physical properties of the scatterer. In this case, $\Sigma$ is a perfectly electric conducting scatterer, however, in a completely equivalent way, we can also treat the case of perfectly magnetic conducting scatterers.

We assume minimal regularity assumptions on the scatterer $\Sigma$ and the inhomogeneity surrounding it, that is, on the coefficients $\epsilon$ and $\mu$. Indeed, we assume that $\epsilon$ and $\mu$ are measurable functions in $G$ with values in the set of positive definite symmetric matrices, and that they coincide with the identity matrix outside a large ball. We also have uniform positive bounds, from below and above, respectively, on the lowest and highest eigenvalues of $\epsilon$ and $\mu$, all over $G$.

The existence and uniqueness of a solution to \eqref{directscattering00} are classical results when $\Sigma$ and the coefficients $\epsilon$ and $\mu$ are smooth; for example when the medium is homogeneous and isotropic, see for instance \cite{CK,Ned}. Also well-understood is the case of Lipschitz scatterer, see for instance \cite{Mon}. However, there has been a lot of effort to obtain existence and uniqueness results under minimal assumptions on $\Sigma$ and the coefficients $\epsilon$ and $\mu$, that allow, for instance, $\Sigma$ to be formed by screens and not obstacles only, and $\epsilon$ and $\mu$ to represent an inhomogeneous and anisotropic medium. We refer to \cite{PWW} for a very general result in this direction and to its references for the development of this topic.
The result in \cite{PWW} allows to prove existence and uniqueness for \eqref{directscattering00} provided $\Sigma$ satisfies the so-called Rellich and Maxwell compactness properties (see Definition~\ref{regularityscatterer}) and the Maxwell equations satisfy the unique continuation property. A sufficient condition for the  
Rellich and Maxwell compactness properties to hold is given in \cite{PWW} and a slightly different one, suited to our purposes, can be found in Proposition~\ref{suffMCP}. For the unique continuation property, Lipschitz regularity of the coefficients $\epsilon$ and $\mu$ is enough, see \cite{Ngu-Wan}; actually, in the 
scale of $C^{0,\alpha}$ regular functions, Lipschitz regularity is optimal as unique continuation may fail for H\"older continuous functions, see \cite{Dem}.
It might be possible that, in the scale of Sobolev regular functions, unique continuation may hold, as for the scalar elliptic equations in divergence form, with less regular coefficients provided some extra structure is assumed, for example isotropy. In fact, we conjecture that in the isotropic case $W^{1,3}$ regularity might be enough, but up to our knowledge such a result is still unavailable in the literature. For this reason, in this paper we consider only the Lipschitz regularity assumption, and we are sure that the interested reader would be able to easily adapt our arguments to other cases as long as unique continuation holds true. Our aim here, instead, is to allow the coefficients to be piecewise smooth, particularly piecewise Lipschitz, and still keep the unique continuation property. Using
\cite{Ngu-Wan} as the basic tool,
this has been shown to be true in \cite{Bal-et-al} and a different, and slightly more general, formulation may be found in Proposition~\ref{piecewiseLipschitzprop}.

The aim of the present paper is to establish continuity properties, namely stability, of solutions to \eqref{directscattering00} with respect to variations of the scatterer $\Sigma$ and the coefficients of the inhomogeneity at the same time. We keep the assumptions on $\Sigma$ and the coefficients $\epsilon$ and $\mu$ to a minimum. Indeed, 
we use regularity assumptions not stronger than the ones required to establish the existence and uniqueness of  solutions to the scattering problem
\eqref{directscattering00}. The main difference is that we need to consider a quantitative version of these conditions. It is emphasised that we allow an extremely general class of scatterer in our study, that includes obstacles, screens and also quite complicated combinations of obstacles and screens at the same time.
The main stability result is stated in Theorem~\ref{continuitycor}.

We wish to point that, for each of the results in this paper, we aim to the greatest generality. In general we assume no regularity on the coefficients of the Maxwell system, unless where we need unique continuation properties, where we use a regularity of piecewise Lipschitz type, since Lipschitz continuity is optimal in this context, at least in the scale of $C^{0,\alpha}$ regular functions. The surfaces where discontinuities are allowed are also of extremely general kind. Also concerning the domains where the Maxwell system is defined, as pointed out above with respect to scatterers, these are always as general as possible, with boundaries which may have a rather complex structure. The classes of domains we use, described in 
Subsection~\ref{classes-subsec}, were introduced in \cite{MR} and further developed in \cite{LPRX}. We believe that, even if improvements may be obtained in this direction, their generality and versatility put them at the top of the art in this moment.

The stability of the direct scattering problem is of significant importance in the theory of electromagnetism. For example, it is an indispensable ingredient in showing convergence of numerical solutions, when $\Sigma$ and the coefficients $\mu$ and $\epsilon$ are approximated by discrete variables. It could also be of use in proving existence of solutions for optimisation problems, when the scatterer or the medium need to satisfy suitable optimality conditions. 

The major motivation of the present work comes from the field of inverse problems, in particular from the inverse electromagnetic scattering problem. In such an inverse problem,
the scatterer or the surrounding inhomogeneity are unknown, or only partially known,
and the aim is to recover information on them
by sending one or more suitable time-harmonic
incident electromagnetic waves, and by measuring the perturbations produced by the presence of the scatterer and the inhomogeneity. Such measurements are usually performed far-away from the region where the scatterer and the inhomogeneity are contained, and in this case we speak of far-field or scattering measurements. That is, one usually measures the so-called far-field pattern either of the electric or of the magnetic scattered field, for each of the incident waves.
These kinds of inverse problems are of central importance to many areas of science and technology, including radar and sonar, geophysical exploration, medical imaging, nondestructive testing and remote sensing; see for instance \cite{CK, Isa2} and the references therein.

In order to tackle the stability issue of such an inverse scattering problem, a fundamental step is to obtain uniform bounds on the solution of the direct problem, that is, bounds which are independent of
the scatterer $\Sigma$ and the surrounding medium. In fact, in the inverse problem, the scatterer and the inhomogeneous medium are, at least in part, unknown, and, in general, we only know some a priori regularity assumptions on them. This is the reason why, for the applications, it is crucial to obtain such uniform bounds for the largest possible class of admissible scatterers and surrounding media, which is precisely one of the main aims of the present paper. In fact, our general stability result for the direct problem immediately leads to uniform bounds for solutions for an extremely wide class of scatterers and surrounding media;
Theorem~\ref{mainstabthm}.

In order to obtain the stability result for the direct electromagnetic scattering problem, two fundamental ingredients are needed. The first one is a suitable notion of Mosco convergence, and the second one is a higher integrability property of the solutions to Maxwell equations, independent on the domain of definition and the coefficients $\epsilon$ and $\mu$, provided they satisfy suitable, although minimal, assumptions.
These two results are the main technical achievements of the present paper and we comment them in the following.

Mosco convergence, introduced in \cite{Mos}, is a crucial tool for studying stability of solutions of elliptic problems with homogeneous Neumann conditions under perturbations of the domain. In fact, it can be shown that Mosco convergence of $H^1$ spaces associated to the domains of definition of the elliptic equations is essentially equivalent to convergence of the corresponding solutions with Neumann boundary conditions. Therefore, the Mosco convergence for $H^1$ spaces has attracted a lot of attention and several results appeared in the literature. 
In dimension $2$, the breakthrough was a sufficient condition proved in \cite{ChD}, which is still useful since it is easy to be verified. The problem was
finally solved in \cite{Buc-Var,Buc-Var2} where a sufficient and necessary condition is proved. In dimension $2$, the key is the use of duality arguments and complex analytic techniques. These methods can not be used in dimension $3$ and higher. In this case the first result is in \cite{Gia}, where suitable Lipschitz conditions are employed. A slightly different condition, still of Lipschitz type but more general in several respects, is considered in \cite{MR}.

In \cite{MR} it was shown that Mosco convergence of $H^1$ spaces is still the key point to obtain convergence results for solutions of Neumann problems for the Helmholtz equation. Moreover, for the Helmholtz equation, a quantitative version of the Rellich compactness property turns out to be essential. Such a quantitative version is given by higher integrability properties of $H^1$ functions of the domain of the equation,
which need to be uniform with respect to the considered domains.
We recall that the Rellich compactness property is crucial to obtain existence and uniqueness for the solution of the direct acoustic scattering problem with sound-hard scatterers.
Mosco convergence and such a higher integrability property for $H^1$ spaces, in fact, allowed to prove, 
still in \cite{MR},
convergence results, and correspondingly uniform bounds, for the direct acoustic scattering problem with respect to variations of the sound-hard scatterer. The corresponding result for sound-soft scatterers, that is, in the Dirichlet case, was considered earlier in \cite{Ron1}.

One of the main aims of the current paper is to extend the analysis for the acoustic case, \cite{MR}, to the technically much more challenging electromagnetic case. In the electromagnetic case, the suitable spaces, in which the solutions are to be sought, are the $H(\mathrm{curl})$ spaces. We have therefore developed the corresponding study for the Mosco convergence of $H(\mathrm{curl})$, which is stated as Theorem~\ref{Moscoconv}. We notice that we are able to prove the Mosco convergence for $H(\mathrm{curl})$ spaces under the same assumptions on the domains used for the $H^1$ spaces in \cite{MR}, which are quite general. We recall that we are restricted here to the  three-dimensional case, so duality arguments may not be used.

As in the acoustic case, higher integrability properties of solutions are also required. Therefore, a quantitative version of the Maxwell compactness property, stated in Definition~\ref{highintdefin},
need to be established. For Lipschitz domains, such a higher integrability result was proved in \cite{Dru}. Since no regularity condition is imposed on the coefficients, we believe that this result is essentially optimal and we recall it as Proposition~\ref{Druprop}. Our aim is to use this result as an essential tool to extend the higher integrability result to domains whose boundaries have  much weaker regularity and may have a quite complex structure, including screen-type portions, for instance. The result is given in
Proposition~\ref{highintprop}, and we observe that here no regularity is required on the coefficients.

Theorem~\ref{Moscoconv} and Proposition~\ref{highintprop} provide the right conditions on the scatterers that guarantee the stability of the direct electromagnetic scattering problem. For what concerns the coefficients $\epsilon$ and $\mu$ of the inhomogeneity, we need to find a class of piecewise Lipschitz coefficients, that in particular guarantee unique continuation, which is compact with respect to a suitable, possibly rather weak, convergence. Such a class is defined in Definition~\ref{coeffdefn} and its main compactness property is stated in Lemma~\ref{coeffcompactlemma}.
With these ingredients at our disposal, the stability result, Theorem~\ref{continuitycor}, 
for the solutions to the direct electromagnetic scattering problem can be finally obtained.
In turn, we prove corresponding uniform bounds, with respect to a wide class of 
scatterers and inhomogeneities, Theorem~\ref{mainstabthm}.

As important consequence and application of our uniform estimates on solutions to the direct electromagnetic scattering problems, we establish optimal stability estimates for the following inverse electromagnetic scattering problem. We assume that there is no inhomogeneity, that is $\epsilon=\mu=I_3$ everywhere, but there is some unknown scatterer $\Sigma$ to be determined. This is attempted by performing scattering measurements, that is by sending one or more time-harmonic
incident electromagnetic waves and by measuring the far-fields patterns of the corresponding scattered waves produced by the presence of the scatterer. We say that we have a single scattering measurement if we perform this experiment just once, that is, we send just one incident wave. We say that we have two or more scattering measurements if we send two or more, suitably chosen, incident waves and measure their corresponding scattered waves.

The inverse scattering problem described above is nonlinear, and it is also known to be ill-posed in the sense of Hadamard. Moreover, it is readily seen that the inverse problem is formally-posed
if we perform a single measurement; that is, physically speaking, one expects to recover $\Sigma$ by sending a single incident plane wave with fixed wavenumber, polarisation vector and incident direction, and then collecting the far-field data in every possible observation direction.
Though there is a widespread belief about unique determination of scatterers, or at least obstacles, by a single measurement, such uniqueness result still remains a largely open problem with very limited progress in the literature, also in the, relatively simpler, acoustic case. However, scatterers with a particular structure may be actually determined by one or a few scattering measurements.
This is true, for example, for scatterers of polyhedral type, that is, scatterers whose boundaries are made of a collection of suitable portions of hyperplanes.
The breakthrough in this direction was done for the acoustic case in \cite{C-Y}, where sound-soft polyhedral obstacles, that is, polyhedra, with an additional nontrapping condition, were considered.
For sound-soft polyhedral scatterers the problem was completely solved in \cite{AR}, where it is shown that a single measurement is enough to uniquely determine any sound-soft polyhedral scatterer, without any further condition and in any dimension. We notice that such polyhedral scatterers include obstacles, that is, polyhedra, and screen-like scatterers, at the same time. For the acoustic sound-hard case in dimension $N$, $N$ scattering measurements are enough for determining a general sound-hard polyhedral scatterer, and the number of measurements may not be reduced if screen-type scatterers are allowed, \cite{Liu-Zou, Liu-Zou3}. However, for polyhedral obstacles, that is, polyhedra, a single measurement is enough, \cite{Els-Yam1, Els-Yam2}
We notice that the acoustic sound-hard case is the closer one to the electromagnetic case.

About the electromagnetic case, in \cite{Liu3} and \cite{Liu}, uniqueness results for the determination of conducting scatterers of general polyhedral type by scattering measurements were established. As in the acoustic case, there might be finitely many (with an unknown number) solid polyhedra and screen-type scatterers simultaneously. Two scattering measurements are needed, \cite{Liu3} for general polyhedral scatterers, whereas a single scattering measurement is enough for polyhedral obstacles, \cite{Liu}. The number of scattering measurements has also been shown to be optimal.

In all of these cases, uniqueness is achieved by exploiting a reflection principle,
along with the use of a certain path argument,
for the Helmholtz equation in the acoustic case, and for the Maxwell equations in the electromagnetic one.

The corresponding stability issue has been treated first in dimension $3$ for general sound-soft polyhedral scatterers, \cite{Ron2}, with a single scattering measurement and under minimal regularity assumptions. This result has been extended to any dimension and for a more general and versatile class of polyhedral scatterers in \cite{LPRX}. More importantly, in \cite{LPRX} the stability result has been extended to the sound-hard case, for the same class of polyhedral scatterers as in the sound-soft case, and
with the minimal amount of
scattering measurements, that is, $N$ measurements for general polyhedral scatterers and a single measurement for polyhedra, at least for $N=2,3$ in this last case.

In all these stability results, a crucial preliminary result is to establish bounds for the solutions of the direct scattering problems which are uniform for the widest possible class of admissible scatterers. For the sound-soft case this was done in \cite{Ron2}, whereas for the sound-hard case such a result follows from \cite{MR}.

The analysis developed for the acoustic sound-hard case can be adapted
to extending these stability results to the electromagnetic case, by following the spirit of the uniqueness arguments of \cite{Liu3} and \cite{Liu}, and with some necessary modifications. However, the required uniform bounds for solutions of the corresponding direct electromagnetic scattering problem
turns out to be a highly nontrivial modification of known results in the acoustic case. 
Using the stability results established earlier for the direct electromagnetic scattering problem, we are able to prove the following essentially optimal stability results for the inverse problem. In Theorem~\ref{mainteoN}, we show a stability result for the determination of general polyhedral scatterers, including screen-type ones, by two electromagnetic scattering measurements. In Theorem~\ref{mainteo1} we present a stability result for the determination of polyhedra, that is, of polyhedral obstacles, by a single electromagnetic scattering measurement. We notice that the assumptions on the unknown scatterers are the same as those used in the acoustic case in \cite{LPRX}, and also the stability estimates obtained are exactly of the same logarithmic type.

The plan of the paper is as follows. In Section~\ref{sect:1} we consider the direct electromagnetic scattering problem. In Subsection~\ref{subsect:1.0}, we set most of the standard notations needed in the sequel.
Moreover, in Proposition~\ref{Linftydensity}, we prove that the subspace $L^{\infty}\cap H(\mathrm{curl})$ is dense in the space $H(\mathrm{curl})$ for an arbitrary open set; a result that, rather surprisingly, we could not find in this generality in the literature. 
In Subsection~\ref{subsect:1.1} we formulate the problem and we treat existence and uniqueness of its solution in Subsection~\ref{subsect:ex+un}. In particular, in Subsection~\ref{subsection1:3}, we discuss the Rellich and Maxwell compactness properties, as well as the unique continuation property for the Maxwell equations.

Section~\ref{properties-section} is devoted to some preliminary results. 
In Subsection~\ref{prop-subsec}, properties of the solutions to the Maxwell equations are considered. In particular, changes of variables and reflection principles for solutions to the Maxwell equations are studied in detail. In Subsection~\ref{classes-subsec}, we discuss suitable classes of admissible scatterers. These kinds of classes were introduced in \cite{MR} and further developed in \cite{LPRX}.

The main technical section of the paper is Section~\ref{Moscosec}. We prove the Mosco convergence result for $H(\mathrm{curl})$ spaces, Theorem~\ref{Moscoconv},
and the higher integrability properties for solutions to the Maxwell equations in nonsmooth domains, Proposition~\ref{highintprop}

The main section of the paper is Section~\ref{stabilitysec}. We consider suitable classes of admissible coefficients for the inhomogeneities, Definition~\ref{coeffdefn}, and discuss their properties. Then we state and prove the main results of the paper, the stability result, Theorem~\ref{continuitycor}, and the uniform bounds, Theorem~\ref{mainstabthm}, for the solutions to the direct electromagnetic scattering problem.

In Section~\ref{polystabsec}, we deal with the inverse electromagnetic scattering problem, for polyhedral scatterers. The stability results, Theorem~\ref{mainteoN} for general polyhedral scatterers with two measurements and 
Theorem~\ref{mainteo1} for polyhedra with one measurement, are presented in Subsection~\ref{finalsub0}. A few details and comments about the proofs of these stability results are in the final Subsection~\ref{finalsub}.

Finally, in the Appendix, there are the proofs of Propositions~\ref{Linftydensity}
and \ref{suffMCP}.

\section{The electromagnetic scattering problem}\label{sect:1}

\subsection{Notations and conditions}\label{subsect:1.0}
We shall use the following notations.
The integer $N\geq 2$ always denotes the space dimension. Apart from a few instances, $N$ will be equal to $3$.
In what follows, we always omit the dependence of constants on the space dimension $N$.

For any $\mathbf{x}\in\mathbb{R}^N$, $N\geq 2$,  we denote $\mathbf{x}=(\mathbf{x}',x_N)\in\mathbb{R}^{N-1}\times \mathbb{R}$ and $\mathbf{x}=(\mathbf{x}'',x_{N-1},x_N)\in\mathbb{R}^{N-2}\times\mathbb{R}\times\mathbb{R}$.
For any $s>0$ and any $\mathbf{x}\in\mathbb{R}^N$, 
$B_s(\mathbf{x})$ denotes the ball contained in $\mathbb{R}^N$ with radius $s$ and center $\mathbf{x}$, and $B'_s(\mathbf{x}')$ denotes the ball contained in $\mathbb{R}^{N-1}$ with radius $s$ and center $\mathbf{x}'$.
We also use $B_s=B_s(0)$ and $B'_s=B'_s(0)$.
For any $E\subset \mathbb{R}^N$, we denote $B_s(E)=\bigcup_{\mathbf{x}\in E}B_s(\mathbf{x})$.
For any ball $B$ centered at zero we denote $B^{\pm}=B\cap\{\mathbf{y}\in\mathbb{R}^N:\ y_N\gtrless 0\}$.
Analogously, for any hyperplane $\Pi$ in $\mathbb{R}^N$, we use the following notations. If, with respect to a suitable Cartesian coordinate system, one has $\Pi=\{\mathbf{y}\in\mathbb{R}^N:\ y_N=0\}$; then for any $\mathbf{x}\in \Pi$ and any $r>0$ we denote $B^{\pm}_r(\mathbf{x})=
B_r(\mathbf{x})\cap\{y_N\gtrless 0\}$. Furthermore, we denote by $T_{\Pi}$ the reflection with respect to $\Pi$, namely in this case for any $\mathbf{y}=(y_1,\ldots,y_{N-1},y_N)\in\mathbb{R}^N$ we have
$T_{\Pi}(\mathbf{y})=(y_1,\ldots,y_{N-1},-y_N)$.


The $s$-dimensional Hausdorff measure in $\mathbb{R}^N$, $0\leq s\leq N$, will be denoted by $\mathcal{H}^s$. We recall that $\mathcal{H}^N$ coincides with the Lebesgue measure. For any Borel set $E\subset \mathbb{R}^N$ we denote with $|E|$ its Lebesgue measure.

For any $\mathbf{x}\in \mathbb{R}^3$ we call
$r=\|\mathbf{x}\|$ and, provided $\mathbf{x}\neq 0$,  $\hat{\mathbf{x}}=\mathbf{x}/\|\mathbf{x}\|\in \mathbb{S}^2=\{\mathbf{x}\in \mathbb{R}^3:\ \|\mathbf{x}\|=1\}$. We finally define $\rho=(1+r^2)^{1/2}$.

We consider the following direct scattering problem for the electromagnetic (EM) case, in the presence of a scatterer. We say that $\Sigma\subset \mathbb{R}^3$ is a \emph{scatterer} if $\Sigma$ is a compact set (possibly empty) such that $G=\mathbb{R}^3\backslash \Sigma$ is connected.
On the other hand, an open connected set $G$ whose complement is compact will be referred to as an \emph{exterior domain}. Notice that $\mathbb{R}^3$ itself is an exterior domain.

Let us assume that outside $\Sigma$ we have a medium characterised by the electric permittivity $\epsilon_1$ and the magnetic permeability $\mu_1$. For the time being, we do not assume any regularity or isotropy conditions, that is, they are tensors in $G$ satisfying uniform ellipticity conditions. We recall that $M^{3\times 3}_{sym}(\mathbb{R})$ is the space of real-valued $3\times 3$ symmetric matrices and that, for any open set $D\subset\mathbb{R}^3$, we say that $a$ is a \emph{tensor} in $D$ satisfying \emph{uniform ellipticity conditions} if $a\in L^{\infty}(D,M^{3\times 3}_{sym}(\mathbb{R}))$ and it satisfies
\begin{equation}\label{ellipticity}
a_0\|\xi\|^2 \leq a(\mathbf{x})\xi\cdot\xi\leq a_1\|\xi\|^2\quad\text{for a.e. }\mathbf{x}\in D\text{ and every }\xi\in \mathbb{R}^3,
\end{equation}
where $0<a_0<a_1$ are the so-called \emph{ellipticity constants}. In short the above condition may be written as
$$a_0I_3\leq a(\mathbf{x})\leq a_1I_3\quad\text{for a.e. }\mathbf{x}\in D.$$

We also assume that, outside a bounded set, the space is homogeneous and isotropic, that is there exists $R_0>0$ and two positive constants $\epsilon_{\infty}$, $\mu_{\infty}$ such that
\begin{equation}\label{boundedness}
\Sigma=\mathbb{R}^3\backslash G\subset \overline{B_{R_0}}
\end{equation}
and
\begin{equation}\label{coefficientatinfty}
\epsilon_1(\mathbf{x})=\epsilon_{\infty}I_3\text{ and }
\mu_1(\mathbf{x})=\mu_{\infty}I_3\quad\text{for every }\mathbf{x}\text{ such that }\|\mathbf{x}\|>R_0.
\end{equation}

In the sequel, we fix a \emph{frequency} $\omega>0$, and we call
$k=\omega \sqrt{\epsilon_{\infty}\mu_{\infty}}>0$ the corresponding \emph{wavenumber} and $\epsilon=\epsilon_{\infty}^{-1}\epsilon_1$ and $\mu=\mu_{\infty}^{-1}\mu_1$. 
We shall take $k>0$ and assume that $\epsilon$ and $\mu$ belong to $L^{\infty}(G,M^{3\times 3}_{sym}(\mathbb{R}))$ and, for some positive constants $R_0$ and $0<\lambda_0<1<\lambda_1$, satisfy
\begin{equation}\label{ellipticity2}
\lambda_0 I_3\leq \epsilon(\mathbf{x}),\, \mu(\mathbf{x})\leq \lambda_1 I_3\quad\text{for a.e. }\mathbf{x}\in G
\end{equation}
and
\begin{equation}\label{infinity}
\epsilon(\mathbf{x})=\mu(\mathbf{x})=I_3\quad\text{for a.e. }\mathbf{x}\in G\text{ such that }\|\mathbf{x}\|>R_0.
\end{equation}

We say that $D\subset \mathbb{R}^3$ is a domain if it is an open connected set.
We say that a domain $D$ is \emph{Lipschitz} if for any $\mathbf{x}\in \partial D$ there exist $r>0$ and a Lipschitz function $\varphi:\mathbb{R}^2\to\mathbb{R}$, with Lipschitz constant $L$, such that, up to a rigid change of coordinates, we have
$$D\cap B_r(\mathbf{x})=\{\mathbf{y}=(y_1,y_2,y_3)\in B_r(\mathbf{x}):\ 
y_3<\varphi(y_1,y_2)\}.$$
Clearly, $r$, $L$, and the change of coordinates depend on the point $\mathbf{x}$. If we can use the same constant $r$ and $L$ for all $\mathbf{x}\in\partial D$, then we say that $D$ is a Lipschitz domain with constants $r$ and $L$. We remark that any Lipschitz domain $D$ with compact boundary is Lipschitz with constants $r$ and $L$ for suitable positive constants $r$ and $L$ depending on $D$.

Let $D$ 
and $D'$ be two open sets and $T:D\to D'$ be a function. We say that $T$ is locally $W^{1,\infty}$ if
$T\in W^{1,\infty}_{loc}(D)$ and we notice that this is equivalent to requiring that $T$ is locally Lipschitz. As usual we say that $T$ is Lipschitz (with Lipschitz constant $L$) if
$\|T(\mathbf{x})-T(\mathbf{y})\|\leq L\|\mathbf{x}-\mathbf{y}\|$ for any $\mathbf{x}$, $\mathbf{y}\in D$. 
We say that $T$ is a \emph{bi}-\emph{Lipschitz mapping} with constant $L$ if 
$T$ is bijective and
$T$ and $T^{-1}$ are Lipschitz with Lipschitz constant $L$, on $D$ and $D'$ respectively.
We say that $T$ is a \emph{bi}-$W^{1,\infty}$ \emph{mapping} with constant $L$ if $T$ is bijective and
$\|JT\|_{L^{\infty}(D)}$ and $\|J(T^{-1})\|_{L^{\infty}(D')}$ are both bounded by $L$.
Here, and in what follows, $JT$ denotes the Jacobian matrix of $T$.

The following remark will be of interest and a proof may be found, for instance, in \cite[Chapter~6]{Res}.

\begin{rem}\label{detsign}
Let $D$ and $D'$ be open sets and $T:D\to D'$ be a bi-$W^{1,\infty}$ mapping.
If $D$ is connected, then either $\det JT(\mathbf{x})>0$ for almost every $\mathbf{x}\in D$ or 
$\det JT(\mathbf{x})<0$ for almost every $\mathbf{x}\in D$.  
\end{rem}

We shall also need the following spaces.
Let $D\subset \mathbb{R}^3$ be a domain and let $a\in L^{\infty}(D,M^{3\times 3}_{sym}(\mathbb{R}))$ satisfying \eqref{ellipticity} with constants $0<a_0<a_1$. 

We define, essentially following \cite{PWW},
$$H(\mathrm{curl},D)=\{u\in L^2(D,\mathbb{C}^3):\ \nabla\wedge u\in L^2(D,\mathbb{C}^3)\}$$
and
$$H(\mathrm{div},D)=\{u\in L^2(D,\mathbb{C}^3):\ \nabla\cdot u\in L^2(D)\}.$$
We notice that $\nabla\wedge u$ and $\nabla\cdot u$ are always meant in the sense of distributions. Moreover, $L^2(D)=L^2(D,\mathbb{C})$. We notice that these are Hilbert spaces endowed with the usual graph norm.
We also need
$$H_{loc}(\mathrm{curl},D)=\{u\in L^2_{loc}(D,\mathbb{C}^3):\ \nabla\wedge u\in L^2_{loc}(D,\mathbb{C}^3)\}.$$
and the following two spaces
\begin{multline*}
H_0(\mathrm{curl},D)=\{u\in H(\mathrm{curl},D):\ 
\langle\nabla\wedge\ u,\phi\rangle-\langle u,\nabla \wedge\phi \rangle=0\\ \text{for any }\phi\in H(\mathrm{curl},D)\text{ with bounded support}\},
\end{multline*}
and 
\begin{multline*}
H_0(\mathrm{div},D)=\{u\in H(\mathrm{div},D):\ 
\langle\nabla\cdot u,\varphi \rangle+\langle u,\nabla\varphi \rangle=0\\ \text{for any }\varphi\in H^1(D)\text{ with bounded support}\}.
\end{multline*}
Here $\langle u,v\rangle=\int_{D}\overline{u}\cdot v$ is the usual $L^2$ scalar product,
either in $L^2(D)$ or in $L^2(D,\mathbb{C}^3)$.
We call $H(\mathrm{div}_a,D)$ and $H_0(\mathrm{div}_a,D)$ the set of $u\in L^2(D,\mathbb{C}^3)$ such 
that $au\in H(\mathrm{div},D)$ or $au\in H_0(\mathrm{div},D)$, respectively.

We notice that $u\in H_0(\mathrm{curl},D)$ satisfies, in a weak sense, the boundary condition
$$\nu\wedge u=0\quad\text{on }\partial D$$
whereas $u\in H_0(\mathrm{div}_a,D)$ satisfies, again in a weak sense, the boundary condition
$$\nu\cdot (au)=0\quad\text{on }\partial D.$$
As usual, $\nu$ is the exterior unit normal to $D$ on $\partial D$.

\begin{rem}
An important remark related to solutions to the Maxwell equations is the following.
If $\mathbf{E}\in H_0(\mathrm{curl},D)$ and $\nabla\wedge \mathbf{E}=c a \mathbf{H}$ in $D$, for some $c\in \mathbb{C}\backslash\{0\}$, then $\mathbf{H}\in H_0(\mathrm{div}_a,D)$. In fact, first of all we notice that
$\nabla\cdot (a\mathbf{H})=0$ in $D$ in the sense of distribution, and therefore
$\mathbf{H}\in H(\mathrm{div}_a,D)$. We need to show that for any $\varphi\in H^1(D)$  with bounded support we have
$\langle a\mathbf{H},\nabla\varphi \rangle=0$. Since $\nabla\varphi\in H(\mathrm{curl},D)$, with $\nabla\wedge (\nabla\varphi)=0$, for any $\varphi\in H^1(D)$,
$$\langle a\mathbf{H},\nabla\varphi \rangle=(1/\overline{c})\langle \nabla\wedge \mathbf{E},\nabla\varphi \rangle=(1/\overline{c})\left(\langle \nabla\wedge \mathbf{E},\nabla\varphi \rangle-
\langle  \mathbf{E},\nabla\wedge(\nabla\varphi) \rangle
\right)=0.$$
\end{rem}

Let us notice that, under certain smoothness assumptions on $D$, we can better characterise these spaces. The smoothness assumption is of Lipschitz type and the result is the following, see for instance \cite[Chapter~3]{Mon}.

\begin{prop}\label{Lipdomain}
Let $D$ be a bounded Lipschitz domain. 

Then $H(\mathrm{div},D)=\overline{C^{\infty}(\overline{D})}$ and $H_0(\mathrm{div},D)=\overline{C_0^{\infty}(D)}$ with respect to the $H(\mathrm{div})$ norm.

For any $u\in H(\mathrm{div},D)$, we can define $\gamma_{\nu}(u)=\nu\cdot u|_{\partial D}\in H^{-1/2}(\partial D)$ such that
$$\langle\nabla\cdot u,\varphi \rangle+\langle u,\nabla\varphi \rangle=
\langle\gamma_{\nu}(u),\varphi\rangle_{\partial D}
\quad\text{for any }\varphi\in H^1(D).$$
Here and in the sequel, $\langle\cdot,\cdot\rangle_{\partial D}$ is the $H^{-1/2}$-$H^{1/2}$ duality on $\partial D$.
Therefore, $H_0(\mathrm{div},D)=\{u\in H(\mathrm{div},D):\ \gamma_{\nu}(u)=0\text{ in }H^{-1/2}(\partial D)\}$.

Correspondingly,
$H(\mathrm{curl},D)=\overline{C^{\infty}(\overline{D})}$ and $H_0(\mathrm{curl},D)=\overline{C_0^{\infty}(D)}$ with respect to the $H(\mathrm{curl})$ norm.

For any $u\in H(\mathrm{curl},D)$, we can define $\gamma_{\tau}(u)=\nu\wedge u|_{\partial D}\in H^{-1/2}(\partial D,\mathbb{C}^3)$ such that
$$\langle \nabla\wedge u,\phi \rangle-\langle u,\nabla\wedge\phi \rangle=
\langle\gamma_{\tau}(u),\phi\rangle_{\partial D}
\quad\text{for any }\phi\in H^1(D,\mathbb{C}^3).$$
Therefore, $H_0(\mathrm{curl},D)=\{u\in H(\mathrm{curl},D):\ \gamma_{\tau}(u)=0\text{ in }H^{-1/2}(\partial D,\mathbb{C}^3)\}$.
\end{prop}

Also the following density result will be of use. This density is trivial in the case of Lipschitz open sets but for a general open set we could not find a reference for it.

\begin{prop}\label{Linftydensity}
Let $D$ be any open set contained in $\mathbb{R}^3$. Then
$H(\mathrm{curl},D)\cap L^{\infty}(D,\mathbb{C}^3)$ is dense in $H(\mathrm{curl},D)$ with respect to the $H(\mathrm{curl})$ norm.
\end{prop}

We postpone the proof of this result to the Appendix.

Finally, let us assume that $D$ is an exterior domain, that is, it contains the exterior of a ball. In order to control the behaviour at infinity, we shall use the following notation.
For any $t\in\mathbb{R}$ and any $u\in L^2_{loc}(D)$, either complex-valued or $\mathbb{C}^3$-valued, we set
$$\|u\|_{0,t}=\left(\int_D \rho^{2t}\|u\|^2\right)^{1/2},$$
where we recall that $\rho=(1+r^2)^{1/2}$, $r=\|\mathbf{x}\|$. We finally define
$$H(\mathrm{curl},D,t)=\{u\in H_{loc}(\mathrm{curl},D):\ \|u\|_{0,t}+
\|\nabla\wedge u \|_{0,t}<+\infty\}$$
and
$$H(\mathrm{div},D,t)=\{u\in H_{loc}(\mathrm{div},D):\ \|u\|_{0,t}+
\|\nabla\cdot u \|_{0,t}<+\infty\}.$$

\subsection{Mathematical formulation}\label{subsect:1.1}

The electromagnetic wave is described by the electric field $\mathscr{E}(\mathbf{x}, t)$ and the the magnetic field $\mathscr{H}(\mathbf{x},t)$ for $(\mathbf{x},t)\in G\times\mathbb{R}_+$. The electromagnetic wave propagation is governed by the Maxwell equations
\begin{equation}\label{eq:Maxwell1}
\nabla\wedge \mathscr{E}(\mathbf{x},t)+\mu_1\frac{\partial \mathscr{H}}{\partial t}(\mathbf{x},t)=0,\quad \nabla\wedge \mathscr{H}(\mathbf{x},t)-\epsilon_1\frac{\partial \mathscr{E}}{\partial t}(\mathbf{x},t)=0.
\end{equation}
For time-harmonic electromagnetic waves of the form
\[
\mathscr{E}(\mathbf{x},t)=\Re(\epsilon_{\infty}^{-1/2}\mathbf{E}(\mathbf{x})e^{-\rmi\omega t}),\quad \mathscr{H}(\mathbf{x},t)=\Re(\mu_{\infty}^{-1/2}\mathbf{H}(\mathbf{x})e^{-\rmi\omega t})
\]
where $(\mathbf{E}, \mathbf{H})(\mathbf{x})\in\mathbb{C}^3\times\mathbb{C}^3$ and $\omega>0$ is the frequency, it is directly verified that one has the reduced Maxwell equations
\begin{equation}\label{eq:Maxwell2}
\nabla\wedge \mathbf{E}(\mathbf{x})-\rmi k\mu \mathbf{H}(\mathbf{x})=0,\quad \nabla\wedge \mathbf{H}(\mathbf{x})+\rmi k \epsilon \mathbf{E}(\mathbf{x})=0,
\end{equation}
with the wavenumber $k>0$ and the coefficients $\epsilon$ and $\mu$ defined as in Subsection~\ref{subsect:1.0}.
Notice that, when the medium is homogenous and isotropic, that is $\epsilon_1$ and $\mu_1$ are simply positive constants coinciding with $\epsilon_{\infty}$ and $\mu_{\infty}$ everywhere, then \eqref{eq:Maxwell2}
reduces to
\begin{equation}\label{eq:Maxwell2bis}
\nabla\wedge \mathbf{E}(\mathbf{x})-\rmi k \mathbf{H}(\mathbf{x})=0,\quad \nabla\wedge \mathbf{H}(\mathbf{x})+\rmi k \mathbf{E}(\mathbf{x})=0.
\end{equation}

In the sequel we shall consider the system \eqref{eq:Maxwell2} with $k>0$ and $\epsilon$ and $\mu$ satisfying \eqref{ellipticity2} and \eqref{infinity}

The equation \eqref{eq:Maxwell2} is to be understood in the sense of distribution, and hence it is well-formulated, for instance, for $(\mathbf{E},\mathbf{H})\in (H_{loc}(\mathrm{curl},G))^2$.

We say that a solution $(\mathbf{E},\mathbf{H})$ to \eqref{eq:Maxwell2} is \emph{outgoing} or \emph{radiating} if it satisfies the following \emph{Silver-M\"uller radiation condition}
\begin{equation}\label{eq:radiation}
\lim_{r\to+\infty}r\left(\frac{\mathbf{x}}{\|\mathbf{x}\|}\wedge \mathbf{H}(\mathbf{x})+
\mathbf{E}(\mathbf{x})\right)=0,\quad r=\|\mathbf{x}\|
\end{equation}
which holds uniformly in all directions $\hat{\mathbf{x}}:=\mathbf{x}/\|\mathbf{x}\|\in \mathbb{S}^2$. The Silver-M\"uller radiation condition characterises the radiating nature of solutions to the Maxwell equations. We recall that this is equivalent to requiring that
\begin{equation}\label{eq:radiation2}
\lim_{r\to+\infty}r\left(\frac{\mathbf{x}}{\|\mathbf{x}\|}\wedge \mathbf{E}(\mathbf{x})-
\mathbf{H}(\mathbf{x})\right)=0,\quad r=\|\mathbf{x}\|,
\end{equation}
see for instance \cite{CK}.

In the case of a homogeneous and isotropic medium, the following connection between the Maxwell equations and the vectorial Helmholtz equation is known, see again \cite{CK}.

\begin{lem}\label{lem:MaxwHelm}
If $(\mathbf{E},\mathbf{H})$ is a solution to the Maxwell equations \eqref{eq:Maxwell2bis}, that is
$$
\nabla\wedge \mathbf{E} -\rmi k\mathbf{H}=0, \quad \nabla \wedge \mathbf{H}+\rmi k\mathbf{E}=0,
$$
then $\mathbf{E}$ and $\mathbf{H}$ satisfy the following vectorial Helmholtz equations,
\begin{equation}\label{eq:Helm1}
\Delta \mathbf{E} + k^2 \mathbf{E}=0,\quad \nabla\cdot \mathbf{E}=0,
\end{equation}
\begin{equation}\label{eq:Helm2}
\Delta \mathbf{H} + k^2 \mathbf{H}=0,\quad \nabla\cdot \mathbf{H}=0.
\end{equation}

Conversely, if $\mathbf{E}$ satisfies \eqref{eq:Helm1} \textnormal{(}or $\mathbf{H}$ satisfies \eqref{eq:Helm2}\textnormal{)}, then $\mathbf{E}$ and $\mathbf{H} =(\nabla\wedge \mathbf{E})/(\rmi k)$ \textnormal{(}or $\mathbf{H}$ and $\mathbf{E} =- (\nabla\wedge \mathbf{H})/(\rmi k)$, respectively\textnormal{)} satisfy the Maxwell equations \eqref{eq:Maxwell2bis}.

Moreover, for $(\mathbf{E},\mathbf{H})$, a solution to the Maxwell equations \eqref{eq:Maxwell2bis}, the Silver-M\"uller radiation condition
\eqref{eq:radiation}
is equivalent to the Sommerfeld radiation condition for all components of $\mathbf{E}$ and $\mathbf{H}$; that is
\begin{equation}\label{eq:radiationSommer}
\lim_{r\to+\infty}r\left(\frac{\partial u}{\partial r} -\rmi ku \right ) =0,\quad r=\|\mathbf{x}\|,
\end{equation}
where the limit holds uniformly in all directions $\hat{\mathbf{x}}=\mathbf{x}/\|\mathbf{x}\|\in\mathbb{S}^2$, and $u$ denotes any of the Cartesian components of 
$\mathbf{E}$ or $\mathbf{H}$.
\end{lem}

Using the asymptotic behaviour of outgoing solutions to the Helmholtz equation,
we can deduce the following asymptotic behaviour of outgoing solutions of the Maxwell equations. That is, as $r=\|\mathbf{x}\|\to+\infty$,
\begin{equation}\label{eq:farfield}
\begin{array}{l}
\displaystyle{\mathbf{E}(\mathbf{x})=\frac{\rme^{\rmi k\|\mathbf{x}\|}}{\|\mathbf{x}\|} {\mathbf{E}_\infty}(\hat{\mathbf{x}})+\mathcal{O}\left(\frac{1}{\|\mathbf{x}\|^2}\right),}\\
\displaystyle{\mathbf{H}(\mathbf{x})=\frac{\rme^{\rmi k\|\mathbf{x}\|}}{\|\mathbf{x}\|} \mathbf{H}_\infty(\hat{\mathbf{x}})+\mathcal{O}\left(\frac{1}{\|\mathbf{x}\|^2}\right),}
\end{array}
\end{equation}
which hold uniformly in $\hat{\mathbf{x}}=\mathbf{x}/\|\mathbf{x}\|\in\mathbb{S}^2$. Here $\mathbf{E}_\infty$ and $\mathbf{H}_\infty$ are complex-valued functions defined on $\mathbb{S}^2$ and
denote the \emph{electric and magnetic far-field patterns}, respectively, and they satisfy
for any $\hat{\mathbf{x}}\in\mathbb{S}^2$,
\begin{equation}\label{eq:relation}
\mathbf{H}_\infty(\hat{\mathbf{x}})=\hat{\mathbf{x}}\wedge \mathbf{E}_\infty(\hat{\mathbf{x}})\quad \mbox{and}\quad \hat{\mathbf{x}}\cdot \mathbf{E}_\infty(\hat{\mathbf{x}})=\hat{\mathbf{x}}\cdot \mathbf{H}_\infty(\hat{\mathbf{x}})=0.
\end{equation}
It is also known that $\mathbf{E}_\infty$ and $\mathbf{H}_\infty$ are real-analytic  and hence if they are given on any open patch of $\mathbb{S}^2$, then they are known on the whole unit sphere by analytic continuation.

For the scatterer $\Sigma$ defined earlier, the EM wave cannot penetrate inside the scatterer and the Maxwell system \eqref{eq:Maxwell2} is defined only in $G=\mathbb{R}^3\backslash \Sigma$. Depending on the physical property of the scatterer, one would have the following boundary conditions on $\partial \Sigma$,
\begin{equation}\label{eq:PEC}
{\nu}({\mathbf{x}})\wedge \mathbf{E}(\mathbf{x})=0,\quad \mathbf{x}\in\partial \Sigma,
\end{equation}
corresponding to a \emph{perfectly electric conducting} (PEC) scatterer $\Sigma$; or
\begin{equation}\label{eq:PMC}
{\nu}(\mathbf{x})\wedge \mathbf{H}(\mathbf{x})=0,\quad \mathbf{x}\in\partial \Sigma,
\end{equation}
corresponding to a \emph{perfectly magnetic conducting} (PMC) scatterer $\Sigma$. Here, $\nu\in\mathbb{S}^2$
denotes the exterior unit normal vector to $G$ (or interior unit normal to $\Sigma$) on $\partial G=\partial\Sigma$.

\begin{rem}\label{PEC-PMC}
Let us remark here once for all that if $(\mathbf{E},\mathbf{H})$ solves \eqref{eq:Maxwell2}, then $(\mathbf{H},-\mathbf{E})$ solve the same equation provided we swap $\epsilon$ with $\mu$. Moreover, if $(\mathbf{E},\mathbf{H})$ is an outgoing solution to \eqref{eq:Maxwell2}, then $(\mathbf{H},-\mathbf{E})$ is also outgoing.
Therefore, by this kind of change of variables, we can always turn a perfectly magnetic conducting scatterer $\Sigma$ into a perfectly electric conducting scatterer. Hence, in what follows, we shall only consider the perfectly electric conducting case but it is clear that all the results of this paper hold for the perfectly magnetic conducting case as well.
\end{rem}

Summarising our discussion above, the EM scattering problem we are interested in is the following. We fix $k>0$ and consider a pair of \emph{incident electric and magnetic fields} $(\mathbf{E}^i,\mathbf{H}^i)$ given by an entire solution to \eqref{eq:Maxwell2bis}. For example,
we can choose
\begin{equation}\label{eq:planewave}
\mathbf{E}^i(\mathbf{x})=\frac{\rmi}{k}\nabla\wedge\left(\nabla\wedge \mathbf{p}\rme^{\rmi k\mathbf{x}\cdot \mathbf{d}}\right),\quad \mathbf{H}^i(\mathbf{x})=\nabla\wedge \mathbf{p}\rme^{\rmi k\mathbf{x}\cdot {\mathbf{d}}},\quad \mathbf{x}\in\mathbb{R}^3.
\end{equation}
In this case $(\mathbf{E}^i, \mathbf{H}^i)= 
(\mathbf{E}^i, \mathbf{H}^i)(\mathbf{p},\mathbf{d})$ is known as the \emph{normalised electromagnetic plane wave} with the \emph{polarisation vector} $\mathbf{p}\in\mathbb{R}^3$, $\mathbf{p}\neq 0$, and the \emph{incident direction} ${\mathbf{d}}\in\mathbb{S}^2$.
Given a perfectly electric conducting scatterer $\Sigma$, and $\epsilon$ and $\mu$ satisfying the above hypotheses, 
we look for a solution $(\mathbf{E},\mathbf{H})$ to the following exterior boundary value problem
\begin{equation}\label{directscattering}
\left\{\begin{array}{ll}
\nabla\wedge \mathbf{E}-\rmi k\mu \mathbf{H}=0
&\text{in }G=\mathbb{R}^3\backslash \Sigma\\
\nabla\wedge \mathbf{H}+\rmi k \epsilon \mathbf{E}=0
&\text{in }G\\
(\mathbf{E},\mathbf{H})=(\mathbf{E}^i,\mathbf{H}^i)+(\mathbf{E}^s,\mathbf{H}^s)
&\text{in }G\\
\nu\wedge \mathbf{E}=0&\text{on }\partial G\\
\lim_{r\to+\infty}r\left(\frac{\mathbf{x}}{\|\mathbf{x}\|}\wedge \mathbf{H}^s(\mathbf{x})+
\mathbf{E}^s(\mathbf{x})\right)=0 & r=\|\mathbf{x}\|.
\end{array}
\right.
\end{equation}
Here $\mathbf{E}^s$ and $\mathbf{H}^s$ in \eqref{directscattering} are called the \emph{scattered electric and magnetic fields}, respectively, and by the last line they satisfy the Silver-M\"uller radiation condition. We call $\mathbf{E}$ and $\mathbf{H}$ the \emph{total electric and magnetic fields}, respectively.

The Maxwell system \eqref{directscattering} describes the following electromagnetic wave scattering. In the homogeneous and isotropic space, the EM incident wave $(\mathbf{E}^i, \mathbf{H}^i)$ would propagate indefinitely, since it is an entire solution to the Maxwell system \eqref{eq:Maxwell2bis}. The presence of the scatterer $\Sigma$, and of possible inhomogeneity around it,
perturbs the propagation of the incident wave. Such a perturbation is given by the scattered wave field $(\mathbf{E}^s,\mathbf{H}^s)$, which, outside a large enough ball,  is an outgoing, that is, radiating, solution to the Maxwell system \eqref{eq:Maxwell2bis}.

The existence and uniqueness of a solution to \eqref{directscattering} is classical in the case of a homogeneous and isotropic medium and of smooth scatterers, see for instance \cite{CK,Ned}. More effort is needed for the case of Lipschitz scatterers, and this case is also rather well understood, see for instance \cite{Mon}.

In order to deal with the most general scenario, that is, with minimal assumptions on the scatterer $\Sigma$ and the coefficients $\epsilon$ and $\mu$, we shall make use of a quite general result proved in \cite{PWW}. The uniqueness and existence result will be treated in the next subsection.

\subsection{Existence and uniqueness}\label{subsect:ex+un}
Given an exterior domain $G$, we call $G_R=G\cap B_R(0)$ for any $R>0$.
We assume that \eqref{boundedness}, \eqref{ellipticity2} and \eqref{infinity} hold for some constants $R_0>0$ and $0<\lambda_0<\lambda_1$.
We say that $(\mathbf{E},\mathbf{H})$ is a (weak) solution to the direct scattering problem \eqref{directscattering} for a given incident field $(\mathbf{E}^i,\mathbf{H}^i)$ if  
$(\mathbf{E},\mathbf{H})\in H(\mathrm{curl},G_R)^2$ for any $R>0$, the Maxwell system is satisfied in $G$ in the sense of distributions, $\mathbf{E}\in H_0(\mathrm{curl},G)$, and 
$(\mathbf{E}^s,\mathbf{H}^s)=(\mathbf{E},\mathbf{H})-(\mathbf{E}^i,\mathbf{H}^i)$ is an outgoing solution to the Maxwell system \eqref{eq:Maxwell2bis}
 in $\mathbb{R}^3\backslash \overline{B_{R_0}}$.
 
 Fixed an auxiliary function $\chi\in C^{\infty}_0(\mathbb{R}^3)$ such that
 $\chi\equiv 1$ in a neighbourhood of $\overline{B_{R_0}}$, we have that
 $(\mathbf{E},\mathbf{H})$ solves \eqref{directscattering} if and only if
 $(\mathbf{E}_1,\mathbf{H}_1)=(\mathbf{E},\mathbf{H})-(1-\chi)(\mathbf{E}^i,\mathbf{H}^i)$ solves, in an analogous weak sense,
\begin{equation}\label{directscatteringbis}
\left\{\begin{array}{ll}
\nabla\wedge \mathbf{E}_1-\rmi k\mu \mathbf{H}_1=\mathbf{F}
&\text{in }G\\
\nabla\wedge \mathbf{H}_1+\rmi k \epsilon \mathbf{E}_1=\mathbf{G}
&\text{in }G\\
\nu\wedge \mathbf{E}_1=0&\text{on }\partial G\\
\lim_{r\to+\infty}r\left(\frac{\mathbf{x}}{\|\mathbf{x}\|}\wedge \mathbf{H}_1(\mathbf{x})+
\mathbf{E}_1(\mathbf{x})\right)=0 & r=\|\mathbf{x}\|
\end{array}
\right.
\end{equation}
where
$$\mathbf{F}=\nabla\wedge ((\chi-1)\mathbf{E}^i)-\rmi k(\chi-1)\mu \mathbf{H}^i,\quad \mathbf{G}=\nabla\wedge ((\chi-1)\mathbf{H}^i)+\rmi k (\chi-1)\epsilon \mathbf{E}^i.$$
Notice that $\mathbf{F}$ and $\mathbf{G}$ belong to $L^2(D,\mathbb{C}^3)$ and have bounded support.

In order to connect to the existence and uniqueness result in \cite{PWW} we need the following.

\begin{lem}\label{connection}
Let us fix $R>0$.
Let $(\mathbf{E},\mathbf{H})\in H_{loc}(\mathrm{curl},\mathbb{R}^3\backslash\overline{B_R})^2$ solve \eqref{eq:Maxwell2bis} in $\mathbb{R}^3\backslash\overline{B_R}$.

Then the following two facts hold true.

\begin{enumerate}[A\textnormal{)}]

\item\label{partA} $(\mathbf{E},\mathbf{H})$ satisfies the Silver-M\"uller radiation condition if and only if
\begin{equation}\label{conda}
(\mathbf{E},\mathbf{H})\in \bigcap_{t<-1/2}H(\mathrm{curl},\mathbb{R}^3\backslash\overline{B_{R+1}},t)^2
\end{equation}
and
\begin{equation}\label{condb}
\frac{\mathbf{x}}{\|\mathbf{x}\|}\wedge \mathbf{H}(\mathbf{x})+
\mathbf{E}(\mathbf{x}),\ \frac{\mathbf{x}}{\|\mathbf{x}\|}\wedge \mathbf{E}(\mathbf{x})-
\mathbf{H}(\mathbf{x})\in \bigcup_{t>-1/2}H(\mathrm{curl},\mathbb{R}^3\backslash\overline{B_{R+1}},t).
\end{equation}

\item\label{partB}
 If $(\mathbf{E},\mathbf{H})$ satisfies the Silver-M\"uller radiation condition and
$$(\mathbf{E},\mathbf{H})\in \bigcap_{t\in\mathbb{R}}H(\mathrm{curl},\mathbb{R}^3\backslash\overline{B_{R+1}},t)^2,$$
then $\mathbf{E}=\mathbf{H}=0$ in $\mathbb{R}^3\backslash\overline{B_R}$.

\end{enumerate}

\end{lem}

\begin{proof}
We begin with part~\textit{\ref{partA}}) and we assume that $(\mathbf{E},\mathbf{H})$ satisfies the Silver-M\"uller radiation condition. Let $u$ be any Cartesian component of $\mathbf{E}$ or $\mathbf{H}$. Then, by the Sommerfeld radiation condition, we have
$$u(\mathbf{x})=\frac{\rme^{\rmi k\|\mathbf{x}\|}}{\|\mathbf{x}\|} {\mathbf{u}_\infty}(\hat{\mathbf{x}})+\frac{a(r\hat{\mathbf{x}})}{\|\mathbf{x}\|^2},\quad r=\|\mathbf{x}\|\geq R+1.$$
Here, for some constant $C_0$,
$$|a(r\hat{\mathbf{x}})|\leq C_0\quad \text{for any }r\geq R+1,\ \hat{\mathbf{x}}\in\mathbb{S}^2.$$
Since $r\leq\rho\leq C(R)r$ for any $r\geq R>0$, with $C(R)$ depending on $R$ only, we deduce that, for any $t\in\mathbb{R}$,
\begin{multline*}
\int_{\mathbb{R}^3\backslash \overline{B_{R+1}}}\rho^{2t}\|u\|^2\\\leq C_1\int_{R+1}^{+\infty}r^{2t}\left(\int_{\mathbb{S}^2}\|{\mathbf{u}_\infty}(\hat{\mathbf{x}})+a(r\hat{\mathbf{x}})/r\|^2\rmd\sigma(\hat{\mathbf{x}})\right)\rmd r\leq
C_2\int_{R+1}^{+\infty}r^{2t}\rmd r,
\end{multline*}
with $C_1$ and $C_2$ depending on $R$ and $t$. Since the right hand side
is finite for any $t<-1/2$, we conclude that \eqref{conda} holds.

We also notice that, whenever $\mathbf{u}_\infty\not\equiv 0$, then for any $t\in\mathbb{R}$ there exist $\tilde{R}\geq R+1$ and $\tilde{C}_0>0$ such that
$$\int_{\mathbb{R}^3\backslash \overline{B_{R+1}}}\rho^{2t}\|u\|^2\geq \tilde{C}_0
\int_{\tilde{R}}^{+\infty} r^{2t}\rmd r.$$
The right hand side is infinite for any $t\geq -1/2.$
Therefore, if $\|u\|_{0,t}$ on $\mathbb{R}^3\backslash \overline{B_{R+1}}$ is finite for every $t\in\mathbb{R}$, then we obtain that $\mathbf{u}_\infty\equiv 0$, and, in turn by the Rellich lemma, that $u\equiv 0$ in $\mathbb{R}^3\backslash \overline{B_R}$. Therefore, part~\textit{\ref{partB}}) is also proved.

Now we prove that \eqref{condb} holds. This follows immediately by the arguments developed in \cite[Chapter~6]{CK}, since the Silver-M\"uller radiation condition implies that, for any $\|\mathbf{x}\|=r\geq R+1$,
$$\left\|\frac{\mathbf{x}}{\|\mathbf{x}\|}\wedge \mathbf{H}(\mathbf{x})+
\mathbf{E}(\mathbf{x})\right\|,\ \left\|\frac{\mathbf{x}}{\|\mathbf{x}\|}\wedge \mathbf{E}(\mathbf{x})-
\mathbf{H}(\mathbf{x})\right\|\leq C/r^2 $$
for some constant $C$.
We readily conclude that \eqref{condb} holds.

We conclude the proof by showing that \eqref{conda} and \eqref{condb} imply the Silver-M\"uller radiation condition.
We call
$$a(r)=\int_{\partial B_r}\left\|\frac{\mathbf{x}}{\|\mathbf{x}\|}\wedge \mathbf{H}(\mathbf{x})+
\mathbf{E}(\mathbf{x})\right\|^2\rmd\sigma(\mathbf{x}).$$
By \eqref{condb}, we deduce that
\begin{equation}\label{viol}
\int_{R+1}^{+\infty}r^{2t}a(r)\rmd r<+\infty\quad\text{for some }t>-1/2.
\end{equation}
If $\liminf_{r\to+\infty}a(r)=\overline{a}>0$, then \eqref{viol} is violated. Therefore, we can find $\{r_n\}_{n\in\mathbb{N}}$ such that $R+1\leq r_n<r_{n+1}$ for any $n\in\mathbb{N}$, and
satisfying
$\lim_{n\to\infty} r_n=+\infty$, and
$$a(r_n)=\int_{\partial B_{r_n}}\left\|\frac{\mathbf{x}}{\|\mathbf{x}\|}\wedge \mathbf{H}(\mathbf{x})+
\mathbf{E}(\mathbf{x})\right\|^2\rmd\sigma(\mathbf{x})\to 0\quad\text{as }n\to\infty.$$

Following the proof of Theorem~6.6 in \cite{CK}, we can  deduce that the Stratton-Chu formulas hold. In turn, they imply the Silver-M\"uller radiation condition, thus the proof is concluded.\end{proof}

Concerning the assumptions on $G$ (or equivalently $\Sigma$) we need the following definition, again from \cite{PWW}.

\begin{defn}\label{regularityscatterer} Let $D$ be a bounded domain. We say that $D$ satisfies the \emph{Rellich compactness property} (in short RCP) if the natural immersion of $H^1(D)$ into $L^2(D)$ is compact.

We say that $D$ satisfies the \emph{Maxwell compactness property} (in short MCP) if the natural immersions of $H_0(\mathrm{curl},D)\cap H(\mathrm{div},D)$ and of $H(\mathrm{curl},D)\cap H_0(\mathrm{div},D)$ into $L^2(D,\mathbb{C}^3)$ are compact.
\end{defn}

Sufficient conditions that guarantee the RCP may be found in many standard reference books on Sobolev spaces, for instance in \cite{Ada,Maz}.
A detailed description of sufficient conditions for MCP to hold may be found in \cite[Theorem~3.6]{PWW}. In the next subsection we provide 
another sufficient condition, that is useful for our purposes, see Proposition~\ref{suffMCP}.

Now we state the following crucial theorem that is just a rephrasing of the main result of \cite{PWW}.

\begin{thm}[{Theorem~2.10 in \cite{PWW}}]\label{mainex+unthm}
Fix positive constants $k$, $R_0$, and $\lambda_0<\lambda_1$.

Let $G$ be an exterior domain satisfying \eqref{boundedness} and such that, for some $R>R_0$, $G_R$ satisfies the RCP and the MCP.

Let $\epsilon$, $\mu\in L^{\infty}(G,M^{3\times 3}_{sym}(\mathbb{R}))$ satisfy \eqref{ellipticity2} and
\eqref{infinity}.
Let $\mathbf{F}$, $\mathbf{G}\in L^2(G,\mathbb{C}^3)$ with bounded support.

Then \eqref{directscatteringbis} admits a unique solution $(\mathbf{E}_1,\mathbf{H}_1)$ if and only if the corresponding problem with $\mathbf{F}=\mathbf{G}=0$ in $G$ admits only the trivial solution $\mathbf{E}_1=\mathbf{H}_1=0$ in $G$.

Moreover, if $(\mathbf{E}_1,\mathbf{H}_1)$ solves \eqref{directscatteringbis} with $\mathbf{F}=\mathbf{G}=0$ in $G$, then
$(\mathbf{E}_1,\mathbf{H}_1)\in \bigcap_{t\in\mathbb{R}}H(\mathrm{curl},\mathbb{R}^3\backslash\overline{B_{R+1}},t)^2$.
\end{thm}

Under the assumptions of Theorem~\ref{mainex+unthm}, we have existence and uniqueness of a solution to \eqref{directscatteringbis} and, in turn, of a solution to \eqref{directscattering}, if the Maxwell system \eqref{eq:Maxwell2} admits the \emph{unique continuation property} (in short UCP). In fact, by part \textit{B}) of Lemma~\ref{connection} and the last part of Theorem~\ref{mainex+unthm}, we obtain that 
if $(\mathbf{E}_1,\mathbf{H}_1)$ solves \eqref{directscatteringbis} with $\mathbf{F}=\mathbf{G}=0$ in $G$, then $\mathbf{E}_1=\mathbf{H}_1=0$ outside a sufficiently large ball.

We summarise the existence and uniqueness result in the following theorem.

\begin{thm}\label{mainthmex+uniq}
Fix positive constants $k$, $R_0$, and $\lambda_0<\lambda_1$.

Let $G$ be an exterior domain satisfying \eqref{boundedness} and such that, for some $R>R_0$, $G_R$ satisfies the RCP and the MCP.

Let $\epsilon$, $\mu\in L^{\infty}(G,M^{3\times 3}_{sym}(\mathbb{R}))$ satisfy \eqref{ellipticity2} and
\eqref{infinity}. Moreover, let us assume that \eqref{eq:Maxwell2} satisfies the UCP in $G$.

Then, for any $(\mathbf{E}^i,\mathbf{H}^i)$ entire solution to \eqref{eq:Maxwell2bis},
the problem \eqref{directscattering} admits a unique solution $(\mathbf{E},\mathbf{H})$.
\end{thm}

\subsection{Sufficient conditions for RCP, MCP and UCP}\label{subsection1:3}
A useful sufficient condition for RCP and MCP to hold is contained in the following proposition whose proof is postponed to the Appendix, since it requires some results from Subsection~\ref{prop-subsec}.

\begin{prop}\label{suffMCP}
Let $D$ be a bounded domain. Let us assume that for any $\mathbf{x}\in\partial D$ there exists an open neighbourhood $U_{\mathbf{x}}$ such that $U_{\mathbf{x}}\cap D$ has a finite number of connected components. Moreover, each connected component of $U_{\mathbf{x}}\cap D$ such that $\mathbf{x}$ belongs to its boundary
may be mapped onto a Lipschitz domain by a
bi-$W^{1,\infty}$ mapping.

Then $D$ satisfies both the RCP and MCP.
\end{prop}

In the literature there are several sufficient conditions on the coefficients $\epsilon$ and $\mu$ in \eqref{eq:Maxwell2} for UCP to hold. We notice that if $\epsilon=\mu=I_3$ everywhere, that is, we consider \eqref{eq:Maxwell2bis}, then UCP trivially holds.

The first result on unique continuation that we wish to recall follows immediately from \cite[Theorem~1.1]{Ngu-Wan}.

\begin{thm}\label{LipUCP}
Let $D$ be a domain. We assume that $\epsilon$ and $\mu$ belong to $L^{\infty}(D,M^{3\times 3}_{sym}(\mathbb{R}))$ and satisfy \eqref{ellipticity2} in $D$ for some constants $0<\lambda_0<\lambda_1$.

If $\epsilon$ and $\mu$ are locally Lipschitz in $D$,
then \eqref{eq:Maxwell2} satisfies the UCP in $D$. 
\end{thm}

Let us notice that Lipschitz continuity is essentially optimal, as shown by an example in \cite{Dem}. Inspired by the results and constructions in \cite{Bal-et-al} for the piecewise Lipschitz case, we state the following.

\begin{prop}\label{piecewiseLipschitzprop}
Let $D$ be an open set. Assume that $\epsilon$ and $\mu$ belong to $L^{\infty}(D,M^{3\times 3}_{sym}(\mathbb{R}))$ and satisfy \eqref{ellipticity2} in $D$ for some constants $0<\lambda_0<\lambda_1$.

Assume that:
\begin{enumerate}[i\textnormal{)}]
\item there exists a family $\{D_i\}$ of domains, that are contained in $D$ and which are pairwise disjoint, such that
$$D\subset \bigcup_i\overline{D_i}.$$
\item We have $|\sigma|=0$ where
$$\sigma=D\cap\left(\bigcup_i\partial D_i\right).$$
\item\label{connectedC} We say that $\mathbf{x}\in\sigma$ \emph{separates exactly two partitions} if there exist $\delta>0$ and two different indexes $i$ and $j$ such that
$$\left|B_{\delta}(\mathbf{x})\backslash\left(D_i\cup D_j\right)\right|=0$$
and $B_{\delta}(\mathbf{x})\cap D_i$ and $B_{\delta}(\mathbf{x})\cap D_j$ are not empty.
We call
$$C=\{\mathbf{x}\in \sigma:\ \mathbf{x}\text{ does not separate exactly two partitions}\}.$$
We assume that $\tilde{D}=D\backslash C$ is connected.
\item $(\epsilon,\mu)=(\epsilon_i,\mu_i)$ in $D_i$ where $\{\epsilon_i\}$ and $\{\mu_i\}$
are families of locally Lipschitz $M^{3\times 3}_{sym}(\mathbb{R})$-valued functions in $D$.
\end{enumerate}

Then \eqref{eq:Maxwell2} satisfies the UCP in $D$.  

\end{prop}

\begin{proof} We begin with a few preliminary remarks. First of all,
the family $\{D_i\}$ is countable. Second,
$\sigma=D\backslash\left(\bigcup_i D_i\right)$ is closed in $D$. Moreover,
$C$ is closed in $\sigma$, thus in $D$ as well, therefore $\tilde{D}$ is open and $D$ is connected.

Let us assume that there exists an open nonempty set $A\subset D$ such that $(\mathbf{E},\mathbf{H})=(0,0)$ everywhere in $A$. Without loss of generality, we can find $\hat{\mathbf{x}}\in D$, $r>0$, and an index $\hat{i}$ such that
$B_r(\hat{\mathbf{x}})\subset A\cap D_{\hat{i}}$. By Theorem~\ref{LipUCP},
we deduce that  
$(\mathbf{E},\mathbf{H})=(0,0)$ everywhere in $D_{\hat{i}}$.

Let us assume, by contradiction, that there exists an index $\hat{j}$ such that 
$(\mathbf{E},\mathbf{H})$ is not identically equal to 
$(0,0)$ in $D_{\hat{j}}$. Let us pick any $\hat{\mathbf{y}}\in D_{\hat{j}}$ and let $\gamma:[0,1]\to \tilde{D}$ be a smooth curve such that $\gamma(0)=\hat{\mathbf{x}}$ and $\gamma(1)=\hat{\mathbf{y}}$.

Let us define 
$$\hat{t}=\inf\{t\in[0,1]:\ \gamma(t)\in D_i\text{ for some }i\text{ s.t. }
(\mathbf{E},\mathbf{H})\not\equiv(0,0)
\text{ in }D_i\}.$$
We notice that $0<\hat{t}<1$. Furthermore, by the definition of $\tilde{D}$ and of $C$, we obtain that $\gamma(\hat{t})\in \partial D_i\cap\partial D_j$, for two different indexes $i$ and $j$, and separates exactly the two partitions $D_i$ and $D_j$. Moreover, for some positive $\delta_i$ and $\delta_j$, and up to swapping $i$ with $j$, we have $\gamma(\hat{t}+\delta_i)\in D_i$ and $(\mathbf{E},\mathbf{H})\not\equiv(0,0)$ in
$D_i$, and $\gamma(\hat{t}-\delta_j)\in D_j$. By the definition of $\hat{t}$, we deduce that $(\mathbf{E},\mathbf{H})\equiv(0,0)$ in
$D_j$.

Then the proof can be concluded by the arguments developed in \cite{Bal-et-al} that we briefly recall here. For simplicity, let us assume that $\gamma(\hat{t})=0$. For some $\delta>0$, we have $\left|B_{\delta}\cap\left(D_i\cup D_j\right)\right|=0$. We have that $(\epsilon_i,\mu_i)\in W^{1,\infty}(B_{\delta})^2$ and $(\mathbf{E},\mathbf{H})\equiv 0$ in $D_j$. Since $(\mathbf{E},\mathbf{H})$ solves \eqref{eq:Maxwell2} in $B_{\delta}$, it is immediate to notice that $(\mathbf{E},\mathbf{H})$ solves in $B_{\delta}$
$$\nabla\wedge \mathbf{E}-\rmi k\mu_i \mathbf{H}=0,\quad 
\nabla\wedge \mathbf{H}+\rmi k \epsilon_i \mathbf{E}=0$$
as well. Since $(\mathbf{E},\mathbf{H})\equiv 0$ in $B_{\delta}\cap D_j$, which is not empty, we conclude, again by Theorem~\ref{LipUCP}, that $(\mathbf{E},\mathbf{H})\equiv 0$ in $B_{\delta}$, thus in
$B_{\delta}\cap D_i\neq \emptyset$. Using again Theorem~\ref{LipUCP}, we conclude that $(\mathbf{E},\mathbf{H})\equiv 0$ in $D_i$ and we obtain a contradiction, which completes the proof. 
\end{proof}

We conclude this excursus on the UCP with the following lemma that provides a sufficient condition for assumption \textit{\ref{connectedC}}) in Proposition~\ref{piecewiseLipschitzprop} to hold.

\begin{lem}\label{disconnectionlemma}
Let $D$ be a connected open set contained in $\mathbb{R}^N$, $N\geq 2$. Let $C\subset D$ be closed in $D$.
If $D\backslash C$ is not connected, then $C$ has Hausdorff dimension greater than or equal to $N-1$.
\end{lem}

\begin{rem}\label{piecewiseLipschitzrem}
By Lemma~\textnormal{\ref{disconnectionlemma}}, we can replace
assumption \textit{\ref{connectedC}}\textnormal{)} in Proposition~\textnormal{\ref{piecewiseLipschitzprop}} with the following. We assume that $D$ is connected and that for some $s<2$ we have $\mathcal{H}^s(C)<+\infty$, where
$\mathcal{H}^s$ is the $s$-dimensional Hausdorff measure.
\end{rem}

\begin{proof}[Proof of Lemma~\textnormal{\ref{disconnectionlemma}}]
We begin with the following remark. Let $A\subset \mathbb{R}^N$, $N\geq 2$, be a bounded open set. Then the Hausdorff dimension of $\partial A$ is at least $N-1$.
This is a classical result that may be proved as follows. We call $\pi:\mathbb{R}^N\to\mathbb{R}^{N-1}$ the projection onto the first $N-1$ coordinates. We notice that $\pi(A)$ is a bounded open set in $\mathbb{R}^{N-1}$ and that $\pi(\partial A)=\overline{\pi(A)}$. Hence, $0<\mathcal{H}^{N-1}(\pi(A))\leq \mathcal{H}^{N-1}(\overline{\pi(A)})
=\mathcal{H}^{N-1}(\pi(\partial A))\leq C_1\mathcal{H}^{N-1}(\partial A)$ for some positive constant $C_1$, thus the property is proved.

Let us assume, by contradiction, that the Hausdorff dimension of $C$ is less than $N-1$ and that $D\backslash C$ is not connected; that is 
$D\backslash C=D_1\cup D_2$ where $D_1$, $D_2$ are open, nonempty sets which are disjoint. There exist $\mathbf{x}\in C$ and $\delta>0$ such that $\overline{B_{\delta}(\mathbf{x})}\subset D$ and $B_{\delta}(\mathbf{x})\cap D_i\neq\emptyset$ for $i=1,2$.
We also observe that there exist $\mathbf{x}_i\in \partial B_{\delta}(\mathbf{x})\cap D_i$ 
for $i=1,2$. In fact, otherwise, $\partial B_{\delta}(\mathbf{x})\cap D_i\subset C$, thus, by the property proved above, $C$ has at least dimension $N-1$ and we have a contradiction.

By a bi-Lipschitz map $T$ we transform $\overline{B_{\delta}(\mathbf{x})}$ onto
$\mathbb{S}^+_{\delta}$ where
$$\mathbb{S}^+_{\delta}=\{\mathbf{x}=(x_1,\ldots,x_N,x_{N+1})\in\mathbb{R}^{N+1}:\
\|\mathbf{x}\|=\delta\text{ and }x_{N+1}\geq 0\}.$$
Simply by a reflection in the plane $\Pi=\{\mathbf{x}=(x_1,\ldots,x_N,x_{N+1})\in\mathbb{R}^{N+1}:\
x_{N+1}= 0\}$, we may find two open subset of $\mathbb{S}_{\delta}=
\{\mathbf{x}\in\mathbb{R}^{N+1}:\
\|\mathbf{x}\|=\delta\}$, $\Omega_1$ and $\Omega_2$ and a closed set $\tilde{C}$ that are symmetric with respect to $\Pi$ and such that
$\Omega_i\cap \mathbb{S}^+_{\delta}=T(D_i\cap \overline{B_{\delta}(\mathbf{x})})$ for $i=2$ and $\tilde{C}\cap \mathbb{S}^+_{\delta}=T(C\cap \overline{B_{\delta}(\mathbf{x})})$. Clearly $\Omega_1$ and $\Omega_2$ are nonempty and disjoint. Fix $\mathbf{y}_1\in\Omega_1$ and consider a stereographic projection $T_1$ with pole $\mathbf{y}_1$. We have that $T_1(\Omega_2)$ is a bounded open set, contained in $\mathbb{R}^{N-1}$.
Since $\partial T_1(\Omega_2)\subset T_1(\tilde{C})$ and $T_1(\tilde{C})$
is a bi-Lipschitz image of $\tilde{C}$, we can easily conclude that
the Hausdorff dimension of $C$ is at least $N-1$.

The proof is complete. 
\end{proof}

\section{Properties of solutions to the Maxwell system and classes of admissible scatterers}\label{properties-section}

Before passing to the stability results, in this section we collect a few properties of solutions to the Maxwell system that will be needed in the sequel.
Then we define and discuss suitable classes of admissible scatterer for our study.

\subsection{Properties of solutions to the Maxwell system}\label{prop-subsec}

We begin with the following higher integrability and compact immersion results that are proved in \cite{Dru}.

\begin{prop}\label{Druprop}
Let $D\subset\mathbb{R}^3$ be a bounded Lipschitz domain.
Let $a\in L^{\infty}(D,M^{3\times 3}_{sym}(\mathbb{R}))$ satisfy \eqref{ellipticity} for some constants $0<a_0<a_1$.

We call
$$W^{2,r}_{a,\nu}(D)=\{u\in H(\mathrm{curl},D)\cap H(\mathrm{div}_a,D):\ \gamma_{\nu}(au)\in L^r(\partial D)\}$$
and 
$$W^{2,r}_{a,\tau}(D)=\{u\in H(\mathrm{curl},D)\cap H(\mathrm{div}_a,D):\ \gamma_{\tau}(u)\in L^r(\partial D,\mathbb{C}^3)\}.$$

Then there exists $q_1$, $2<q_1<6$ depending on $D$, $a_0$, and $a_1$ only, such that for any $r>4/3$ we have that
$W^{2,r}_{a,\nu}(D)$ and $W^{2,r}_{a,\tau}(D)$ are immersed continuously into $L^s(D,\mathbb{C}^3)$, with $s=\min\{3r/2,q_1\}>2$,
and are immersed compactly into $L^2(D,\mathbb{C}^3)$.
More precisely, there exists a constant $C$, depending on $D$, $a_0$, $a_1$, and $r$ only, such that
\begin{multline*}\|u\|_{L^s(D,\mathbb{C}^3)}\leq C\big[\|u\|_{L^2(D,\mathbb{C}^3)}+
\|\nabla\wedge u\|_{L^2(D,\mathbb{C}^3)}\\
+\|\nabla\cdot (au)\|_{L^2(D)}
+\|\gamma_{\nu}(au)\|_{L^r(\partial D)}
\big] \quad\text{for any }u\in W^{2,r}_{a,\nu}(D)
\end{multline*}
and
\begin{multline*}\|u\|_{L^s(D,\mathbb{C}^3)}\leq C\big[\|u\|_{L^2(D,\mathbb{C}^3)}+
\|\nabla\wedge u\|_{L^2(D,\mathbb{C}^3)}\\
+\|\nabla\cdot (au)\|_{L^2(D)}
+\|\gamma_{\tau}(u)\|_{L^r(\partial D,\mathbb{C}^3)}
\big] \quad\text{for any }u\in W^{2,r}_{a,\tau}(D).
\end{multline*}
\end{prop}

\begin{proof} The first part is a consequence of Corollary~1 in \cite{Dru}, whereas the second follows from Corollary~2 in \cite{Dru}.\end{proof}

We remark that, when $a\equiv I_3$ and $r=2$, these results would follow from the immersions into $H^{1/2}(D)$ proved in \cite{Cos}.

We next investigate how the Maxwell equations are transformed by change of variables.
The change of variables we consider is of the following type.

Let $D$ and $D'$ be two connected open sets and let $T:D\to D'$ be a bi-$W^{1,\infty}$ mapping, with constant $L$. We call $S=T^{-1}$ and
$J=JT$, the Jacobian matrix of $T$. Also, $J^T$ denotes the transpose of $J$ and $J^{-T}$ denotes the transpose of $J^{-1}$.
Recalling Remark~\ref{detsign}, we set $\mathrm{or}(T)=1$ if $\det J(\mathbf{x})>0$ for almost every $\mathbf{x}\in D$ and $\mathrm{or}(T)=-1$ if $\det J(\mathbf{x})<0$ for almost every $\mathbf{x}\in D$.

We begin by investigating how the suitable spaces we are dealing with are transformed.
It is classical that
for any $\varphi \in H^1(D')$ we call $\psi=\tilde{T}(\varphi)=\varphi\circ T$. We have that $\psi\in H^1(D)$ and, for almost every $\mathbf{y}\in D'$,
$$\nabla \varphi(\mathbf{y})=J^{-T}(T^{-1}(\mathbf{y}))\nabla \psi(T^{-1}(\mathbf{y})).$$
Then we have that $\tilde{T}:H^1(D')\to H^1(D)$ is a linear homeomorphism with inverse $\tilde{S}$.

The change of variables that is suited for $H(\mathrm{curl})$ spaces is the following.
For any $u\in L^2(D',\mathbb{C}^3)$ we define $v=\hat{T}(u)$ as follows
$$v(\mathbf{x})=\hat{T}(u)(\mathbf{x})=J^T(\mathbf{x})u(T(\mathbf{x}))\quad\text{ for a.e. }\mathbf{x}\in D$$
or, equivalently,
$$u(\mathbf{y})=\hat{S}(v)(\mathbf{y})=J^{-T}(T^{-1}(\mathbf{y}))v(T^{-1}(\mathbf{y}))\quad\text{ for a.e. }\mathbf{y}\in D'.$$

For $H(\mathrm{div})$ spaces the change of variables is given by the following.
Let $a'\in L^{\infty}(D',M^{3\times 3}_{sym}(\mathbb{R}))$ satisfy \eqref{ellipticity} in $D'$ for some constants $0<a'_0<a'_1$. For any $u\in H(\mathrm{div}_{a'},D')$, then $v=\hat{T}(u)\in H(\mathrm{div}_a,D)$ with the following formulas
\begin{multline}\label{divchange}
a(\mathbf{x})=T_{\ast}(a')(\mathbf{x})=\left(\frac{J^{-1}a'(T)J^{-T}}{|\det J^{-1}|}\right)(\mathbf{x})\quad\text{and}\\
\nabla\cdot(av)(\mathbf{x})=\left(\frac{(\nabla\cdot (a'u))(T)}{|\det J^{-1}|}\right)(\mathbf{x})
\quad\text{ for a.e. }\mathbf{x}\in D.
\end{multline}
We notice that $a\in L^{\infty}(D,M^{3\times 3}_{sym}(\mathbb{R}))$ and satisfies \eqref{ellipticity} in $D$ for some constants $0<a_0<a_1$ depending on
$a'_0$, $a'_1$, and $L$ only.

The following proposition can be proved as in \cite{Mon} with a little more care since here we are using a bi-$W^{1,\infty}$ mapping instead of a $C^1$-diffeomorphism.

\begin{prop}\label{changeofvariables}
Under the previous notations and assumptions, if $v\in H(\mathrm{curl},D)$ then
$u=\hat{S}(v)\in H(\mathrm{curl},D')$ and the following formula holds
\begin{equation}
(\nabla\wedge u)(\mathbf{y})=\left(\frac{J}{\det J}\nabla\wedge v\right)(T^{-1}(\mathbf{y}))\quad\text{ for a.e. }\mathbf{y}\in D'.
\end{equation}
\end{prop}

By this formula, simple computations lead to the following.

\begin{cor}\label{changecor}
We have that $\hat{T}:H(\mathrm{curl},D')\to H(\mathrm{curl},D)$ is a linear homeomorphism with inverse $\hat{S}$. Its norm is bounded by a constant depending on $L$ only. Moreover
$\hat{T}:H_0(\mathrm{curl},D')\to H_0(\mathrm{curl},D)$
is bijective.
Finally, if  if $u$, $\varphi\in H(\mathrm{curl},D')$ and $v=\hat{T}(u)$, $\psi=\hat{T}(\varphi)\in H(\mathrm{curl},D)$, then
\begin{equation}\label{H0curlchange}
\langle\nabla\wedge u,\varphi\rangle_{D'}-\langle u,\nabla\wedge\varphi\rangle_{D'}=
\mathrm{or}(T)\left[\langle\nabla\wedge v,\psi\rangle_{D}-\langle v,\nabla\wedge\psi\rangle_{D}\right].
\end{equation}

We also have that  $\hat{T}:H(\mathrm{div}_{a'},D')\to H(\mathrm{div}_a,D)$ is a linear homeomorphism with inverse $\hat{S}$, where the tensor $a$ is given in \eqref{divchange}. Its norm is bounded by a constant depending on $L$ only. Moreover, $\hat{T}:H_0(\mathrm{div}_{a'},D')\to H_0(\mathrm{div}_a,D)$ is bijective.
Finally, if $u\in H(\mathrm{div}_{a'},D')$ and
$\varphi\in H^1(D')$, and $v=\hat{T}(u)\in H(\mathrm{div}_a,D)$ and $\psi=\tilde{T}(\varphi)\in H^1(D)$, then
\begin{equation}\label{H0divchange}
\langle\nabla\cdot (a'u),\varphi\rangle_{D'}+\langle a'u,\nabla\varphi\rangle_{D'}=\langle\nabla\cdot (av),\psi\rangle_{D}+\langle av,\nabla\psi\rangle_{D}.
\end{equation}
\end{cor}

Let us now investigate the transformation of the Maxwell equations under these changes of variables.
Let $\epsilon'$, $\mu'\in L^{\infty}(D',M^{3\times 3}_{sym}(\mathbb{R}))$ satisfy \eqref{ellipticity2} in $D'$ for some constants $0<\lambda'_0<\lambda'_1$.
Let $(\mathbf{E}',\mathbf{H}')\in H(\mathrm{curl},D')^2$ solve the Maxwell equations for some $k>0$
$$\nabla\wedge \mathbf{E}'-\rmi k\mu'\mathbf{H}'=0\quad\text{and}\quad
\nabla\wedge \mathbf{H}'+\rmi k\epsilon'\mathbf{E}'=0\quad\text{in }D'.$$
Then straightforwad computations show that $\mathbf{E}=\mathrm{or}(T)\hat{T}(\mathbf{E}')$ and
$\mathbf{H}=
\hat{T}(\mathbf{H}')$ solve
$$\nabla\wedge \mathbf{E}-\rmi k\mu\mathbf{H}=0\quad\text{and}\quad
\nabla\wedge \mathbf{H}+\rmi k\epsilon\mathbf{E}=0\quad\text{in }D$$
where 
\begin{multline*}
\epsilon(\mathbf{x})=T_{\ast}(\epsilon')(\mathbf{x})=\left(\frac{J^{-1}\epsilon'(T)J^{-T}}{|\det J^{-1}|}\right)(\mathbf{x})\quad\text{and}\\
\mu(\mathbf{x})=T_{\ast}(\mu')(\mathbf{x})=\left(\frac{J^{-1}\mu'(T)J^{-T}}{|\det J^{-1}|}\right)(\mathbf{x})
\quad\text{ for a.e. }\mathbf{x}\in D.
\end{multline*}
Again we have that $\epsilon$, $\mu\in L^{\infty}(D,M^{3\times 3}_{sym}(\mathbb{R}))$ and satisfy \eqref{ellipticity2} in $D$ for some constants $0<\lambda_0<\lambda_1$ depending on
$\lambda'_0$, $\lambda'_1$, and $L$ only.

With these results at hand, we also investigate the following reflection principles.
Let $\Pi$ be a plane in $\mathbb{R}^3$ and let $T=T_{\Pi}$ be the reflection in $\Pi$. We notice that $T=T^{-1}$ and $JT$ is identically equal to an orthogonal matrix $J$ with $J=J^T=J^{-1}$ and
$\det J=-1$, therefore $\mathrm{or}(T)=-1$. We call $\nu$ one of the two unit vectors orthogonal to $\Pi$.
Let $D'$ be an open connected set such that $D'\subset \mathbb{R}^3_+=\{\mathbf{y}\in\mathbb{R}^3:\ \mathbf{y}\cdot \nu>0\}$. We suppose that there exists $\Gamma$ a nonempty open subset of $\Pi$ such that $\Gamma=\partial D'\backslash (\overline{\partial D'\cap \mathbb{R}^3_+})$. We call $D=T(D)$ and $\Omega=
D\cup D'\cup\Gamma$. We notice that $\Omega$ is a connected set such that
$\Omega\cap \Pi=\Gamma$.

\begin{prop}\label{reflprop}
Let $u\in H(\mathrm{curl},D')$. Then the function
$$w=\left\{\begin{array}{ll}
u&\text{in }D'\\
\hat{T}(u)&\text{in }D
\end{array}\right.$$
belongs to $H(\mathrm{curl},\Omega)$ 
and we have the following formula
$$\nabla\wedge w=\left\{\begin{array}{ll}
\nabla\wedge u&\text{in }D'\\
\nabla\wedge\hat{T}(u)&\text{in }D.
\end{array}\right.$$

If $\nu\wedge u=0$ in 
$H^{-1/2}(\Gamma_1,\mathbb{C}^3)$ for any open $\Gamma_1$ compactly contained in $\Gamma$, then we also have that
$$w_1=\left\{\begin{array}{ll}
u&\text{in }D'\\
-\hat{T}(u)&\text{in }D
\end{array}\right.$$
belongs to $H(\mathrm{curl},\Omega)$ 
and
$$\nabla\wedge w_1=\left\{\begin{array}{ll}
\nabla\wedge u&\text{in }D'\\
-\nabla\wedge\hat{T}(u)&\text{in }D.
\end{array}\right.$$

Given $a'\in L^{\infty}(D',M^{3\times 3}_{sym}(\mathbb{R}))$ satisfying \eqref{ellipticity} with constant $0<a'_0<a'_1$, we notice that if $u\in H(\mathrm{div}_{a'},D')$, then
the function $w_1$ defined above belongs to  $H(\mathrm{div}_a,\Omega)$,
where
$$a=\left\{\begin{array}{ll}
a'
&\text{in }D'\\
T_{\ast}(a')&\text{in }D,
\end{array}\right.$$
and we have the following formula
$$\nabla\cdot (aw_1)=\left\{\begin{array}{ll}
\nabla\cdot (a'u)&\text{in }D'\\
-\nabla\cdot(T_{\ast}(a')\hat{T}(u))&\text{in }D.
\end{array}\right.$$

If $\nu\cdot (a'u)=0$ in 
$H^{-1/2}(\Gamma_1)$ for any open $\Gamma_1$ compactly contained in $\Gamma$, then the function 
$w$ defined above belongs to  $H(\mathrm{div}_a,\Omega)$,
and we have the following formula
$$\nabla\cdot (aw)=\left\{\begin{array}{ll}
\nabla\cdot (a'u)&\text{in }D'\\
\nabla\cdot(T_{\ast}(a')\hat{T}(u))&\text{in }D.
\end{array}\right.$$

Let us finally notice that if $a'=I_3$, and $T$ is a reflection, then $a=T_{\ast}(a')=
T_{\ast}(I_3)=I_3$.
\end{prop}

\begin{proof}
Let $\varphi\in C^{\infty}_0(\Omega,\mathbb{C}^3)$. We can find a bounded Lipschitz domain $\Omega_1$ compactly contained in $\Omega$ and  containing the support of $\varphi$ satisfying the following properties.
We assume that $\Omega_1$ is symmetric with respect to $\Pi$ and $D'_1=\Omega_1\cap  \mathbb{R}^3_+$ is a bounded Lipschitz domain as well. We call $D_1=T(D_1)$ and $\Gamma_1=\Omega_1\cap\Pi$, which is compactly contained in $\Gamma$. 
Our aim is to prove that
\begin{multline*}\langle w,\nabla\wedge \varphi\rangle_{\Omega}=
\langle w,\nabla\wedge \varphi\rangle_{\Omega_1}=
\langle u,\nabla\wedge \varphi\rangle_{D'_1}
+\langle \hat{T}(u),\nabla\wedge \varphi\rangle_{D_1}
\\=
\langle \nabla\wedge u,\varphi\rangle_{D'_1}+
\langle \nabla\wedge \hat{T}(u),\varphi\rangle_{D_1}=
\langle \nabla\wedge u,\varphi\rangle_{D'}+
\langle \nabla\wedge \hat{T}(u),\varphi\rangle_{D}.
\end{multline*}
We know that
$$
\langle u,\nabla\wedge\varphi\rangle_{D'_1}-
\langle \nabla\wedge u,\varphi\rangle_{D'_1}=-\langle(-\nu)\wedge u,\varphi\rangle_{\Gamma_1}$$
and, analogously,
$$\langle \hat{T}(u),\nabla\wedge\varphi\rangle_{D_1}-
\langle \nabla\wedge \hat{T}(u),\varphi\rangle_{D_1}=-\langle \nu\wedge \hat{T}(u),\varphi\rangle_{\Gamma_1}.$$
By an approximation argument, it is not difficult to show that
$\nu\wedge u=\nu \wedge \hat{T}(u)$ in $H^{-1/2}(\Gamma_1,\mathbb{C}^3)$.
Hence the first result is proved. The others follow by analogous reasonings, 
just replacing $\nu\wedge u$ with $\nu\cdot (a'u)$ for the $H(\mathrm{div})$ spaces.\end{proof}

\bigskip

With the same notations as before, we have that if 
$(\mathbf{E}',\mathbf{H}')\in H(\mathrm{curl},D')^2$ solve the Maxwell equations for some $k>0$
$$\nabla\wedge \mathbf{E}'-\rmi k\mu'\mathbf{H}'=0\quad\text{and}\quad
\nabla\wedge \mathbf{H}'+\rmi k\epsilon'\mathbf{E}'=0\quad\text{in }D'$$
and $\nu\wedge \mathbf{E}=0$ in 
$H^{-1/2}(\Gamma_1,\mathbb{C}^3)$ for any open $\Gamma_1$ compactly contained in $\Gamma$, we have that the couple
\begin{equation}\label{reflectionrules}
(\mathbf{E},\mathbf{H})=\left\{\begin{array}{ll}
(\mathbf{E}',\mathbf{H}')
&\text{in }D'\\
(-\hat{T}(\mathbf{E}'),\hat{T}(\mathbf{H}'))&\text{in }D
\end{array}\right.
\end{equation}
belongs to $H(\mathrm{curl},\Omega)^2$ and solves
\begin{equation}\label{reflectioneq}
\nabla\wedge \mathbf{E}-\rmi k\mu\mathbf{H}=0\quad\text{and}\quad
\nabla\wedge \mathbf{H}+\rmi k\epsilon\mathbf{E}=0\quad\text{in }\Omega
\end{equation}
where
\begin{equation}\label{reflectioncoefficients}
(\epsilon,\mu)=\left\{\begin{array}{ll}
(\epsilon',\mu')
&\text{in }D'\\
(T_{\ast}(\epsilon'),T_{\ast}(\mu'))&\text{in }D.
\end{array}\right.
\end{equation}

We conclude our discussion on reflection properties by stating, without proofs, the following lemmas.

\begin{lem}\label{reflectionestimatelem}
Under the previous assumptions, let $(\mathbf{E},\mathbf{H})$
be as in \eqref{reflectionrules} and $(\epsilon,\mu)$ as in \eqref{reflectioncoefficients}.
Let $(\tilde{\mathbf{E}},\tilde{\mathbf{H}})$ solve \eqref{reflectioneq} in $\Omega$ and let us call
$(\tilde{\mathbf{E}}_1,\tilde{\mathbf{H}}_1)=
(-\hat{T}(\tilde{\mathbf{E}}),\hat{T}(\tilde{\mathbf{H}}))$ in 
$\Omega$. We have that $(\tilde{\mathbf{E}}_1,\tilde{\mathbf{H}}_1)$ solves \eqref{reflectioneq} in $\Omega$ and
\begin{multline}\label{reflectionestimate}
\|\tilde{\mathbf{E}}-\tilde{\mathbf{E}}_1\|_{L^2(\Omega,\mathbb{C}^3)}\leq 
\|\tilde{\mathbf{E}}-\mathbf{E}\|_{L^2(\Omega,\mathbb{C}^3)}
+\|\mathbf{E}-
\tilde{\mathbf{E}}_1\|_{L^2(\Omega,\mathbb{C}^3)}\\=
2\|\tilde{\mathbf{E}}-\mathbf{E}\|_{L^2(\Omega,\mathbb{C}^3)}
\end{multline}
and, analogously,
\begin{multline*}
\|\tilde{\mathbf{H}}-\tilde{\mathbf{H}}_1\|_{L^2(\Omega,\mathbb{C}^3)}\leq 
\|\tilde{\mathbf{H}}-\mathbf{H}\|_{L^2(\Omega,\mathbb{C}^3)}
+\|\mathbf{H}-
\tilde{\mathbf{H}}_1\|_{L^2(\Omega,\mathbb{C}^3)}\\=
2\|\tilde{\mathbf{H}}-\mathbf{H}\|_{L^2(\Omega,\mathbb{C}^3)}.
\end{multline*}
\end{lem}

\begin{lem}\label{reflectionlemma}
Under the assumptions and notations of Lemma~\textnormal{\ref{reflectionestimatelem}}, 
we have that, for any open $\Gamma_1$ compactly contained in $\Gamma$,
$$\nu\wedge (\tilde{\mathbf{E}}-\tilde{\mathbf{E}}_1)=2(\nu\wedge \tilde{\mathbf{E}})\quad\text{in }H^{-1/2}(\Gamma_1,\mathbb{C}^3).$$

Moreover, if \eqref{reflectioneq} satisfies the UCP in $\Omega$ and $\nu\wedge \tilde{\mathbf{E}}=0$ in 
$H^{-1/2}(\Gamma_1,\mathbb{C}^3)$ for any open $\Gamma_1$ compactly contained in $\Gamma$,
then we have that
$(\tilde{\mathbf{E}},\tilde{\mathbf{H}})=(\tilde{\mathbf{E}}_1,\tilde{\mathbf{H}}_1)$.
\end{lem}

\subsection{Classes of admissible scatterers}\label{classes-subsec}

We wish to define suitable classes of admissible scatterers. We begin with some definitions that are taken from \cite{LPRX}, where these classes are deeply discussed and are defined for any $N\geq 2$. We keep such a generality here as well.
Let us also notice that similar classes have been developed earlier in \cite{Gia} and in
\cite{MR}.

We fix a bounded open set $\Omega\subset\mathbb{R}^N$, $N\geq 2$.
Let $K$ be a compact subset of $\mathbb{R}^N$. We say that $K$ is a \emph{mildly} \emph{Lipschitz surface}, with or without boundary, with positive constants $r$ and $L$ if the following holds.

For any $\mathbf{x}\in K$ there exists a bi-Lipschitz function $\Phi_{\mathbf{x}}: B_{r}(\mathbf{x})\to\mathbb{R}^N$ such that
\begin{enumerate}[a)]
\item\label{conditiona} for any $\mathbf{z}_1$, $\mathbf{z}_2\in B_{r}(\mathbf{x})$ we have
$$L^{-1}\|\mathbf{z}_1-\mathbf{z}_2\|\leq\|\Phi_{\mathbf{x}}(\mathbf{z}_1)-\Phi_{\mathbf{x}}(\mathbf{z}_2)\|\leq L\|\mathbf{z}_1-\mathbf{z}_2\|;$$
\item\label{conditionb} $\Phi_{\mathbf{x}}(\mathbf{x})=0$ and
$\Phi_{\mathbf{x}}(K\cap B_{r}(\mathbf{x}))\subset \Pi=\{\mathbf{y}\in\mathbb{R}^N:\ y_N=0\}$.
\newcounter{enumi_saved}
\setcounter{enumi_saved}{\value{enumi}}
\end{enumerate}
We say that $\mathbf{x}\in K$ belongs to the interior of $K$ if there exists $\delta$, $0<\delta\leq r$, such that $B_{\delta}(0)\cap \Pi \subset \Phi_{\mathbf{x}}(K\cap B_{r}(\mathbf{x}))$.
Otherwise we say that $\mathbf{x}$ belongs to the boundary of $K$. We remark that the boundary of $K$ might be empty. Further we assume that

\begin{enumerate}[a)]\setcounter{enumi}{\value{enumi_saved}}
\item\label{conditionc} for any $\mathbf{x}$ belonging to the boundary of $K$, we have that
$$\Phi_{\mathbf{x}}(K\cap B_{r}(\mathbf{x}))=\Phi_{\mathbf{x}}(B_{r}(\mathbf{x}))\cap \Pi^+$$
where $\Pi^+=\{\mathbf{y}\in\mathbb{R}^N:\ y_N=0,\ y_{N-1}\geq 0\}$.
\end{enumerate}

%

We shall call $\mathcal{B}=\mathcal{B}(r,L,\Omega)$ the set of $K\subset\overline{\Omega}$ such that $K$ is a mildly Lipschitz hypersurface with constants $r$ and $L$.
We recall that $\mathcal{B}(r,L,\Omega)$  is closed, and actually compact, with respect to the Hausdorff distance, see \cite[Lemma~3.6]{MR}. For details about convergence in the Hausdorff distance, we refer, for instance, to \cite{DM}.

We call $\mathcal{D}=\mathcal{D}(r,L,\Omega)$ the class of sets $\partial D$ where $D\subset \Omega$ is an open set which is Lipschitz with constants $r$ and $L$. We have that
$\mathcal{D}(r,L,\Omega)\subset \mathcal{B}(\tilde{r},\tilde{L},\Omega)$, for some constants $\tilde{r}$, $\tilde{L}$ depending on $r$ and $L$ only. Moreover, also 
$\mathcal{D}(r,L,\Omega)$ is compact with respect to the Hausdorff distance. 

We further call $\hat{\mathcal{D}}=\hat{\mathcal{D}}(r,L,\Omega)$ the class of compact sets $\Sigma\subset\overline{\Omega}$ such that $\partial \Sigma\in \mathcal{D}(r,L,\Omega)$. Also this class is compact with respect to the Hausdorff distance.

Again following \cite{MR}, we combine different mildly Lipschitz hypersurfaces to obtain more complex structures.

\begin{defn}\label{classdefin}
Let us fix positive constants $r$, $L$, and a bounded open set $\Omega$.
Let us also fix $\omega:(0,+\infty)\to (0,+\infty)$ a nondecreasing left-continuous function.

We say that a compact set $K\subset \overline{\Omega}$ belongs to the class $\tilde{\mathcal{B}}=
\tilde{\mathcal{B}}(r,L,\Omega,\omega)$ if it satisfies the following conditions:

\begin{enumerate}[1)]
\item\label{1} $K=\bigcup_{i=1}^M K^i$ where $K^i\in\mathcal{B}(r,L,\Omega)$ for any $i=1,\ldots,M$;
\item\label{3} for any $i\in \{1,\ldots,M\}$, and any $\mathbf{x}\in K^i$, if its distance from the boundary of $K^i$ is $t>0$, then the distance of $\mathbf{x}$ from the union of $K^j$, with $j\neq i$, is greater than or equal to $\omega(t)$.
\end{enumerate}

We say that a compact set $\Sigma\subset \overline{\Omega}$ belongs to the class 
$\tilde{\mathcal{B}}_1=
\tilde{\mathcal{B}}_1(r,L,\Omega,\omega)$ if 
$\partial\Sigma\in \tilde{\mathcal{B}}(r,L,\Omega,\omega)$.
\end{defn}

Let us notice that in the previous definition the number $M$ may depend on $K$. However, there exists an integer $M_0$, depending on $r$, $L$, the diameter of $\Omega$, and $\omega$ only, such that $M\leq M_0$ for any $K\in \tilde{\mathcal{B}}$. We have that $\mathcal{H}^{N-1}(K)$ is bounded by a constant depending on $r$, $L$, the diameter of $\Omega$, and $M_0$ only, hence $|K|=0$. Furthermore, if we set as the boundary of $K$ the union of the boundaries of $K^i$, $i=1,\ldots,M$, then
the boundary of $K$ has $\mathcal{H}^{N-2}$ measure bounded by a constant again depending on $r$, $L$, the diameter of $\Omega$, and $M_0$ only. Finally, the number of connected components of $\mathbb{R}^N\backslash K$ is bounded by a constant $M_1$ depending on $r$, $L$, the diameter of $\Omega$, and $\omega$ only.

Without loss of generality, in the sequel we shall always assume that $\omega(t)\leq t$ for any $t>0$, and that
$\lim_{t\to+\infty}\omega(t)$ is equal to a finite real number which we call $\omega(+\infty)$.

The following compactness results hold.

\begin{lem}\label{compactHausdorff0}
The classes $\tilde{\mathcal{B}}$ and $\tilde{\mathcal{B}}_1$ introduced in Definition~\textnormal{\ref{classdefin}}
are compact under convergence in the Hausdorff distance.

Moreover, let $\Sigma\in \tilde{\mathcal{B}}_1$ and $\mathbf{x}\in\partial\Sigma$. We call $G=\mathbb{R}^N\backslash \Sigma$.
For any $r_1>0$, the number of connected components $U$ of $B_{r_1}(\mathbf{x})\cap G$ such that $\mathbf{x}\in\partial U$ is bounded by a constant $M_2$ depending on $r_1$, $r$, $L$, and $\omega$ only. Finally, the number of  connected components of $B_{r_1}(\mathbf{x})\cap G$ intersecting $B_{r_1/2}(\mathbf{x})$ 
is bounded by a constant $M_3$ depending on $r_1$, $r$, $L$, and $\omega$ only.
\end{lem}

\begin{proof}
The compactness of the class $\tilde{\mathcal{B}}$ is proved in \cite[Lemma~3.8]{MR}, whereas the compactness of the class
$\tilde{\mathcal{B}}_1$ and the second part of the lemma are proved in 
\cite[Lemma~2.6]{LPRX}.
\end{proof}

Finally, we consider the following definition.

\begin{defn}\label{scatdefin}
Let us fix a bounded open set $\Omega$ and
positive constants $r$, $L$, $0<r_1< r$ and $\tilde{C}>0$.
Let us also fix $\omega:(0,+\infty)\to (0,+\infty)$, 
a nondecreasing left-continuous functions.

We call $\hat{\mathcal{B}}=\hat{\mathcal{B}}(r,L,\Omega,r_1,\tilde{C},\omega)$
 the class of sets satisfying the following assumptions:

\begin{enumerate}[i)]
\item\label{uniformboundedness}
any $\Sigma\in \hat{\mathcal{B}}$ is a compact set contained in $\overline{\Omega}\subset\mathbb{R}^N$ such that $\Sigma\in\tilde{\mathcal{B}}_1(r,L,\Omega,\omega)$. We call $G=\mathbb{R}^N\backslash\Sigma$.
\item\label{pcondition} for any $\mathbf{x}\in \partial\Sigma$ and any $U$ connected
component of $G\cap B_{r_1}(\mathbf{x})$, with $\mathbf{x}\in\partial U$, we can find an open set $U'$
such that 
\begin{equation}\label{changecond1}
U\subset U'\subset G,
\end{equation}
and a bi-$W^{1,\infty}$ mapping $T:(-1,1)^{N-1}\times (0,1)\to U'$, with constant $\tilde{C}$, such that the following
properties hold. By the regularity of $Q=(-1,1)^{N-1}\times (0,1)$, $T$
can be actually extended up to the boundary and we have that $T:\overline{Q}\to\mathbb{R}^N$ is a Lipschitz map with Lipschitz
constant bounded by $\tilde{C}$. Furthermore, if we set
$\Gamma=[-1,1]^{N-1}\times\{0\}$, we require that 
\begin{equation}\label{changecond2}
T(0)=\mathbf{x}\quad \text{and}\quad\partial U\cap B_{r_1}(\mathbf{x})\subset T(\Gamma)\subset \partial G,
\end{equation}
and that,
for any $0<s<r_1$ and any $\mathbf{y}\in U\cap B_{r_1-s}(\mathbf{x})$, we have
\begin{equation}\label{changecond3}
\mathrm{dist}(T^{-1}(\mathbf{y}), \partial Q\backslash \Gamma)\geq \omega(s).
\end{equation}
\end{enumerate}
\end{defn}


\begin{rem}\label{propertiesrem}
We notice that clearly $T(\partial Q)=\partial U'$ and  
$\partial U\cap B_{r_1}(\mathbf{x})\subset\partial G$.

It is pointed out that, up to suitably changing the constants $r_1$ and $\tilde{C}$ involved, Condition~\textnormal{\ref{pcondition})} is satisfied provided it holds only for points $\mathbf{x}$ belonging to the boundaries of $K_i$, $i=1,\ldots,M$, where $\partial\Sigma=\bigcup_{i=1}^MK_i$ by Condition~\textnormal{\ref{uniformboundedness}}).

We also remark that,  for some constants and functions depending on $r$, $L$, and
the diameter of $\Omega$ only, we have
 $\hat{\mathcal{D}}(r,L,\Omega)\subset \hat{\mathcal{B}}(\tilde{r},\tilde{L},\Omega,r_1,\tilde{C},\omega)$.
\end{rem}

It is emphasised that Condition~\ref{pcondition}) is an extremely weak regularity condition and that it is satisfied by rather complex structures, see for instance the discussion on sets in $\mathbb{R}^3$ satisfying this assumption in Section 4 of \cite{Ron07}, where several examples are shown.


The following compactness result holds true, see \cite[Lemma~2.9]{LPRX}.

\begin{lem}\label{compactHausdorff}
The class $\hat{\mathcal{B}}$ introduced in Definition~\textnormal{\ref{scatdefin}}
is compact under convergence in the Hausdorff distance.
\end{lem}

We are in a position to define a quite general class of scatterers $\Sigma\subset \mathbb{R}^N$, $N\geq 2$. We need a quantitative assumption on the connectedness of $G=\mathbb{R}^N\backslash \Sigma$ as follows.
Let $\delta:(0,+\infty)\to (0,+\infty)$ be a nondecreasing left-continuous function.
Let $\Sigma$ be a compact set contained in $\mathbb{R}^N$. We say that $\Sigma$ satisfies the \emph{uniform exterior connectedness} with function $\delta$ if
for any $t>0$,
for any two points $x_1$, $x_2\in\mathbb{R}^N$ so that
$B_t(x_1)$ and $B_t(x_2)$ are contained in $\mathbb{R}^N\backslash \Sigma$, and for any $s$, $0<s<\delta(t)$,
then we can find a smooth (for instance $C^1$) curve $\gamma$ connecting
$x_1$ to $x_2$ so that $B_{s}(\gamma)$ is contained in
$\mathbb{R}^N\backslash \Sigma$ as well. 

Let us notice that such an assumption is closed under convergence in the Hausdorff distance and that $\delta(t)\leq t$ for any $t>0$. A detailed investigation on sufficient conditions for such an assumption to hold may be found in Section~2, in particular in Proposition~2.1, in \cite{LPRX}.

\begin{defn}\label{scatdefin2}
Let us fix positive constants $r$, $L$, and $R$, $0<r_1< r$ and $\tilde{C}>0$.
Let us also fix $\omega:(0,+\infty)\to (0,+\infty)$ and $\delta:(0,+\infty)\to (0,+\infty)$ two  nondecreasing left-continuous functions.

We call $\hat{\mathcal{B}}_{scat}=\hat{\mathcal{B}}_{scat}(r,L,R,r_1,\tilde{C},\omega,\delta)$  the class of compact sets $\Sigma$ belonging to $\hat{\mathcal{B}}(r,L,B_R,r_1,\tilde{C},\omega)$ 
and satisfying
the uniform exterior connectedness with function $\delta$.

We also define $\tilde{\mathcal{B}}_{scat}=\tilde{\mathcal{B}}_{scat}(r,L,R,\omega,\delta)$ the class of compact sets $\Sigma$ belonging to $\tilde{\mathcal{B}}_1(r,L,B_R,\omega)$ and 
satisfying
the uniform exterior connectedness with function $\delta$.

We further call $\hat{\mathcal{D}}_{obst}=\hat{\mathcal{D}}_{obst}(r,L,R)$ the class of compact sets $\Sigma$ belonging to $\hat{\mathcal{D}}(r,L,B_R)$
and such that $G=\mathbb{R}^3\backslash \Sigma$ is connected.
\end{defn}

Obviously, we have $\hat{\mathcal{B}}_{scat}(r,L,R,r_1,\tilde{C},\omega,\delta)\subset
 \tilde{\mathcal{B}}_{scat}(r,L,R,\omega,\delta)$.
We notice that any scatterer $\Sigma\in \hat{\mathcal{D}}_{obst}$ is indeed an obstacle, that is,
$\Sigma$ is the closure of its interior which is a bounded open set with Lipschitz boundary, with constants $r$ and $L$.
By Corollary~2.3 and Proposition~2.1 in \cite{LPRX}, 
for some constants and functions depending on $r$, $L$, and $R$ only, we have
 $\hat{\mathcal{D}}_{obst}(r,L,R)\subset \hat{\mathcal{B}}_{scat}(\tilde{r},\tilde{L},R,r_1,\tilde{C},\omega,\delta)$. 

By our earlier discussion, in particular by Lemmas~\ref{compactHausdorff0} and \ref{compactHausdorff}, it is easy to note that these classes
$\tilde{\mathcal{B}}_{scat}$, $\hat{\mathcal{B}}_{scat}$, and $\hat{\mathcal{D}}_{obst}$ are compact with respect to the Hausdorff distance.
 
\section{Mosco convergence for $H(\mathrm{curl})$ spaces
and higher integrability for solutions to Maxwell equations}\label{Moscosec}

In this section we consider in detail the Mosco convergence for $H(\mathrm{curl})$ spaces. 
Finally we discuss higher integrability properties related to solutions to the Maxwell equations. These will be the main ingredients needed to study the stability issue for the electromagnetic scattering problem with respect to variations of scatterers.

The Mosco convergence, introduced in \cite{Mos}, has been widely investigated for $H^1$ spaces since it is essentially equivalent to the stability of solutions of elliptic Neumann problems with respect to variations of the domain. 
We recall that, in dimension $2$, the problem is completely solved, \cite{Buc-Var,Buc-Var2}. The breakthrough was the sufficient condition proved in \cite{ChD}, which is still a quite useful one for the applications. In dimension $2$ extensive use is made of complex analytic techniques, in particular duality arguments. 
In dimension $3$ and higher, such a problem has been considered first in \cite{Gia} and then in \cite{MR}. In both cases, conditions of Lipschitz type are used. 

In the current article, instead of $H^1$ spaces,
we are mainly concerned with the Mosco convergence of $H(\mathrm{curl})$ spaces, which shall be of fundamental importance to study the solutions of Maxwell equations.

The general abstract definition of Mosco convergence is the following.
For a sequence $\{A_n\}_{n\in\mathbb{N}}$ of closed subspaces of a reflexive Banach space $X$,
we call
$$A'=\{x\in X:\ x=w\text{-}\lim_{k\to\infty} x_{n_k},\ x_{n_k}\in A_{n_k}\}$$
and
$$A''=\{x\in X:\ x=s\text{-}\lim_{n\to\infty} x_n,\ x_n\in A_n\}.$$
The sets $A'$ and $A''$ are subspaces of $X$ and we have that $A''\subset A'$, and that $A''$ is closed. We say that
$A_n$ converges, as $n\to\infty$, to a closed subspace $A$ in the sense of Mosco if $A=A'=A''$. Equivalently, the following two conditions need to be satisfied:
\begin{enumerate}[i)]
\item for any $x\in X$, if there exists a subsequence $A_{n_k}$ and a sequence $x_k$, $k\in\mathbb{N}$,
such that $x_k$ converges weakly to $x$ as $k\to\infty$ and $x_k\in A_{n_k}$ for any $k\in\mathbb{N}$, then $x\in A$;\label{Mosco1}
\item for any $x\in A$, there exists a sequence $x_n\in A_n$, $n\in\mathbb{N}$, converging strongly to $x$
as $n\to\infty$.\label{Mosco2}
\end{enumerate}

For any bounded open set $D$ contained in $\mathbb{R}^3$, we call $\mathcal{K}$ the set of all compact subsets of $\overline{D}$. It is well-known that $\mathcal{K}$ is compact with respect to the Hausdorff distance. 
Moreover, if $K_n\in\mathcal{K}$, $n\in\mathbb{N}$, converges in the Hausdorff  distance, as $n\to\infty$, to $K\in\mathcal{K}$, then we also have $\lim_{n\to\infty}|K_n\backslash K|=0$. 

For any $K\in\mathcal{K}$, we consider the isometric immersion of $H(\mathrm{curl},D\backslash K)$ into $L^2(D,\mathbb{C}^6)$ as follows. To each $u\in H(\mathrm{curl},D\backslash K)$ we associate
$(u,\nabla\wedge u)\in L^2(D,\mathbb{C}^6)$ with the convention that 
$u$ and $\nabla\wedge u$ are extended to zero in $K$. With this convention, $H(\mathrm{curl},D\backslash K)$ may be considered as a closed subspace of $L^2(D,\mathbb{C}^6)$.

Given a sequence $\{K_n\}_{n\in\mathbb{N}}$ contained in $\mathcal{K}$ and $K\in\mathcal{K}$, we say that $H(\mathrm{curl},D\backslash K_n)$ converges, as $n\to\infty$, to $H(\mathrm{curl},D\backslash K)$ \emph{in the sense of Mosco}
if this holds by the previous abstract definition considering 
$H(\mathrm{curl},D\backslash K_n)$, $n\in\mathbb{N}$, and 
$H(\mathrm{curl},D\backslash K)$ as subspaces of 
$L^2(D,\mathbb{C}^6)$.

We wish to find general sufficient conditions on $K_n$, $n\in\mathbb{N}$, and $K$  in order to have such a Mosco convergence for the
corresponding $H(\mathrm{curl})$ spaces. The result is the following.

\begin{thm}\label{Moscoconv}
Let us fix positive constants $r$, $L$, and $\omega:(0,+\infty)\to (0,+\infty)$ a nondecreasing left-continuous function. 
Let $D$ be a bounded open set contained in $\mathbb{R}^3$. Let 
$\tilde{\mathcal{B}}=
\tilde{\mathcal{B}}(r,L,D,\omega)$.

Let, for any $n\in \mathbb{N}$, $K_n\subset \overline{D}$ be a compact set such that
$\partial (D\backslash K_n)\in \tilde{\mathcal{B}}$.

If, as $n\to\infty$, $K_n$ converges to a compact $K$ in the Hausdorff distance, then
$H(\mathrm{curl},D\backslash K_n)$ converges to $H(\mathrm{curl},D\backslash K)$
in the sense of Mosco.
\end{thm}

\begin{proof} 
We assume that $\overline{D}\subset B_R$, for some $R>0$.
We call $\tilde{K}_n=K_n\cup (\overline{B_{R+1}}\backslash D)$, $n\in\mathbb{N}$,
and $\tilde{K}=K\cup (\overline{B_{R+1}}\backslash D)$. We have that
$\partial \tilde{K}_n=\partial B_{R+1}\cup \partial (D\backslash K_n)$, $n\in\mathbb{N}$, and
$\partial \tilde{K}=\partial B_{R+1}\cup \partial (D\backslash K)$. Clearly,
as $n\to\infty$, $\tilde{K}_n$ converges to $\tilde{K}$ in the Hausdorff distance.
We set $D_1=B_{R+1}$ and we rename $D=B_{R+2}$.

Therefore, without loss of generality, we can assume that $D$ is a Lipschitz domain and there exists $D_1$, a Lipschitz domain compactly contained in $D$, such that
$K_n$, $n\in\mathbb{N}$, and $K$ are contained in $\overline{D_1}$. 

Since it is not restrictive to pass to subsequences, by the arguments used in the proof of Lemma~\ref{compactHausdorff0},
we can assume that not only $K_n\to K$ but also $\partial K_n\to \partial K$ in the Hausdorff distance, as $n\to\infty$. Moreover, $\partial K\in \tilde{B}$ as well.

We call $A_n=H(\mathrm{curl},D\backslash K_n)$, $n\in\mathbb{N}$, and $A=H(\mathrm{curl},D\backslash K)$.
For simplicity and without loss of generality we consider only real vectors valued functions.

We begin with the following restrictive assumption. It is assumed that $K_n$ has no interior point, that is $K_n=\partial K_n$, $n\in\mathbb{N}$, and consequently
$K=\partial K$ as well.

We first prove that if $|K\backslash K_n|\to 0$ as $n\to\infty$, then 
$A'\subset A$. By our restrictive hypothesis, we have that $|K|=0$, hence we conclude that in this case $A'\subset A$.

For proving the claim, let us consider a subsequence $\{A_{n_k}\}_{k\in\mathbb{N}}$ and let
$(u_k,\nabla\wedge u_k)\in A_{n_k}$ for any $k\in\mathbb{N}$. We
assume that $(u_k,\nabla\wedge u_k)$ converges to $(u,V)$ weakly in $L^2(D,\mathbb{R}^6)$. For any point $\mathbf{x}$ and any $r>0$ such that
$\overline{B_r(\mathbf{x})}\subset D\backslash K$, it is easy to show that $u|_{B_r(\mathbf{x})}\in
H(\mathrm{curl},B_r(\mathbf{x}))$ and $\nabla\wedge u=V$ in $B_r(\mathbf{x})$. Therefore,
$u\in
H(\mathrm{curl},D\backslash K)$ and $\nabla\wedge u=V$ in $D\backslash K$.
It remains to prove that $(u,V)$ are identically equal to $0$ in $K$. If $|K\backslash K_n|\to 0$ as $n\to\infty$, this property follows by the arguments of Lemma~2.3 in \cite{MR}.

The next step is to prove that $A\subset A''$, that is, for every $(u,\nabla\wedge u) \in A$ there exists $(u_n,\nabla\wedge u_n)\in A_n$ such that $(u_n,\nabla\wedge u_n)$ converges, as $n\to\infty$, to $(u,\nabla\wedge u)$ in $L^2(\Omega,\mathbb{R}^6)$.
Let us recall that it is enough to prove that for any subsequence
$A_{n_{k}}$ there exists a further subsequence $A_{n_{k_j}}$ with this property. Hence, without loss of generality, we can always pass to subsequences, which we usually do not relabel.

Since $A''$ is closed, it is enough to prove the result for any $(u,\nabla\wedge u)$ in a dense subset of $A$. Let us denote with $\hat{K}_n$, $n\in\mathbb{N}$, and
$\hat{K}$ the boundary of $K_n$ and $K$, respectively, as defined for the elements of
the class $\tilde{\mathcal{B}}$ after Definition~\ref{classdefin}. We recall that, up to a subsequence, $\hat{K}_n$ converges to $\hat{K}$ in the Hausdorff distance as $n\to\infty$, see \cite[Lemma~3.6]{MR}.

We consider the following subset of $A$
$$\tilde{A}=\{u\in A:\ u\in L^{\infty}(D,\mathbb{R}^3)\text{ and }u=0\text{ in a neighborhood of }\hat{K}\}.$$
We shall prove that $\tilde{A}$ is dense in $A$.

By the density result of Proposition~\ref{Linftydensity}, we need to prove that for any $u\in A\cap L^{\infty}(D,\mathbb{R}^3)$, we can find a sequence $\{u_n\}_{n\in\mathbb{N}}\in \tilde{A}$ converging to $u$ in $A$.
Since $\mathcal{H}^{N-2}(\hat{K})$ is finite, then $\hat{K}$ has zero capacity. Hence for any $U$, an open neighbourhood of $\hat{K}$ compactly contained in $D$, and for any $\varepsilon>0$,
there exists a function $\chi_{\varepsilon}$ such that $\chi_{\varepsilon}\in H^1(D)$, $0\leq\chi_{\varepsilon}\leq 1$ almost everywhere in $D$, $\chi_{\varepsilon}=1$ almost everywhere outside $U$, $\chi_{\varepsilon}=0$ almost everywhere in a neighbourhood of $\hat{K}$, and
$$\int_{D}\|\nabla \chi_{\varepsilon}\|^2\leq \varepsilon.$$
Take $u\in H(\mathrm{curl},D\backslash K)\cap L^{\infty}(D,\mathbb{R}^3)$.
Then $\chi_{\varepsilon}u\in \tilde{A}$ and
$$\|\chi_{\varepsilon}u-u\|_{L^2(D\backslash K,\mathbb{R}^3)}\leq \|u\|_{L^2(U,\mathbb{R}^3)}$$
and
\begin{multline*}
\|\nabla\wedge (\chi_{\varepsilon}u)-\nabla\wedge u\|_{L^2(D\backslash K,\mathbb{R}^3)}\\\leq  
\|(\chi_{\varepsilon}-1)\nabla\wedge u\|_{L^2(D\backslash K,\mathbb{R}^3)}+\|\nabla\chi_{\varepsilon}\wedge u\|_{L^2(D\backslash K,\mathbb{R}^3)}\\\leq 
\|\nabla\wedge u\|_{L^2(U,\mathbb{R}^3)}+\|u\|_{L^{\infty}(D,\mathbb{R}^3)}\|
\|\nabla \chi_{\varepsilon}\|_{L^2(U,\mathbb{R}^3)}.
\end{multline*}
Since $U$ and $\varepsilon$ are arbitrary, we conclude that $\tilde{A}$ is dense in $A$.

Take $u\in\tilde{A}$ and let $\tilde{U}$ be an open set compactly contained in $D$ such that $\hat{K}$ is contained in $\tilde{U}$ and $u$ is zero
on $\tilde{U}$. We can find an open subset $D_0$ compactly contained in $D\backslash K$, a finite number of points $\mathbf{x}_j\in \partial D\cup K$ and positive numbers $\delta_j<\delta_j'$, $j=1,\ldots,m$, such that
$B_{\delta_j'}(\mathbf{x}_j)\cap \hat{K}=\emptyset$ for any $j=1,\ldots,m$ and
$$\overline{D}\subset \tilde{U}\cup D_0\cup\left(\bigcup_{j=1}^m B_{\delta_j}(\mathbf{x}_j)\right).$$
Moreover,
we assume that for any $j=1,\ldots,m_0$, $\mathbf{x}_j\in K$ and 
$B_{\delta_j'}(\mathbf{x}_j)\cap \partial D$ is empty, whereas for any
$j=m_0+1,\ldots,m$, $\mathbf{x}_j\in \partial D$ and 
$B_{\delta_j'}(\mathbf{x}_j)\cap K$ is empty.

If $K=\bigcup_{i=1}^M K^i$ as in Definition~\ref{classdefin},
 for any $j=1,\ldots,m_0$, we have that $\mathbf{x}_j$ belongs to the interior of $K^{i(j)}$ for some $i=i(j)\in\{1,\ldots,M\}$ and we may assume that 
$2\delta_j'\leq r$ and that $B_{\delta_j'}(\mathbf{x}_j)\cap  K=B_{\delta_j'}(\mathbf{x}_j)\cap K^{i(j)}$.

By using a partition of unity, it is enough to consider a function $u\in\tilde{A}$ that is compactly supported either in $B_{\delta_j}(\mathbf{x}_j)$, for some $j\in\{1,\ldots,m\}$, or in $D_0$. In the latter case, we have that $D_0\subset (D
\backslash K_n)$, and hence $u\in A_n$, for any $n$ large enough, so the convergence is trivially proved.

We limit ourselves to consider $\mathbf{x}_j\in K$, that is $j\in\{1,\ldots,m_0\}$, the case in which $\mathbf{x}_j\in\partial D$ is completely analogous.

We show the required convergence for $u\in\tilde{A}$ that is compactly supported
in $B_{\delta}(\mathbf{x})$ for some $0<\delta<\delta'\leq r/2$ and $\mathbf{x}\in K^i$ such that
$B_{\delta'}(\mathbf{x})\cap (\partial D\cup \hat{K})=\emptyset$ and
$B_{\delta'}(\mathbf{x})\cap K=B_{\delta'}(\mathbf{x})\cap K^i$, for some $i\in\{1,\ldots,M\}$.

We adapt the reasoning developed in the proof of Theorem~4.2 in \cite{Gia}  to the $H(\mathrm{curl})$ case.
Possibly passing to a subsequence, let $\mathbf{x}_n\in K_n$ converge to $\mathbf{x}$ and
$\Phi^n_{\mathbf{x}_n}$ converge to a function $\Phi_{\mathbf{x}}:B_r(\mathbf{x})\to\mathbb{R}^3$ satisfying Conditions~\ref{conditiona}) and \ref{conditionb}) of the definition of a mildly Lipschitz hypersurface.
Without loss of generality, we may assume that $B_{\delta}(\mathbf{x})\subset \Phi_{\mathbf{x}}^{-1}(B_{r_1})$ for some positive $r_1$ such that
$B_{r_1}\subset \Phi_{\mathbf{x}}(B_{\delta'}(\mathbf{x}))$. Moreover, we may also assume that $B_{r_1}\cap \Pi\subset 
\Phi_{\mathbf{x}}(B_{\delta'}(\mathbf{x}))\cap \Pi=
\Phi_{\mathbf{x}}(B_{\delta'}(\mathbf{x})\cap K^i)$ and $B_{r_1}\cap \Pi\subset 
\Phi^n_{\mathbf{x}_n}(B_{\delta'}(\mathbf{x}_n))\cap \Pi=
\Phi^n_{\mathbf{x}_n}(B_{\delta'}(\mathbf{x}_n)\cap K_n)$ for any $n$
large enough.

Let $v=\widehat{\Phi_{\mathbf{x}}^{-1}}(u)=(J\Phi_{\mathbf{x}}^{-1})^T u(\Phi_{\mathbf{x}}^{-1})$. Then $v\in H(\mathrm{curl},B_{r_1}\backslash \Pi)$, with bounded support in $B_{r_1}$.
We denote by $v^{\pm}$ the function $v$ restricted to the the halfspaces $T^{\pm}=\{\mathbf{y}\in\mathbb{R}^3:\ \pm y_3>0\}$, respectively. Then, by a reflection as developed in the previous section, we may define two $H_0(\mathrm{curl},B_{r_1})$ functions, $\tilde{v}^{\pm}$ such that $\tilde{v}^{\pm}=v^{\pm}$ on $T^{\pm}$. Let $\tilde{u}^{\pm}=\widehat{\Phi_{\mathbf{x}}}(\tilde{v}^{\pm})\in H_0(\mathrm{curl},B_{\delta'}(\mathbf{x}))$.
We can assume that $B_r(\mathbf{x})$ is compactly contained in $D$, therefore by extending them to zero, we have that $\tilde{u}^{\pm}\in H_0(\mathrm{curl},D)$ and
have compact support contained in $B_{\delta'}(\mathbf{x})$ and, for $n$ large enough, in $B_{\delta'}(\mathbf{x}_n)$ as well.
We then define
$$u_n=\left\{\begin{array}{ll}
\tilde{u}^{+}(\mathbf{x}) &\text{if }\Phi^n_{\mathbf{x}_n}(\mathbf{x})\in T^+\\
\tilde{u}^{-}(\mathbf{x}) &\text{if }\Phi^n_{\mathbf{x}_n}(\mathbf{x})\in T^-.
\end{array}\right.
$$
By construction we have that $u_n\in H(\mathrm{curl},D\backslash K_n)$
and has compact support contained in $B_{\delta'}(\mathbf{x})$. Furthermore, $u_n$ and $\nabla\wedge u_n$ converge almost everywhere in $D\backslash K$ to $u$ and $\nabla\wedge u$, respectively.

Moreover, almost everywhere in $B_{\delta'}(\mathbf{x})\backslash K_n$, we have
$$\|u_n\|+\|\nabla\wedge u_n\|\leq \|\tilde{u}^{+}\|+\|\tilde{u}^{-}\|+
\|\nabla\wedge\tilde{u}^{+}\|+\|\nabla\wedge\tilde{u}^{-}\|.$$
We immediately conclude by Lebesgue theorem that 
$(u_n,\nabla\wedge u_n)$ converges to $(u,\nabla \wedge u)$, as $n\to\infty$, in $L^2(D,\mathbb{R}^6)$.

The general case, that is when we drop the assumption that 
$K_n= \partial K_n$ for any $n\in\mathbb{N}$, is a consequence of the next general lemma.
\end{proof}

\begin{lem}\label{bdryversus}
Let $D$ be a bounded open set and $D_1$ a bounded open set compactly contained in $D$. Let $\{K_n\}_{n\in\mathbb{}N}$ be a sequence of compact sets, and $K$ and $\tilde{K}$ be two compact sets, all of them contained in $\overline{D_1}$.

Let us assume that $K_n\to K$ and $\partial K_n\to \tilde{K}$ in the Hausdorff distance as $n\to\infty$.

If, as $n\to\infty$, $\tilde{A}_n=H(\mathrm{curl},D\backslash \partial K_n)$ converges to 
 $\tilde{A}=H(\mathrm{curl},D\backslash \tilde{K})$ in the sense of Mosco, then one also has that
$A_n=H(\mathrm{curl},D\backslash K_n)$ converges to 
$A=H(\mathrm{curl},D\backslash K)$ in the sense of Mosco.
\end{lem}

\begin{proof}
The proof that $A'\subset A$ follows from a similar argument of the first part of the proof of Proposition~2.2 in \cite{MR}, observing that, for any $n\in\mathbb{N}$, we trivially have
that $A_n\subset\tilde{A}_n$.

We need to show that $A\subset A''$. Let $u\in A$. We recall that $\partial K\subset \tilde{K}\subset K$, hence
$D\backslash \tilde{K}= (D\backslash K)\cup (\stackrel{\circ}{K}\backslash\tilde{K})$ and
$u\in \tilde{A}$ as well. Then there exists
a sequence $\{u_n\}_{n\in\mathbb{N}}$ such that $u_n\in \tilde{A}_n$ for any $n\in\mathbb{N}$ and $(u_n,\nabla\wedge u_n)$ converges to $(u,\nabla\wedge u)$ in $L^2(D,\mathbb{R}^6)$ and, up to a subsequence, also almost everywhere in $D$, as $n\to\infty$. Let us consider $v_n=u_n$ in $D\backslash K_n$ and zero everywhere else. It is easy to show that $v_n\in A_n$. For any $\mathbf{x}\in D\backslash K$, there exists $r>0$ and $\overline{n}\in\mathbb{N}$ such that $B_r(x)\cap K_n=\emptyset$ for any $n\geq \overline{n}$. Therefore, we have that $(v_n,\nabla\wedge v_n)$ converges to $(u,\nabla\wedge u)$ almost everywhere in $D\backslash K$, and hence by the Lebesgue theorem also in $L^2(D\backslash K,\mathbb{R}^6)$. On the other hand,
\begin{multline*}
\|v_n-u\|_{L^2(K,\mathbb{R}^3)}=
\|v_n\|_{L^2(K,\mathbb{R}^3)}\\\leq \|u_n\|_{L^2(K,\mathbb{R}^3)}=
\|u_n-u\|_{L^2(K,\mathbb{R}^3)}\to 0\quad\text{as }n\to\infty.\end{multline*}
The same reasoning holds for the curl and the proof is concluded.
\end{proof}

We conclude this section by discussing the following higher integrability property for solutions of the Maxwell equations.

\begin{defn}\label{highintdefin}
Let $\Sigma$ be a compact set contained in $\overline{B_{R_0}}$ for some $R_0>0$.
We say that $\Sigma$ satisfies the \emph{Maxwell higher integrability property} if
for any constants $0<a_0<a_1$,
there exist a constant $p>2$, depending on $\Sigma$, $R_0$, $a_0$, and $a_1$, and   
a constant $C$, depending on $\Sigma$, $R_0$, $p$, $a_0$, and $a_1$, such that for any $a\in 
L^{\infty}(B_{R_0+1},M^{3\times 3}_{sym}(\mathbb{R}))$ satisfying \eqref{ellipticity} in $B_{R_0+1}$, and
any $u$ belonging either to $H_0(\mathrm{curl},B_{R_0+1}\backslash \Sigma)\cap H(\mathrm{div}_a,B_{R_0+1}\backslash \Sigma)$, or to 
$H(\mathrm{curl},B_{R_0+1}\backslash \Sigma)\cap H_0(\mathrm{div}_a,B_{R_0+1}\backslash \Sigma)$,
we have
\begin{multline}\label{higherintestdef}
\|u\|_{L^p(B_{R_0+1}\backslash \Sigma,\mathbb{C}^3)}\\
\leq C\left[\|u\|_{H(\mathrm{curl},B_{R_0+1}\backslash \Sigma)}+
\|\nabla\cdot(au)\|_{L^2(B_{R_0+1}\backslash \Sigma)}\right].
\end{multline}

\end{defn}

By using Proposition~\ref{Druprop}, we can immediately deduce the following corollary.

\begin{cor}\label{highintcorollary}
Let $\Sigma$ be a compact set contained in $\overline{B_{R_0}}$, for some $R_0>0$, such that $\Sigma$ satisfies the Maxwell higher integrability property.

Then 
for any constants
$0<\overline{k}$, and $0<\lambda_0<\lambda_1$, there exists a constant $s_0>2$,
depending on $R_0$, $\lambda_0$, and $\lambda_1$ only, such that the following holds. If we call $s=\min\{s_0,p\}$, with $p$ as in the definition of the Maxwell higher integrability property with $a_0=\lambda_0$ and $a_1=\lambda_1$, then $s>2$ and we can 
find a constant $C_1$, depending on the same $p$, $s_0$, $R_0$, $\overline{k}$, $\lambda_0$, $\lambda_1$,
and the constant $C$ in \eqref{higherintestdef} only,
such that for any $\epsilon$, $\mu\in 
L^{\infty}(B_{R_0+1},M^{3\times 3}_{sym}(\mathbb{R}))$ satisfying \eqref{ellipticity2} in $B_{R_0+1}$, and
any $(\mathbf{E},\mathbf{H})$ solving
$$\left\{\begin{array}{ll}
\nabla \wedge\mathbf{E}-\rmi k \mu \mathbf{H}=0,\ \nabla \wedge\mathbf{H}+\rmi k\epsilon \mathbf{E}=0&\text{in }B_{R_0+1}\backslash \Sigma\\
\nu\wedge \mathbf{E}=0 &\text{on }\partial \Sigma
\end{array}\right.
$$
for some $0<k\leq\overline{k}$,
we have
\begin{multline}\label{higherintest}
\|\mathbf{E}\|_{L^s(B_{R_0+1}\backslash \Sigma,\mathbb{C}^3)}+
\|\mathbf{H}\|_{L^s(B_{R_0+1}\backslash \Sigma,\mathbb{C}^3)}\\
\leq C_1\big[\|\mathbf{E}\|_{L^2(B_{R_0+1}\backslash \Sigma,\mathbb{C}^3)}+
\|\mathbf{H}\|_{L^2(B_{R_0+1}\backslash \Sigma,\mathbb{C}^3)}\\+
\|\nu\wedge \mathbf{E}\|_{L^2(\partial B_{R_0+1},\mathbb{C}^3)}+\|\nu\wedge \mathbf{H}\|_{L^2(\partial B_{R_0+1},\mathbb{C}^3)}\big].
\end{multline}
\end{cor}

In the next proposition we give a sufficient condition for
the Maxwell higher integrability property of Definition~\ref{highintdefin} to hold.

\begin{prop}\label{highintprop}
Let us fix positive constants $r$, $L$, and $R_0$, $0<r_1< r$ and $\tilde{C}>0$, and $\omega:(0,+\infty)\to (0,+\infty)$ a nondecreasing left-continuous function. Let 
$\hat{\mathcal{B}}=
\hat{\mathcal{B}}(r,L,B_{R_0},r_1,\tilde{C},\omega)$.

Then any $\Sigma\in\hat{\mathcal{B}}$ satisfies the Maxwell higher integrability property, with constants in \eqref{higherintestdef} $p>2$, depending on $\hat{\mathcal{B}}$, $a_0$, and $a_1$ only, and   
$C$, depending on $\hat{\mathcal{B}}$, $p$, $a_0$, and $a_1$ only.

Moreover, there exist constants $p_1>2$ and $C_1$, depending on $\hat{\mathcal{B}}$ only, such that, for any $\Sigma\in\hat{\mathcal{B}}$, we have
\begin{equation}\label{Sobolevcondition}
\|v\|_{L^{p_1}(B_{R_0+1}\backslash \Sigma)}\leq C_1\|v\|_{H^1(B_{R_0+1}\backslash \Sigma)}\quad\text{for any }v\in H^1(B_{R_0+1}\backslash \Sigma).
\end{equation}

Finally, $B_{R_0+1}\backslash \Sigma$ satisfies the Maxwell and Rellich compactness properties.
\end{prop}

\begin{proof} We observe that \eqref{Sobolevcondition} and the Rellich compactness property are proved in \cite[Proposition~2.11]{LPRX}.

We begin with the following interior estimate. Given two bounded domains $D'$ and $D$, with $D'$ compactly contained in $D$ and $D$ Lipschitz, let us consider a function $u\in H(\mathrm{curl},D)\cap H(\mathrm{div}_a,D)$.
We fix a cutoff function $\chi\in C^{\infty}_0(D)$ such that $0\leq \chi\leq 1$ everywhere and $\chi$ is identically equal to $1$ in $D'$.
Then $\chi u\in H_0(\mathrm{curl},D)\cap H(\mathrm{div}_a,D)$ and
\begin{multline}\label{cutoffestimate}
\|\chi u\|_{L^2(D,\mathbb{C}^3)}+\|\nabla\wedge(\chi u)\|_{L^2(D,\mathbb{C}^3)}+
\|\nabla\cdot(a\chi u)\|_{L^2(D)}\\\leq
C_1\left[\| u\|_{L^2(D,\mathbb{C}^3)}+\|\nabla\wedge u\|_{L^2(D,\mathbb{C}^3)}+
\|\nabla\cdot(a u)\|_{L^2(D)}\right].
\end{multline}
Here $C_1$ depends only on $\|\nabla\chi\|_{L^{\infty}(D,\mathbb{R}^3)}$, thus on
the distance of $D'$ from the boundary of $D$.

By Proposition~\ref{Druprop}, there exist $s>2$ and a positive constant $C_2$, depending on $D$, $a_0$, and $a_1$ only, such that
\begin{multline}\label{useofDru}
\|\chi u\|_{L^s(D,\mathbb{C}^3)}\\\leq
C_2\left[\|\chi u\|_{L^2(D,\mathbb{C}^3)}
+\|\nabla\wedge(\chi u)\|_{L^2(D,\mathbb{C}^3)}+
\|\nabla\cdot(a\chi u)\|_{L^2(D)}\right].
\end{multline}

We conclude that there exist constants $s>2$, depending on $D$, $a_0$, and $a_1$ only, and $C_3$, depending on $D$, $a_0$, $a_1$, and the distance of $D'$ from the boundary of $D$, such that
\begin{multline}\label{cutoffestimatefinal}
\|u\|_{L^s(D',\mathbb{C}^3)}
\\\leq
C_3\left[\| u\|_{L^2(D,\mathbb{C}^3)}+\|\nabla\wedge u\|_{L^2(D,\mathbb{C}^3)}+
\|\nabla\cdot(a u)\|_{L^2(D)}\right].
\end{multline}

Then we consider the following local construction. We fix $\Sigma\in\hat{\mathcal{B}}$ and we call $G=\mathbb{R}^3\backslash \Sigma$.
We fix $\mathbf{x}\in \partial \Sigma$ and $U$ a connected component of
$B_{r_1}(\mathbf{x})\cap G$ such that $\mathbf{x}\in\partial U$. We consider $U'$
and $T:Q\to U'$ as in Condition~\ref{pcondition}) of Definition~\ref{scatdefin}.

First of all, by \eqref{changecond3}, we infer that $T^{-1}(U\cap B_{3r_1/4}(\mathbf{x}))$ has a positive distance from $\partial Q\backslash \Gamma$ which is bounded from below by a positive constant $r_2\leq 1/8$ depending on $r_1$ and $\omega$ only.

Let us take $u$ belonging to $H_0(\mathrm{curl},B_{R_0+1}\backslash \Sigma)\cap H(\mathrm{div}_a,B_{R_0+1}\backslash \Sigma)$, or to 
$H(\mathrm{curl},B_{R_0+1}\backslash \Sigma)\cap H_0(\mathrm{div}_a,B_{R_0+1}\backslash \Sigma)$. 
Without loss of generality, by an easy extension argument around $\partial B_{R_0+1}$, in either cases we can assume that $u$ and $a$ are defined everywhere outside $\overline{B_{R_0+1}}$ and that $u\in H(\mathrm{curl},\mathbb{R}^3\backslash \Sigma)\cap H(\mathrm{div}_a,\mathbb{R}^3\backslash \Sigma)$ with bounded support and with norms controlled by a constant $C$ times
the corresponding ones in $B_{R_0+1}\backslash \Sigma$.

Then $v=\hat{T}(u)\in H(\mathrm{curl},Q)\cap H(\mathrm{div}_{a_1},Q)$,
where $a_1=T_{\ast}(a)$. Furthermore,
$\nabla\wedge v=0$, or $\nu\cdot(a_1 v)=0$ respectively, on any compact subset of the interior of $\Gamma$, with respect to the induced topology.

Therefore, by a reflection argument around the plane $\Pi$ containing $\Gamma$, in both cases we can extend $v$ to a function $\tilde{w}\in H(\mathrm{curl},Q_1)\cap H(\mathrm{div}_{a_2},Q_1)$ where $Q_1=[-1,1]^3$ and $a_2$
satisfies, calling $T_{\Pi}$ the reflection in $\Pi$,
$$a_2=\left\{\begin{array}{ll}
a_1
&\text{in }Q\\
(T_{\Pi})_{\ast}(a_1)&\text{in }T_{\Pi}(Q).
\end{array}\right.$$

By using the previous interior estimates applied to $Q$ and $Q'=\{\mathbf{x}\in Q: \mathrm{dist}(\mathbf{x},\partial Q)> r_2/2\}$, we deduce that there exist constants $s>2$ and $C$,
depending on $a_0$, $a_1$, $\tilde{C}$, and $r_2$ only, such that
\begin{multline}\label{cutoffestimatefinalbis}
\|u\|_{L^s(U\cap B_{3r_1/4}(\mathbf{x}),\mathbb{C}^3)}
\\\leq
C\left[\| u\|_{L^2(U',\mathbb{C}^3)}+\|\nabla\wedge u\|_{L^2(U',\mathbb{C}^3)}+
\|\nabla\cdot(a u)\|_{L^2(U')}\right].
\end{multline}

Then we proceed with the following covering argument, which was developed in the proof of \cite[Proposition~2.11]{LPRX}.
For any $\mathbf{x}\in\partial\Sigma$, let $W_n$, $n=1,\ldots,n_0$, be the connected components of $B_{r_1/2}(\mathbf{x})\cap G$ such that $W_n\cap B_{r_1/4}(\mathbf{x})\neq\emptyset$. By Lemma~\ref{compactHausdorff0}, $n_0\leq M_3$, where $M_3$ is a constant depending on $r_1$, $r$, $L$, and $\omega$ only.
As in the proof of Proposition~2.11 in \cite{LPRX}, one can show that for any $\mathbf{x}\in\partial\Sigma$, there exist $n_0$ points $\mathbf{x}_1,\ldots,\mathbf{x}_{n_0}$, with $n_0\leq M_3$, with the following property.
For any $n=1,\ldots,n_0$, there exists
$U_n$, a connected component of
$B_{r_1}(\mathbf{x}_n)\cap G$, such that $\mathbf{x}_n\in\partial U_n$, and
moreover
$$B_{r_1/4}(\mathbf{x})\cap G\subset \bigcup_{n=1}^{n_0}(U_n\cap B_{3r_1/4}(\mathbf{x}_n)).$$

We fix $\delta=r_1/16$ and define the compact set $A_1=\overline{B_{\delta}(\partial\Sigma)\cap G}$.
We can find a finite number of points $\mathbf{z}_i\in \partial\Sigma$, $i=1,\ldots,m_1$, such that
$$A_1\subset \bigcup_{i=1}^{m_1}B_{r_1/4}(\mathbf{z}_i).$$
With a rather simple construction, it is possible to choose $m_1$ depending on $r_1$ and $R_0$ only, for instance by taking points such that $B_{r_1/16}(\mathbf{z}_i)\cap B_{r_1/16}(\mathbf{z}_j)$ is empty for $i\neq j$.

We further find a finite number of points $\mathbf{z}_i\in \partial B_{R_0+1}$, $i=m_1+1,\ldots,m_1+m_2$, such that
$$A_2=\overline{B_{1/16}(\partial B_{R_0+1})}\subset\bigcup_{i=m_1+1}^{m_1+m_2}B_{1/4}(\mathbf{z}_i),$$ with $m_2$ depending on $R_0$ only.
We call $r_3=\min\{1,r_1\}$ and $A_3=\{\mathbf{x}\in B_{R_0+1}\backslash \Sigma:\ \mathrm{dist}(\mathbf{x},\partial(B_{R_0+1}\backslash \Sigma)) \geq r_3/16\}$.
We can find points $\mathbf{z}_i\in A_3$, $i=m_1+m_2+1,\ldots,m_1+m_2+m_3$, such that
$$A_3\subset \bigcup_{i=m_1+m_2+1}^{m_1+m_2+m_3}B_{r_3/32}(\mathbf{z}_i).$$
Again $m_3$ may be bounded by a constant depending on $r_1$ and $R_0$ only.

We apply the local argument developed at the beginning of the proof, at most $M_3$ times for any $\mathbf{z}_i$, $i=1,\ldots,m_1$, and we can find $s_1>2$ and $C_1$ such that
\begin{multline*}
\|u\|_{L^{s_1}(A_1\cap G)}\leq C_1(M_3m_1)C\big[\| u\|_{L^2(B_{R_0+1}\backslash K,\mathbb{C}^3)}\\+\|\nabla\wedge u\|_{L^2(B_{R_0+1}\backslash K,\mathbb{C}^3)}+
\|\nabla\cdot(a u)\|_{L^2(B_{R_0+1}\backslash K)}\big].
\end{multline*}
Analogously, we can find $s_2>2$ and $C_2$
such that
\begin{multline*}
\|u\|_{L^{s_2}(A_2\cap B_{R_0+1})}\leq C_2m_2C\big[\| u\|_{L^2(B_{R_0+1}\backslash K,\mathbb{C}^3)}\\+\|\nabla\wedge u\|_{L^2(B_{R_0+1}\backslash K,\mathbb{C}^3)}+
\|\nabla\cdot(a u)\|_{L^2(B_{R_0+1}\backslash K)}\big].
\end{multline*}
Finally, we apply the interior estimate with $D=B_{r_3/16}(\mathbf{z}_i)$
and $D'=B_{r_3/32}(\mathbf{z}_i)$, for $i=m_1+m_2+1,\ldots,m_1+m_2+m_3$,
and we can find
$s_3>2$ and $C_3$
such that
\begin{multline*}
\|u\|_{L^{s_3}(A_3)}\leq C_3m_3\big[\| u\|_{L^2(B_{R_0+1}\backslash K,\mathbb{C}^3)}\\+\|\nabla\wedge u\|_{L^2(B_{R_0+1}\backslash K,\mathbb{C}^3)}+
\|\nabla\cdot(a u)\|_{L^2(B_{R_0+1}\backslash K)}\big].
\end{multline*}
Picking $p=\min\{s_1,s_2,s_3\}$ we obtain that
\begin{multline*}
\|u\|_{L^{p}(B_{R_0+1}\backslash K)}\leq C\big[\| u\|_{L^2(B_{R_0+1}\backslash K,\mathbb{C}^3)}\\+\|\nabla\wedge u\|_{L^2(B_{R_0+1}\backslash K,\mathbb{C}^3)}+
\|\nabla\cdot(a u)\|_{L^2(B_{R_0+1}\backslash K)}\big].
\end{multline*}
Our arguments clearly show that $p$ and $C$ have the dependence required.


Finally, similar reasonings easily show that
$B_{R_0+1}\backslash \Sigma$ satisfies the MCP, therefore
the proof is complete. 
\end{proof}

\section{Stability of solutions of electromagnetic scattering problems}\label{stabilitysec}

In this section we investigate the stability of solutions of Maxwell equations, 
in particular of solutions to electromagnetic scattering problems,
with respect to changes both in the exterior domain and in the coefficients.

In this section we shall keep fixed
positive constants $r$, $L$, and $R_0$, $0<r_1< r$, $\tilde{C}>0$, $0<\lambda_0<1<\lambda_1$, and $0<\underline{k}<\overline{k}$, and two  nondecreasing left-continuous functions
$\omega:(0,+\infty)\to (0,+\infty)$ and $\delta:(0,+\infty)\to (0,+\infty)$.

We begin by defining the following class of admissible coefficients.

\begin{defn}\label{coeffdefn}
We say that $(\epsilon,\mu)$ is a couple of coefficients belonging to the admissible class $\mathcal{N}=\mathcal{N}(r,L,R_0,\omega,\lambda_0,\lambda_1)$ if the following holds.

We assume that  $(\epsilon,\mu)\in L^{\infty}(\mathbb{R}^3,M^{3\times 3}_{sym}(\mathbb{R})^2)$ and satisfies \eqref{ellipticity2} with constants $\lambda_0<\lambda_1$ in $\mathbb{R}^3$.

Then we assume that there exists $K\in\tilde{\mathcal{B}}(r,L,B_{R_0},\omega)$, depending on $(\epsilon,\mu)$,
with the following properties. We call $D_0$ the unbounded connected component of $\mathbb{R}^3\backslash K$ and $D_i$, $i=1,\ldots,\tilde{M}$, the bounded connected components of $\mathbb{R}^3\backslash K$.
We finally assume that
$(\epsilon,\mu)=(I_3,I_3)$ in $D_0$ and that, for any $i=1,\ldots,\tilde{M}$,
$(\epsilon,\mu)=(\epsilon_i,\mu_i)$ where $(\epsilon_i,\mu_i)$ is a couple of Lipschitz functions from $\overline{B_{R_0+1}}$ to $M^{3\times 3}_{sym}(\mathbb{R})$ with Lipschitz constant bounded by $L$.
\end{defn}

We recall that there exists $M_1$, depending on $r$, $L$, $R_0$, and $\omega$ only, such that $\tilde{M}$ in the previous definition, that depend on $K$ thus on $(\epsilon,\mu)$, satisfies $\tilde{M}\leq M_1$.

The following lemmas justify the previous definition.

\begin{lem}\label{UCPforN}
Let $D$ be any connected open set contained in $\mathbb{R}^3$. Let $k>0$ and
$(\epsilon,\mu)\in \mathcal{N}=\mathcal{N}(r,L,R_0,\omega,\lambda_0,\lambda_1)$.

Then the Maxwell system
\begin{equation}\label{limiteq0}
\nabla\wedge \mathbf{E}-\rmi k \mu\mathbf{H}=0,\quad
\nabla\wedge \mathbf{H}+\rmi k \epsilon\mathbf{E}=0\quad\text{in }D
\end{equation}
satisfies the UCP in $D$.
\end{lem}

\begin{proof} This is an easy consequence of Proposition~\ref{piecewiseLipschitzprop}
and Lemma~\ref{disconnectionlemma}, see Remark~\ref{piecewiseLipschitzrem}.
\end{proof}

\begin{lem}\label{coeffcompactlemma}
The class $\mathcal{N}=\mathcal{N}(r,L,R_0,\omega,\lambda_0,\lambda_1)$
is compact with respect to convergence almost everywhere in $\mathbb{R}^3$, as well with respect to
the $L^p(\mathbb{R}^3,M^{3\times 3}_{sym}(\mathbb{R})^2)$ convergence, for any $p$, $1\leq p<+\infty$.
\end{lem}

\begin{proof}
Let us consider $\{(\epsilon^n,\mu^n)\}_{n\in\mathbb{N}}\subset \mathcal{N}$.
We call, for any $n\in \mathbb{N}$, $K_n$ the corresponding set belonging to $\tilde{\mathcal{B}}=\tilde{\mathcal{B}}(r,L,B_{R_0},\omega)$, $M_n$ the number of bounded connected components of $\mathbb{R}^3\backslash K_n$, and $D_0^n$ the unbounded connected component of $\mathbb{R}^3\backslash K_n$. Without loss of generality, up to passing to a subsequence that we do not relabel, we may assume that $M_n=M\leq M_1$ for
any $n\in \mathbb{N}$ and we call $D_i^n$, $i=1,\ldots,M$, the bounded connected components of $\mathbb{R}^3\backslash K_n$.
Moreover, again up to subsequences, we can assume that, as $n\to\infty$,
$K_n\to K\in \tilde{\mathcal{B}}$ in the Hausdorff distance, 
and, up to reordering, that 
$\overline{D_i^n}\to \overline{D_i}$, $i=0,1,\ldots,M$, in the Hausdorff distance. Here 
the sets $D_i$, $i=0,\ldots,M$, are open subsets of $\mathbb{R}^3\backslash K$, which are pairwise disjoint. Moreover, $D_0$ is the only unbounded one. 
Namely, each of the $D_i$ is the union of a finite number of connected components of $\mathbb{R}^3\backslash K$.

For any $i=1,\ldots,M$, $(\epsilon^n,\mu^n)|_{D_i^n}=(\epsilon_i^n,\mu_i^n)$ and 
we can also suppose that, as $n\to\infty$,  $(\epsilon_i^n,\mu_i^n)$ converges to 
$(\epsilon_i,\mu_i)$ uniformly in $\overline{B_{R_0+1}}$. Clearly, for any $i=1,\ldots,M$, 
$\epsilon_i$ and $\mu_i$ belong to $L^{\infty}(\overline{B_{R_0+1}},M^{3\times 3}_{sym}(\mathbb{R}))$ 
and are Lipschitz with Lipschitz constant bounded by $L$.

We then define $(\epsilon,\mu)\in L^{\infty}(\mathbb{R}^3,M^{3\times 3}_{sym}(\mathbb{R})^2)$ such that $(\epsilon,\mu)=(I_3,I_3)$ in $D_0$ and
$(\epsilon,\mu)=(\epsilon_i,\mu_i)$ in $D_i$ for any $i=1,\ldots,M$.

We recall that $|K|=0$. Now, let $\mathbf{x}\in \mathbb{R}^3\backslash K$. We have that
$\mathbf{x}\in D_i$ for some $i\in\{0,1,\ldots,M\}$. It is an easy remark that, for any $n$ large enough, $\mathbf{x}\in D_i^n$, therefore
$(\epsilon^n,\mu^n)(\mathbf{x})=(\epsilon^n_i,\mu^n_i)(\mathbf{x})$ converges, as $n\to\infty$, to
$(\epsilon_i,\mu_i)(\mathbf{x})=(\epsilon,\mu)(\mathbf{x})$. We conclude that
$(\epsilon^n,\mu^n)$ converges, as $n\to\infty$, to
$(\epsilon,\mu)$ almost everywhere in $\mathbb{R}^3$.
Then it is not difficult to observe that $(\epsilon,\mu)\in\mathcal{N}$, thus the compactness is proved.
By the uniform $L^{\infty}$ bound, and since any coefficient coincides with the identity matrix outside a given ball,
 we can immediately conclude the proof also for the convergence in $L^p$, $1\leq p<+\infty$.
\end{proof}

We shall also need the following strong convergence result for solutions to Maxwell systems.

\begin{lem}\label{localconvsollemma}
Let $D$ be any open set contained in $\mathbb{R}^3$.
Let, for any $n\in\mathbb{N}$, $\epsilon_n$, $\mu_n\in L^{\infty}(\mathbb{R}^3,M^{3\times 3}_{sym}(\mathbb{R}))$ satisfy \eqref{ellipticity2} with constants $\lambda_0$ and $\lambda_1$.

We assume that, for any $n\in\mathbb{N}$, $(\mathbf{E}_n,\mathbf{H}_n)\in H_{loc}(\mathrm{curl},D)$ solve
$$\nabla\wedge \mathbf{E}_n-\rmi k_n \mu_n\mathbf{H}_n=0,\quad
\nabla\wedge \mathbf{H}_n+\rmi k_n \epsilon_n\mathbf{E}_n=0\quad\text{in }D$$
for some $0<k_n\leq \overline{k}$.

Let us assume that, for some constant $C$, we have
$$\|\mathbf{E}_n\|_{L^2(D,\mathbb{C}^3)}+\|\mathbf{H}_n\|_{L^2(D,\mathbb{C}^3)}\leq C\quad\text{for any }n\in\mathbb{N}$$
and that $(\epsilon_n,\mu_n)\to (\epsilon,\mu)$ almost everywhere in $D$, as $n\to\infty$.

Then, up to a subsequence that we do not relabel, we have that, as $n\to\infty$,
$$(\mathbf{E}_n,\mathbf{H}_n)\to 
(\mathbf{E},\mathbf{H})\quad\text{strongly in }
H_{loc}(\mathrm{curl},D)$$
where $(\mathbf{E},\mathbf{H})$ solves, for some $0\leq k\leq\overline{k}$,
\begin{equation}\label{limiteq}
\nabla\wedge \mathbf{E}-\rmi k \mu\mathbf{H}=0,\quad
\nabla\wedge \mathbf{H}+\rmi k \epsilon\mathbf{E}=0\quad\text{in }D.
\end{equation}
\end{lem}

\begin{proof}
Obviously $\epsilon$, $\mu\in L^{\infty}(\mathbb{R}^3,M^{3\times 3}_{sym}(\mathbb{R}))$ and satisfy \eqref{ellipticity2} with constants $\lambda_0$ and $\lambda_1$.

By the Maxwell equations, for some constant $C_1$ we have that
$$\|\mathbf{E}_n\|_{H(\mathrm{curl},D)}+\|\mathbf{H}_n\|_{L^2(\mathrm{curl},D)}\leq C_1\quad\text{for any }n\in\mathbb{N}.$$

Therefore, we can assume, by 
passing to a subsequence that we do not relabel, that there exists $(\mathbf{E},\mathbf{H})\in H(\mathrm{curl},D)^2$ and $k$, $0\leq k\leq\overline{k}$,
such that, as $n\to\infty$, $k_n\to k$ and, in $D$, we have
$(\mathbf{E}_n,\mathbf{H}_n)\rightharpoonup
(\mathbf{E},\mathbf{H})$ weakly in $L^2$ and $(\nabla\wedge\mathbf{E}_n,\nabla\wedge\mathbf{H}_n)\rightharpoonup
(\nabla\wedge\mathbf{E},\nabla\wedge\mathbf{H})$ weakly in $L^2$. 
It easily follows that $(\mathbf{E},\mathbf{H})$ solves
\eqref{limiteq}.

The difficult part is to prove that, actually, the convergence is strong in $L^2$, at least locally. Let us the fix $D_1\subset \overline{D}_1\subset D_2\subset \overline{D}_2\subset D$, with $D_1$ and $D_2$ open sets and $D_2$ with Lipschitz boundary.
We consider an auxiliary function $\chi\in C_0^{\infty}(D_2)$ such that
$0\leq \chi\leq 1$ everywhere and $\chi\equiv 1$ in a neighbourhood of $\overline{D_1}$.

We obviously have
$$\mathbf{E}_n\in H(\mathrm{curl},D)\cap H(\mathrm{div}_{\epsilon_n},D)\quad\text{and}\quad
\mathbf{H}_n\in H(\mathrm{curl},D)\cap H(\mathrm{div}_{\mu_n},D)$$
with a corresponding norm bounded by $C$, $\lambda_0$, $\lambda_1$, and  $\overline{k}$ only.
It is not difficult to show that 
\begin{multline*}
\chi\mathbf{E}_n\in H_0(\mathrm{curl},D_2)\cap H(\mathrm{div}_{\epsilon_n},D_2)\quad\text{and}\\
\chi\mathbf{H}_n\in H_0(\mathrm{curl},D_2)\cap H(\mathrm{div}_{\mu_n},D_2)
\end{multline*}
with a corresponding norm bounded by $C$, $\lambda_0$, $\lambda_1$, $\overline{k}$, and the two sets $D_1$ and $D_2$ only.
We consider only the case of the electric fields, the one for the magnetic fields being completely analogous. We call $\psi_n=\chi\mathbf{E}_n$.
We need to investigate the properties of $\nabla\cdot (\epsilon\psi_n)$.
First of all, we notice that, by Proposition~\ref{Druprop}, there exists $q>2$ such that
$\psi_n$, $n\in\mathbb{N}$, is uniformly bounded in $L^q(D_2,\mathbb{C}^3)$.
We have that, for any $\varphi\in C^{\infty}_0(D_2)$
$$\langle \nabla\cdot (\epsilon\psi_n),\varphi\rangle_{D_2}=-\langle \epsilon\psi_n,\nabla\varphi
\rangle_{D_2}=\langle (\epsilon_n-\epsilon)\psi_n,\nabla\varphi
\rangle_{D_2}
-\langle \epsilon_n\psi_n,\nabla\varphi
\rangle_{D_2}.$$
We know that $\{\nabla\cdot(\epsilon_n\psi_n)\}_{n\in\mathbb{N}}$ is bounded in $L^2(D_2)$, therefore it is compact in
$(H^1(D_2))^{\ast}$. We also have that, in $D_2$, $(\epsilon_n-\epsilon)\to 0$ strongly in $L^p$ for
any $p$, $1\leq p<+\infty$, and $\psi_n$ is uniformly bounded in $L^q(D_2,\mathbb{C}^3)$, for some $q>2$. Hence we have that
$(\epsilon_n-\epsilon)\psi_n\to 0$ strongly in $L^2(D_2,\mathbb{C}^3)$. Therefore
$\nabla\cdot((\epsilon_n-\epsilon)\psi_n)$ converges to $0$, as $n\to\infty$, in $(H^1(D_2))^{\ast}$.
We conclude that $\{\nabla\cdot (\epsilon\psi_n)\}_{n\in\mathbb{N}}$ is compact in 
$(H^1(D_2))^{\ast}$. By \cite[Lemma~2.11]{Dru} we can conclude that
$\{\psi_n\}_{n\in\mathbb{N}}$ is compact in $L^2(D_2,\mathbb{C}^3)$.

We obtain that $\{\mathbf{E}_n\}_{n\in\mathbb{N}}$ and $\{\mathbf{H}_n\}_{n\in\mathbb{N}}$ are bounded in $L^q(D_1,\mathbb{C}^3)$, for some $q>2$, and compact in $L^2(D_1,\mathbb{C}^3)$. Since in $D_1$ they converge weakly in $L^2$ to 
$\mathbf{E}$ and $\mathbf{H}$, respectively, it is not difficult to conclude that they
actually converge strongly in $L^2$. 
Using the Maxwell equations, we also obtain that $\{\nabla\wedge\mathbf{E}_n\}_{n\in\mathbb{N}}$ and $\{\nabla\wedge\mathbf{H}_n\}_{n\in\mathbb{N}}$ are bounded in $L^q(D_1,\mathbb{C}^3)$, for some $q>2$, and compact in $L^2(D_1,\mathbb{C}^3)$. 

The proof is complete.
\end{proof}

We are in the position to prove the following general stability result and uniform bounds for the direct electromagnetic scattering problem. These are the main results of the paper. We begin with the uniform bounds.

\begin{thm}\label{mainstabthm}
Let $\hat{\mathcal{B}}_{scat}=\hat{\mathcal{B}}_{scat}(r,L,R_0,r_1,\tilde{C},\omega,\delta)$
as in Definition~\textnormal{\ref{scatdefin2}}.
Let $\mathcal{N}=\mathcal{N}(r,L,R_0,\omega,\lambda_0,\lambda_1)$ as in Definition~\textnormal{\ref{coeffdefn}}.

For any $\Sigma\in\hat{\mathcal{B}}_{scat}$, for any $(\epsilon,\mu)\in \mathcal{N}$, and for any $(\mathbf{E}^i,\mathbf{H}^i)$ as in \eqref{eq:planewave} with $\underline{k}\leq k\leq\overline{k}$,
$\|\mathbf{p}\|\leq 1$, and $\mathbf{d}\in \mathbb{S}^2$, let $(\mathbf{E},\mathbf{H})$ be the solution to \eqref{directscattering} and
$(\mathbf{E}^s,\mathbf{H}^s)$ be the corresponding scattering fields.

Then there exists a constant $C>0$, depending on $r$, $L$, $R_0$, $r_1$, $\tilde{C}$,
$\omega$, $\delta$, $\lambda_0$, $\lambda_1$, $\underline{k}$, and $\overline{k}$ only, such that
\begin{equation}\label{EMbound}
\|\mathbf{E}\|_{L^2(B_{R_0+1}\backslash \Sigma,\mathbb{C}^3)}+\|\mathbf{H}\|_{L^2(R_0+1\backslash \Sigma,\mathbb{C}^3)}\leq C.
\end{equation}

Then for any $R\geq R_0+1$ there exists a constant $E$, $E$ depending on
the constant $C$ in \eqref{EMbound}, $R_0$, $R$, and $\overline{k}$ only, such that
\begin{equation}\label{uniformbound}
\|\mathbf{E}\|_{L^2(B_R\backslash \Sigma,\mathbb{C}^3)}+\|\mathbf{H}\|_{L^2(B_R\backslash \Sigma,\mathbb{C}^3)}\leq E.
\end{equation}

Furthermore,
there exists a constant $E_1$, depending on the constant  $C$ in \eqref{EMbound}, $\overline{k}$, and $R_0$ only,
such that for any $\mathbf{x}\in\mathbb{R}^3$ we have
\begin{equation}\label{udecayestimate}
\|\mathbf{E}^s(\mathbf{x})\|+\|\mathbf{H}^s(\mathbf{x})\|\leq E_1\|\mathbf{x}\|^{-1}\quad\text{if }\|\mathbf{x}\|\geq R_0+1/2.
\end{equation}
\end{thm}

\begin{proof}
We argue by contradiction. Let us assume that there exist, for any $n\in\mathbb{N}$,
$\Sigma_n\in\hat{\mathcal{B}}_{scat}$, $(\epsilon_n,\mu_n)\in \mathcal{N}$, and $(\mathbf{E}^i_n,\mathbf{H}^i_n)$ as in \eqref{eq:planewave} with $\underline{k}\leq k_n\leq\overline{k}$,
$\|\mathbf{p}_n\|\leq 1$, and $\mathbf{d}_n\in \mathbb{S}^2$, such that $(\mathbf{E}_n,\mathbf{H}_n)$, the solution to \eqref{directscattering} with these data,
satisfies
$$
\|\mathbf{E}_n\|_{L^2(B_{R_0+1}\backslash \Sigma_n,\mathbb{C}^3)}+\|\mathbf{H}_n\|_{L^2(R_0+1\backslash \Sigma_n,\mathbb{C}^3)}=b_n\geq n.
$$
We call $(\tilde{\mathbf{E}}_n,\tilde{\mathbf{H}}_n)=(\mathbf{E}_n/b_n,\mathbf{H_n}/b_n)$, $n\in\mathbb{N}$, and, by extending them to $0$ in $\Sigma_n$, we have that
\begin{equation}\label{unbounded}
\|\tilde{\mathbf{E}}_n\|_{L^2(B_{R_0+1},\mathbb{C}^3)}+\|\tilde{\mathbf{H}}_n\|_{L^2(B_{R_0+1},\mathbb{C}^3)}=1\quad\text{for any }n\in\mathbb{N}.
\end{equation}

Up to a subsequence, we have that 
$\Sigma_n$ converges, in the Hausdorff distance, to a scatterer $\Sigma\in\hat{\mathcal{B}}_{scat}$ and such that 
$B_{R_0+1}\backslash \Sigma$ satisfies the RCP and MCP, see Lemma~\ref{compactHausdorff} and Proposition~\ref{highintprop}. Moreover, by Lemma~\ref{coeffcompactlemma}, we have that $(\epsilon_n,\mu_n)$ converges in the whole $\mathbb{R}^3$ to
$(\epsilon,\mu)\in\mathcal{N}$ in $L^p$ for any $1\leq p<+\infty$. We also assume that
$k_n\to k$, with $\underline{k}\leq k\leq \overline{k}$. In particular we notice that $k>0$. 

We also have, again up to subsequences, that
$(\tilde{\mathbf{E}}_n,\tilde{\mathbf{H}}_n)$ converges weakly in $L^2(B_{R_0+1},\mathbb{C}^6)$ to $(\tilde{\mathbf{E}},\tilde{\mathbf{H}})$.
By the first property of Mosco convergence and Theorem~\ref{Moscoconv}, we easily conclude that
$(\tilde{\mathbf{E}},\tilde{\mathbf{H}})\in H(\mathrm{curl},B_{R_0+1}\backslash \Sigma)^2$, where again these functions are extended to $0$ in $\Sigma$.

It is not difficult to show that $(\tilde{\mathbf{E}},\tilde{\mathbf{H}})$ solve \eqref{eq:Maxwell2} in $B_{R_0+1}\backslash \Sigma$. Moreover,
by standard regularity estimates in $B_{R_0+1}\backslash \overline{B_{R_0}}$,
the Stratton-Chu formulas, and the fact that $(\mathbf{E}^i_n,\mathbf{H}^i_n)$ are uniformly bounded,
we can also deduce that $(\tilde{\mathbf{E}}_n,\tilde{\mathbf{H}}_n)$ converges in $L^2_{loc}(\mathbb{R}^3\backslash \overline{B_{R_0}},\mathbb{C}^6)$ to $(\tilde{\mathbf{E}},\tilde{\mathbf{H}})$, with 
$(\tilde{\mathbf{E}},\tilde{\mathbf{H}})$ being an outgoing solution to the Maxwell equations \eqref{eq:Maxwell2bis} in $\mathbb{R}^3\backslash \overline{B_{R_0}}$.

We know that $\tilde{\mathbf{E}}_n\in H_0(\mathrm{curl}, \mathbb{R}^3\backslash \Sigma_n)$. Then, let $\phi\in H(\mathrm{curl}, \mathbb{R}^3\backslash \Sigma)$, with bounded support. By the second property of the Mosco convergence and Theorem~\ref{Moscoconv}, 
we can find $\phi_n\in H(\mathrm{curl}, \mathbb{R}^3\backslash \Sigma_n)$, with bounded support, such that, as $n\to\infty$,
$(\phi_n,\nabla\wedge\phi_n)\to (\phi,\nabla\wedge\phi)$ in $L^2(\mathbb{R}^3,\mathbb{C}^6)$, with the usual convention of extending the functions to $0$ in $\Sigma_n$ and $\Sigma$, respectively.
For any $n\in\mathbb{N}$, we have
$$\langle \nabla\wedge \tilde{\mathbf{E}}_n,\phi_n\rangle_{\mathbb{R}^3}-
\langle \nabla\wedge \phi_n,\tilde{\mathbf{E}}_n\rangle_{\mathbb{R}^3}=0.$$
On the other hand, since
$$\langle \nabla\wedge \tilde{\mathbf{E}},\phi\rangle_{\mathbb{R}^3}-
\langle \nabla\wedge \phi,\tilde{\mathbf{E}}\rangle_{\mathbb{R}^3}=\lim_n\left(\langle \nabla\wedge \tilde{\mathbf{E}}_n,\phi_n\rangle_{\mathbb{R}^3}-
\langle \nabla\wedge \phi_n,\tilde{\mathbf{E}}_n\rangle_{\mathbb{R}^3}\right),$$
we can conclude that
$\tilde{\mathbf{E}}\in H_0(\mathrm{curl},\mathbb{R}^3\backslash\Sigma)$.

Then we notice that $(\tilde{\mathbf{E}},\tilde{\mathbf{H}})$ solves the direct scattering problem \eqref{directscattering} with scatterer $\Sigma$, coefficients $\varepsilon$, $\mu$, and incident fields $(\tilde{\mathbf{E}}^i,\tilde{\mathbf{H}}^i)\equiv(0,0)$. 

By the properties of $\Sigma$ and Lemma~\ref{UCPforN},
as well as Theorem~\ref{mainthmex+uniq}, we conclude that $(\tilde{\mathbf{E}},\tilde{\mathbf{H}})\equiv (0,0)$ in $\mathbb{R}^3$.

We wish to prove that $(\tilde{\mathbf{E}}_n,\tilde{\mathbf{H}}_n)$ converges to
$(\tilde{\mathbf{E}},\tilde{\mathbf{H}})$ not only weakly in $L^2(B_{R_0+1},\mathbb{C}^6)$ but also strongly in $L^2(B_{R_0+1},\mathbb{C}^6)$. This would allow us to conclude the proof. In fact, since
$(\tilde{\mathbf{E}},\tilde{\mathbf{H}})\equiv 0$, it would follow that
$$\|\tilde{\mathbf{E}}_n\|_{L^2(B_{R_0+1},\mathbb{C}^3)}+\|\tilde{\mathbf{H}}_n\|_{L^2(B_{R_0+1},\mathbb{C}^3)}\to 0\quad\text{as }n\to\infty$$
and this contradicts \eqref{unbounded}.

To prove the strong convergence in $L^2$, we begin with the following argument.
For any $\delta>0$, there exists $\overline{n}\in\mathbb{N}$ large enough such that for any $n\geq \overline{n}$ we have
$\Sigma_n\subset B_{\delta}(\Sigma)$. Hence, by Lemma~\ref{localconvsollemma}, we deduce that $(\tilde{\mathbf{E}}_n,\tilde{\mathbf{H}}_n)$ converges to
$(\tilde{\mathbf{E}},\tilde{\mathbf{H}})$ strongly in $L^2(K,\mathbb{C}^6)$ for any compact $K\subset \mathbb{R}^3\backslash\Sigma$. We also notice that, by 
Proposition~\ref{highintprop} and
Corollary~\ref{highintcorollary},
we can find constants $p>2$ and $C_1>0$ such that
\begin{equation}\label{pint}
\|\tilde{\mathbf{E}}_n\|_{L^p(B_{R_0+1},\mathbb{C}^3)}+
\|\tilde{\mathbf{H}}_n\|_{L^p(B_{R_0+1},\mathbb{C}^3)}\leq C_1\quad\text{for any }n\in\mathbb{N}.
\end{equation}
Therefore, for any $\delta>0$,
\begin{multline*}
\|\tilde{\mathbf{E}}_n-\tilde{\mathbf{E}}\|_{L^2(B_{R_0+1},\mathbb{C}^3)}\leq
\|\tilde{\mathbf{E}}_n-\tilde{\mathbf{E}}\|_{L^2(B_{\delta}(\Sigma),\mathbb{C}^3)}+\|\tilde{\mathbf{E}}_n-\tilde{\mathbf{E}}\|_{L^2(B_{R_0+1}\backslash B_{\delta}(\Sigma),\mathbb{C}^3)}\\\leq
\|\tilde{\mathbf{E}}_n\|_{L^2(B_{\delta}(\Sigma)\backslash \Sigma_n,\mathbb{C}^3)}+
\|\tilde{\mathbf{E}}\|_{L^2(B_{\delta}(\Sigma)\backslash \Sigma,\mathbb{C}^3)}+
\|\tilde{\mathbf{E}}_n-\tilde{\mathbf{E}}\|_{L^2(B_{R_0+1}\backslash B_{\delta}(\Sigma),\mathbb{C}^3)}.
\end{multline*}
We also have that, for any $\delta>0$, there exists $\tilde{n}\geq\overline{n}$, such that
for any $n\geq \tilde{n}$ we have $B_{\delta}(\Sigma)\backslash \Sigma_n\subset
B_{\delta}(\partial\Sigma)$. Obviously, $B_{\delta}(\Sigma)\backslash \Sigma\subset
B_{\delta}(\partial\Sigma)$.
By H\"older inequality, we conclude that
\begin{multline*}
\|\tilde{\mathbf{E}}_n-\tilde{\mathbf{E}}\|_{L^2(B_{R_0+1},\mathbb{C}^3)}\\\leq
C_1|B_{\delta}(\partial\Sigma)|^{(p-2)/p}
+\|\tilde{\mathbf{E}}\|_{L^2(B_{\delta}(\partial\Sigma),\mathbb{C}^3)}+
\|\tilde{\mathbf{E}}_n-\tilde{\mathbf{E}}\|_{L^2(B_{R_0+1}\backslash B_{\delta}(\Sigma),\mathbb{C}^3)},
\end{multline*}
with $p$ and $C_1$ as in \eqref{pint}
Since $|\partial\Sigma|=0$, for any $\varepsilon>0$ there exists $\delta>0$ such that
$$C_1|B_{\delta}(\partial\Sigma)|^{(p-2)/p}+\|\tilde{\mathbf{E}}\|_{L^2(B_{\delta}(\partial\Sigma),\mathbb{C}^3)}<\varepsilon/2.$$
Fixed such a $\delta$, we can find
$\hat{n}\geq\tilde{n}$ such that, for any $n\geq\hat{n}$,
$$\|\tilde{\mathbf{E}}_n-\tilde{\mathbf{E}}\|_{L^2(B_{R_0+1}\backslash B_{\delta}(\Sigma),\mathbb{C}^3)}\leq \varepsilon/2.$$
Therefore, we obtain that $\tilde{\mathbf{E}}_n\to \tilde{\mathbf{E}}$ in $L^2(B_{R_0+1},\mathbb{C}^3)$ as $n\to\infty$. The same reasoning applies to $\tilde{\mathbf{H}}_n$, thus
the proof of \eqref{EMbound} is concluded.

From \eqref{EMbound},
the estimate \eqref{uniformbound} and
the uniform decay estimate \eqref{udecayestimate} easily follow.
\end{proof}

We conclude this section with the general stability result.

\begin{thm}\label{continuitycor}
Under the assumptions of Theorem~\textnormal{\ref{mainstabthm}},
let us consider, for any $n\in\mathbb{N}$,
$\Sigma_n\in\hat{\mathcal{B}}_{scat}$, $(\epsilon_n,\mu_n)\in \mathcal{N}$, and $(\mathbf{E}^i_n,\mathbf{H}^i_n)$ as in \eqref{eq:planewave} with $\underline{k}\leq k_n\leq\overline{k}$,
$\|\mathbf{p}_n\|\leq 1$, and $\mathbf{d}_n\in \mathbb{S}^2$.
We call $(\mathbf{E}_n,\mathbf{H}_n)$ the solution to \eqref{directscattering} with these data and we extend them, as well as their curls, to $0$ in $\Sigma_n$.

We assume that, as $n\to\infty$,
$\Sigma_n$ converges, in the Hausdorff distance, to a scatterer $\Sigma\in\hat{\mathcal{B}}_{scat}$, and $(\epsilon_n,\mu_n)$ converges in the whole $\mathbb{R}^3$ to
$(\epsilon,\mu)\in\mathcal{N}$ in $L^p$ for any $1\leq p<+\infty$, and 
$k_n\to k$, with $\underline{k}\leq k\leq \overline{k}$, and 
$\mathbf{p}_n\to\mathbf{p}$, with $\|\mathbf{p}\|\leq 1$, and
$\mathbf{d}_n\to\mathbf{d}\in\mathbb{S}^2$. We notice that, under our hypotheses, this is always true up to passing to subsequences.
We set $(\mathbf{E}^i,\mathbf{H}^i)$ as in \eqref{eq:planewave} with $k$,
$\mathbf{p}$, and $\mathbf{d}$, and we call $(\mathbf{E},\mathbf{H})$ the solution to \eqref{directscattering} and we extend them, as well as their curls, to $0$ in $\Sigma$.

Then we obtain that, as $n\to\infty$,
$(\mathbf{E}_n,\mathbf{H}_n)$ converges to $(\mathbf{E},\mathbf{H})$, and
also $(\nabla\wedge\mathbf{E}_n,\nabla\wedge\mathbf{H}_n)$ converges to $(\nabla\wedge\mathbf{E},\nabla\wedge\mathbf{H})$,
in $L^2(B_R,\mathbb{C}^6)$ for any positive $R$, thus in particular $(\mathbf{E}_n,\mathbf{H}_n)$ converges to $(\mathbf{E},\mathbf{H})$
in $H(\mathrm{curl},D)$ for any bounded open set $D$ compactly contained in $G=\mathbb{R}^3\backslash\Sigma$.
\end{thm}

\begin{proof} 
First of all we notice that, as $n\to\infty$, $(\mathbf{E}^i_n,\mathbf{H}^i_n)$ converges to $(\mathbf{E}^i,\mathbf{H}^i)$ in $H(\mathrm{curl},B_R)$ for any $R>0$.
Using the uniform bound \eqref{EMbound} of Theorem~\ref{mainstabthm}, we can obtain this continuity result by easily adapting the arguments developed in the proof of 
Theorem~\ref{mainstabthm} to study the convergence properties of the sequence
 $\{(\tilde{\mathbf{E}}_n,\tilde{\mathbf{H}}_n)\}_{n\in\mathbb{N}}$.
 \end{proof}

\section{Application to the inverse scattering problem for polyhedral scatterers}\label{polystabsec}

In this section we fix positive constants $r$, $L$, and $R_0$, $0<r_1< r$ and $\tilde{C}>0$.
Let us also fix $\omega:(0,+\infty)\to (0,+\infty)$ and $\delta:(0,+\infty)\to (0,+\infty)$ two  nondecreasing left-continuous functions.
We recall that $\omega(t)\leq t$, that
$\lim_{t\to+\infty}\omega(t)$ is equal to a finite real number which we call $\omega(+\infty)$,
and that $\delta(t)\leq t$ for any $t>0$. We fix a wavenumber $k>0$.
Finally, we fix positive $R_1$ and $\tilde{\rho}$ such that
$R_0+1+\tilde{\rho}\leq R_1$.
We refer to these constants and functions as the \emph{a priori data}.

We introduce suitable classes of \emph{polyhedral scatterers} in $\mathbb{R}^3$.
We define a \emph{cell} as the closure of an open subset of a
plane.
A scatterer $\Sigma$ is \emph{polyhedral} if the boundary of $\Sigma$ is given by a finite union of cells
$\mathcal{C}_j$, $j=1,\ldots,M_1$.

Fixed positive constants $h$ and $L$,
we say that a scatterer $\Sigma$ is \emph{polyhedral with constants} $h$ \emph{and} $L$ if
the boundary of $\Sigma$ is given by a finite union of cells
$\mathcal{C}_j$, $j=1,\ldots,M_1$, where
 each $\mathcal{C}_j$ is the closure of a Lipschitz domain with constants $h$ and $L$ contained in a plane and the cells are pairwise internally disjoint, that is two different cells may intersect only at boundary points.

Let $\hat{\mathcal{B}}_{scat}=\hat{\mathcal{B}}_{scat}(r,L,R_0,r_1,\tilde{C},\omega,\delta)$ be the class of scatterers defined in Definition~\ref{scatdefin2}.
We call $\hat{\mathcal{B}}^h_{scat}=\hat{\mathcal{B}}^h_{scat}(r,L,R_0,r_1,\tilde{C},\omega,\delta)$, for a given size parameter $h>0$, the set of scatterers $\Sigma\in \hat{\mathcal{B}}_{scat}$ such that $\Sigma$ is polyhedral with constants $h$ and $L$.

Analogously, let
$\hat{\mathcal{D}}_{obst}=\hat{\mathcal{D}}_{obst}(r,L,R_0)$ be the class of obstacles also defined in Definition~\ref{scatdefin2}.
Fixed the size parameter $h>0$, we call
$\hat{\mathcal{D}}_{obst}^h=\hat{\mathcal{D}}_{obst}^h(r,L,R_0)$ the set of obstacles $\Sigma\in\hat{\mathcal{D}}_{obst}$ such that $\Sigma$ is polyhedral with constants $h$ and $L$. Notice that in this case any $\Sigma\in \hat{\mathcal{D}}^{h}_{obst}$ is formed by a finite number of polyhedra. In this case, we can drop from the set of a priori data $r_1$, $\tilde{C}$,
$\omega$, and $\delta$.

We call $\eta:(0,1/\rme)\to (0,+\infty)$ the following function
\begin{equation}\label{contmodul}
\eta(s)=\exp(-(\log(-\log s))^{1/2})\quad\text{for any }s,\ 0<s< 1/\rme.
\end{equation}

We consider two different incident
incident fields $(\mathbf{E}^i_j,\mathbf{H}^i_j)$ $j=1,2$, given by 
normalised electromagnetic plane waves with 
incident directions $\mathbf{d}_j\in\mathbb{S}^2$ and polarisation vectors $\mathbf{p}_j\in\mathbb R^3$, $j=1,2$, respectively. We assume that, for any
$j=1,2$, $\|\mathbf{p}_j\|\leq 1$ and that the two vectors
$(\mathbf{d}_j\wedge \mathbf{p}_j)\wedge \mathbf{d}_j$ are linear independent. For example, this is true if
$\mathbf{d}_1=\mathbf{d}_2$ and the three vectors 
$\mathbf{d}_1$, $\mathbf{p}_1$ and $\mathbf{p}_2$ are linearly independent. In order to have a quantitative version of these properties, we call 
\begin{equation}\label{pol+dirprop}
b_j=\|(\mathbf{d}_j\wedge \mathbf{p}_j)\wedge \mathbf{d}_j\|>0,\quad j=1,2
\end{equation}
and
\begin{equation}\label{independency}
b_0=\min_{ \nu \in \mathbb{S}^2} \left\{\max_{j \in \{1,2\}} \|
\nu\wedge [(\mathbf{d}_j\wedge \mathbf{p}_j)\wedge \mathbf{d}_j]\|
\right\}>0.
\end{equation}
We notice that $b_0>0$ since 
$\max_{j \in \{1,2\}} \|
\nu\wedge [(\mathbf{d}_j\wedge \mathbf{p}_j)\wedge \mathbf{d}_j]\|$
is a continuous function of $\nu\in \mathbb{S}^2$
which never vanishes.

We assume that the medium is homogeneous and isotropic outside the scatterer, that is we assume $\epsilon=\mu=I_3$ everywhere.
We also fix a point $\mathbf{x}_0\in\mathbb{R}^3$ such that
$R_0+1+\tilde{\rho}\leq \|\mathbf{x}_0\|\leq R_1$.

We recall that for any two scatterers $\Sigma$ and $\Sigma'$ belonging to $\hat{\mathcal{B}}_{scat}$ or to $\hat{\mathcal{D}}_{obst}$,
we measure their distance by one of the following quantities
\begin{equation}\label{ddefin}
d=\max\left\{\sup_{x\in\partial \Sigma\backslash\Sigma'}\mathrm{dist}(x,\partial\Sigma'),
\sup_{x\in\partial \Sigma'\backslash\Sigma}\mathrm{dist}(x,\partial\Sigma)\right\}
\end{equation}
and
\begin{equation}\label{dddefin}
\hat{d}=d_H(\partial\Sigma,\partial\Sigma')\quad\text{and}\quad\tilde{d}=d_H(\Sigma,\Sigma').
\end{equation}
Here $d_H$ denotes the Hausdorff distance. In \cite[Section~2]{LPRX} there is a detailed analysis of the relationships between these quantities. In particular, for a positive constant $C_1$ depending on $\hat{\mathcal{B}}_{scat}$ only, we have
\begin{equation}\label{distancesopposite}
C_1 d\leq C_1\hat{d}\leq \tilde{d}\leq \delta^{-1}(d)\leq \delta^{-1}(\hat{d}).
\end{equation}
where $\delta^{-1}:(0,+\infty)\to(0,+\infty)$ is a nondecreasing right-continuous function defined in the following way
\begin{equation}\label{delta-1def}
\delta^{-1}(t)=\min\{\sup\{s:\ \delta(s)\leq t\},2R_0\} \quad\text{for any }t>0.
\end{equation}
Let us also notice that in the case of the class $\hat{\mathcal{D}}_{obst}$, 
by \cite[Corollary~2.4]{LPRX}, we can replace
\eqref{distancesopposite} by
\begin{equation}\label{distancesoppositebis}
C_1 d\leq C_1\hat{d}\leq \tilde{d}\leq C_2d\leq C_2\hat{d},
\end{equation}
with $C_1$ and $C_2$ depending on the class $\hat{\mathcal{D}}_{obst}$ only.

The estimates obtained in Theorem~\ref{mainstabthm}, in particular the uniform bound 
\eqref{uniformbound} for $R=R_0+3$ and 
and the uniform decay \eqref{udecayestimate},
are the crucial preliminary results that are required to extend the stability results obtained in the acoustic case, in \cite{Ron2} for sound-soft scatterers and in \cite{LPRX} for sound-hard scatterers, to the electromagnetic case.

Let $\Sigma$, $\Sigma'\in \hat{\mathcal{B}}_{scat}$ be two scatterers and $k>0$. We recall that $\epsilon$ and $\mu$ are identically equal to $I_3$.
Given an incident field $(\mathbf{E}^i,\mathbf{H}^i)$, a normalised electromagnetic plane wave with incident direction $\mathbf{d}\in \mathbb{S}^2$ and polarisation $\mathbf{p}$, with $0<\|\mathbf{p}\|\leq 1$, we call $(\mathbf{E},\mathbf{H})$ the solution to the direct scattering problem \eqref{directscattering}, $(\mathbf{E}^s,\mathbf{H}^s)$ the corresponding scattering fields, and $(\mathbf{E}_{\infty},\mathbf{H}_{\infty})$ their  far-field patterns. We call $(\mathbf{E}',\mathbf{H}')$
the solution to \eqref{directscattering} with $\Sigma$ replaced by $\Sigma'$, and analogously we denote
$((\mathbf{E}^s)',(\mathbf{H}^s)')$ the corresponding scattering fields, and $(\mathbf{E}'_{\infty},\mathbf{H}'_{\infty})$ their  far-field patterns.

About the measurements to be performed for our inverse problem, there are several possibilities.
In our stability results, we use what we refer to as the \emph{near-field error with limited aperture} given by 
\begin{equation}\label{errornear}
\|\mathbf{E}-\mathbf{E}'\|_{L^2(B_{\tilde{\rho}}(\mathbf{x}_0),\mathbb{C}^3)}\leq \varepsilon.
\end{equation}
It is also possible to consider the so-called \emph{far-field error} which is the one usually employed in scattering applications and that is defined as
\begin{equation}\label{errorfar}
\|\mathbf{E}_{\infty}-\mathbf{E}'_{\infty}\|_{L^2(\mathbb{S}^{N-1},\mathbb{C}^3)}\leq \varepsilon_0.
\end{equation}

We recall that 
there exist positive constants $\tilde{\varepsilon}_0< 1/\rme$ and $C_1$, depending on 
$E$ as in \eqref{uniformbound} for $R=R_0+3$, $R_0$, $\tilde{\rho}$, $R_1$, and $k$ only,
such that if $0<\varepsilon_0\leq\tilde{\varepsilon}_0$ then
\begin{equation}\label{fartonear2}
\varepsilon\leq \eta_1(\varepsilon_0)=\exp\left(-C_1(-\log\varepsilon_0)^{1/2}\right).
\end{equation}

\begin{rem}\label{errormagnetic}
We wish to notice here that, without any loss of generality, we may also consider 
errors on the magnetic fields, that is define 
\begin{equation}\label{errornear3}
\|\mathbf{H}-\mathbf{H}'\|_{L^2(B_{\tilde{\rho}}(\mathbf{x}_0),\mathbb{C}^3)}\leq \varepsilon.
\end{equation}
and
\begin{equation}\label{errorfar3}
\|\mathbf{H}_{\infty}-\mathbf{H}'_{\infty}\|_{L^2(\mathbb{S}^{N-1},\mathbb{C}^3)}\leq \varepsilon_0.
\end{equation}
Clearly \eqref{fartonear2} holds true in this case as well. More importantly,  all the stability results stated in this section still hold if we replace the errors related to the electric fields with the ones related to the magnetic fields.
\end{rem}

\subsection{Statement of the stability results}\label{finalsub0}

The following stability result holds for the determination of polyhedral scatterers
by two suitable electromagnetic measurements.

\begin{thm}\label{mainteoN}
Fix $h>0$.
Let $\Sigma$, $\Sigma'$ belong to $\hat{\mathcal{B}}^h_{scat}$ and let $d$ be defined as in \eqref{ddefin}.
For any $j=1,2$, let $(\mathbf{E}^i_j,\mathbf{H}^i_j)$ be
normalised electromagnetic plane waves with 
incident directions $\mathbf{d}_j\in\mathbb{S}^2$ and polarisation vectors $\mathbf{p}_j\in\mathbb R^3$, with $0<\|\mathbf{p}_j\|\leq 1$ and such that the two vectors
$(\mathbf{d}_j\wedge \mathbf{p}_j)\wedge \mathbf{d}_j$, $j=1,2$, are linear independent. Let $b_0$ be defined as in \eqref{independency}.

For any $j=1,2$, let
$(\mathbf{E}_j,\mathbf{H}_j)$ be the solution to
\begin{equation}\label{uscateq}
\left\{\begin{array}{ll}
\nabla\wedge \mathbf{E}_j-\rmi k\ \mathbf{H}_j=0
&\text{in }G=\mathbb{R}^3\backslash \Sigma\\
\nabla\wedge \mathbf{H}_j+\rmi k \mathbf{E}_j=0
&\text{in }G\\
(\mathbf{E}_j,\mathbf{H}_j)=(\mathbf{E}^i_j,\mathbf{H}^i_j)+(\mathbf{E}^s_j,\mathbf{H}^s_j)
&\text{in }G\\
\nu\wedge \mathbf{E}_j=0&\text{on }\partial G\\
\lim_{r\to+\infty}r\left(\frac{\mathbf{x}}{\|\mathbf{x}\|}\wedge \mathbf{H}^s_j(\mathbf{x})+
\mathbf{E}^s_j(\mathbf{x})\right)=0 & r=\|\mathbf{x}\|.
\end{array}
\right.
\end{equation}
and $(\mathbf{E}'_j,\mathbf{H}'_j)$ be the solution to the same problem with $\Sigma$ replaced by $\Sigma'$.

If
\begin{equation}\label{error2}
\max_{j=1,2}\|\mathbf{E}_j-\mathbf{E}'_j\|_{L^2(B_{\tilde{\rho}}(\mathbf{x}_0),\mathbb{C}^3)}\leq\varepsilon
\end{equation}
for some $\varepsilon\leq 1/(2\rme)$,
then for some positive constant $C$ depending on the a priori data and on $b_0$ only, and not on $h$, we have
\begin{equation}\label{penultima}
\min\{d,h\}\leq 2\rme R_0(\eta(\varepsilon))^C.
\end{equation}
Therefore,
\begin{equation}\label{ultima1}
d\leq 2\rme R_0(\eta(\varepsilon))^C,
\end{equation}
provided $\varepsilon\leq \hat{\varepsilon}(h)$ where
\begin{equation}\label{epsilon0defin}
\hat{\varepsilon}(h)=
\min\bigg\{1/(2\rme),\eta^{-1}\bigg(\Big(\frac{h}{2\rme R_0}\Big)^{1/C}\bigg)\bigg\}.
\end{equation}
\end{thm}

If we limit ourselves to polyhedral obstacles, that is to polyhedra, we can reduce the number of electromagnetic measurements to one and have the following stability result.

\begin{thm}\label{mainteo1}
Fix $h>0$.
Let $\Sigma$, $\Sigma'$ belong to $\hat{\mathcal{D}}^h_{obst}$ and let $d$ be defined as in \eqref{ddefin}.
Let $(\mathbf{E}^i_1,\mathbf{H}^i_1)$ be
the normalised electromagnetic plane wave with 
incident direction $\mathbf{d}_1\in\mathbb{S}^2$ and polarisation vector $\mathbf{p}_1\in\mathbb R^3$, with $0<\|\mathbf{p}_1\|\leq 1$ and such that \eqref{pol+dirprop}
holds for some positive constant $b_1.$

Let $(\mathbf{E}_1,\mathbf{H}_1)$ be the solution to
\eqref{uscateq} with $j=1$ and $(\mathbf{E}'_1,\mathbf{H}'_1)$ be the solution to the same problem with $\Sigma$ replaced by $\Sigma'$.

There exists a constant $\hat{\varepsilon}_1(h)$,
$0<\hat{\varepsilon}_1(h)
\leq 1/(2\rme)$,
 depending on the a priori data, on $b_1$,
and on $h$ only,
such that if
\begin{equation}\label{errorN1}
\|\mathbf{E}_1-\mathbf{E}'_1\|_{L^2(B_{\tilde{\rho}}(\mathbf{x}_0),\mathbb{C}^3)}\leq\varepsilon
\end{equation}
for some $\varepsilon\leq\hat{\varepsilon}_1(h)$,
then for some positive constants $A_1$, depending on the a priori data only, and $C_1$, depending on the a priori data, on $b_1$, and on $h$ only, we have
\begin{equation}\label{penultima111}
d\leq A_1(\eta(\varepsilon))^{C_1}.
\end{equation}
\end{thm}

\begin{rem}\label{diffmeasrem}
By the arguments recalled at the beginning of this section, it is an easy task to rephrase the stability estimates of Theorems~\textnormal{\ref{mainteoN}} and \textnormal{\ref{mainteo1}}
if we measure the distance between $\Sigma$ and $\Sigma'$ with 
$\tilde{d}=d_H(\Sigma,\Sigma')$ or $\hat{d}=d_H(\partial\Sigma,\partial\Sigma')$ instead of $d$ or if we replace the near-field error with limited aperture $\varepsilon$ with the far-field error $\varepsilon_0$ on the corresponding solutions. In the latter case, it is enough to replace $\varepsilon$ with $\eta_1(\varepsilon_0)$, $\eta_1$ as in
\eqref{fartonear2}, and observe that we may choose $\tilde{\rho}$ and $R_1$ as depending on the other a priori data.
\end{rem}

\subsection{Remarks  and comments on the proofs}\label{finalsub}

The following auxiliary propositions are needed. Since they hold in a general case, we state them for any $N\geq 2$. We recall that we always drop the dependence of constants on the dimension $N$.

The following three-spheres inequality holds true and is a consequence of results proved in \cite{Bru}.

\begin{prop}\label{3spheresprop}
There exist positive constants $\tilde{\rho}_0$, $C$, and $c_1$, $0<c_1<1$, depending on $k>0$ only, such that for every $0<\rho_1<\rho<\rho_2\leq\tilde{\rho}_0$
and any function $u$ such that 
$$\Delta u+k^2u=0\quad\text{in }B_{\rho_2}\subset\mathbb{R}^N,$$
we have
\begin{equation}\label{3spheres}
\|u\|_{L^2(B_{\rho})}\leq C\|u\|_{L^2(B_{\rho_2})}^{1-\beta}\|u\|_{L^2(B_{\rho_1})}^{\beta},
\end{equation}
for some $\beta$ such that
\begin{equation}\label{betadefin}
c_1\frac{\log(\rho_2/\rho)}{\log(\rho_2/\rho_1)}  \leq\beta\leq 1-c_1\frac{\log(\rho/\rho_1)}{\log(\rho_2/\rho_1)}.
\end{equation}
\end{prop}

The following $L^{\infty}$-$L^2$ estimate is a consequence of classical regularity estimates and a simple dilation argument.

\begin{lem}\label{LinftyL2}
Let us fix a positive constant $\rho_1$. Let us consider $\rho$,
$0<\rho\leq\rho_1$ and a function $u$ such that
$$\Delta u+k^2 u=0\quad\text{in }B_{\rho}\subset \mathbb{R}^N.$$
Then, for any constant $s$, $0<s<1$, there exists a constant $C$, depending on $k$, $\rho_1$, and $s$ only, such that
\begin{equation}
\label{LinftyL2est}
\rho^{N/2}\|u\|_{L^{\infty}(B_{s\rho})}\leq C\|u\|_{L^2(B_{\rho})}.
\end{equation}
\end{lem}

The proofs of Theorems~\textnormal{\ref{mainteoN}} and \textnormal{\ref{mainteo1}}
follow the arguments developed for the proofs of Theorem~3.1 and Theorem~3.4 in \cite{LPRX}, respectively. We just point out the few differences and leave all the other details to the reader. One difference is that we have an $L^2$ a priori bound, 
a consequence of Theorem~\ref{mainstabthm},
instead of an $L^{\infty}$ a priori bound. However, we can exploit the $L^2$ three spheres inequality recalled in Proposition~\ref{3spheresprop}.
For arguments concerning reflections in a plane, we can 
use Lemma~\ref{reflectionlemma} and the estimate given in Lemma~\ref{reflectionestimatelem}. Finally, $L^2$ estimates provide $L^{\infty}$ estimates by using Lemma~\ref{LinftyL2}.

Then we can conclude the proofs of our stability theorems in the following way.

\begin{proof}[Proof of Theorem~\textnormal{\ref{mainteoN}}]
We fix a constant $R_2\geq \max\{2R_1,4R_0\}$, depending on the a priori data and on $b_0$ only, such that 
$$E_1R_2^{-1}\leq kb_0/2$$
where $E_1$ is as in \eqref{udecayestimate}.

By the arguments used to prove \cite[Theorem~3.1]{LPRX}, and in particular Lemma~4.4 in the same paper, with the slight modification pointed out above, we can find a point
$\mathbf{z}$, with $\|\mathbf{z}\|\geq R_2+1$, and a unit vector $\nu$ such that
\begin{equation}\label{verycrucial}
h^{3/2}\|\nu\wedge \mathbf{E}_j(\mathbf{z})\|\leq C_0\varepsilon_2,\quad j=1,2,
\end{equation}
with $C_0$ and $\varepsilon_2$ as in \cite[Lemma~4.4]{LPRX}. Here the term $h^{3/2}$ comes from the application of Lemma~\ref{LinftyL2}.

We recall that for any $\mathbf{x}\in\mathbb{R}^3$ and any $j=1,2$ we have
$$
\nu \wedge \mathbf{E}^i_j(\mathbf{x})=\rmi k\rme^{\rmi k\mathbf{x}\cdot \mathbf{d}_j}\left(\nu \wedge [(\mathbf{d}_j \wedge \mathbf{p}_j) \wedge \mathbf{d}_j]\right).
$$
Therefore there exists $j\in\{1,2\}$ such that for any $\mathbf{x}\in\mathbb{R}^3$
\begin{equation}\label{lowerbound}
\|\nu \wedge \mathbf{E}^i_j(\mathbf{x})\|\geq kb_0.
\end{equation}
Since for any $j=1,2$ we have
$$\|\nu\wedge \mathbf{E}^s_j(\mathbf{z})\|\leq kb_0/2,$$
by our definition of $R_2$, we conclude that
there exists $j\in\{1,2\}$ such that
\begin{equation}
kb_0/2\leq \|\nu\wedge \mathbf{E}_j(\mathbf{z})\|\leq C_0h^{-3/2}\varepsilon_2.
\end{equation}
We can then conclude as in the proof of \cite[Theorem~3.1]{LPRX}.
\end{proof}

We conclude by showing the final argument for the proof of the stability result for the determination of polyhedra by a single electromagnetic measurement. As in the sound-hard acoustic case, the single measurement for polyhedra requires a completely nontrivial analysis. We also point out that, contrary to the $2$ measurements case, the dependence on the size constant $h$ is no longer explicit.

\begin{proof}[Proof of Theorem~\textnormal{\ref{mainteo1}}]
First of all we notice that 
there exists a constant $\tilde{b}_0>0$, depending on the a priori data only, such that for any $\omega\in\mathbb{S}^2$ and any polyhedral obstacle $\Sigma\in\hat{\mathcal{D}}_{obst}$, we can find a cell $\mathcal{C}$ in $\partial \Sigma$, with unit normal $\nu$, such that
$\|\nu\wedge \omega\| \geq \tilde{b}_0$.

To prove such a claim we argue by contradiction. Assume that there exist 
polyhedral obstacles
$\Sigma_n\in\hat{\mathcal{D}}_{obst}$ and $\omega_n\in\mathbb{S}^2$, $n\in\mathbb{N}$, such that, for almost any $\mathbf{x}\in\partial \Sigma_n$, we have $\|\nu_{n}(\mathbf{x})\wedge \omega_n\|<1/n$, where $\nu_n$ denotes the unit normal of $\partial \Sigma_n$. Notice that, up to a subsequence, $\omega_n$ converges to $\omega\in\mathbb{S}^2$ and $\Sigma_n$ converges to $\Sigma\in\hat{\mathcal{D}}_{obst}$ in the Hausdorff distance. Then one obtains that $\|\nu(\mathbf{x})\wedge\omega\|=0$ for almost any $\mathbf{x}\in\partial \Sigma$, $\nu$ being the unit normal of $\partial \Sigma$. This contradicts the fact that $\Sigma$ is a solid obstacle.

Then Theorem~\ref{mainteo1} can be proved using the argument of the proof
of \cite[Theorem~3.4]{LPRX} with the
same kind of modification needed in the proof of
Theorem~\ref{mainteoN}, that is detailed above, and the following remark.

As in the acoustic case, the issue is the following.
Let $\Pi$ be a plane and $\nu$ be its normal.
Then we consider
\begin{equation}\label{eq:5}
\max_{\mathbf{z}\in(\overline{B}_{2R_2+3}\setminus B_{2R_2+2})\cap \Pi} \|\nu \wedge \mathbf{E}_1(\mathbf{z})\|.
\end{equation}
If for one of these $\mathbf{z}$ we have an estimate like in \eqref{verycrucial}, in order to conclude we need a corresponding lower bound as in \eqref{lowerbound}. However, it
might happen that the quantity in \eqref{eq:5} is actually $0$. On the other hand, this happens only on special symmetric cases. Namely, we need that
$\nu \wedge [(\mathbf{d}_1 \wedge \mathbf{p}_1) \wedge \mathbf{d}_1]=0$
and that $\Sigma$ is symmetric with respect to $\Pi$. This can be proved by a simple reflection argument and by using the unique determination of polyhedra with a single scattering electromagnetic measurement proved in \cite{Liu}.
Finally, also a continuous dependence of the quantity in \eqref{eq:5} from $\Sigma$ and $\Pi$ would be needed, but this follows quite easily by Theorem~\ref{continuitycor}.
\end{proof}

\appendix

\section*{Appendix}

We conclude by proving Propositions~\ref{Linftydensity}
and \ref{suffMCP}.

\begin{proof}[Proof of Proposition~\textnormal{\ref{Linftydensity}}]
For any $M>0$ we define a truncation operation $F_M:\mathbb{R}^3\to\mathbb{R}^3$ as follows
$$
F_M(\mathbf{x})=\left\{\begin{array}{ll}
M\mathbf{x}/\|\mathbf{x}\| &\text{if }\|\mathbf{x}\|\geq M\\
\mathbf{x}&\text{if }\|\mathbf{x}\|\leq M.
\end{array}\right.$$

By the general chain rule for vector valued functions proved in \cite{Amb-DM},
for any $M>0$ and any $v\in W^{1,1}(D,\mathbb{R}^3)$, such that
$|\{\mathbf{x}\in D:\ \|v(\mathbf{x})\|=M\}|=0$, we obtain that $F_M\circ v\in W^{1,1}(D,\mathbb{R}^3)$ and, for almost every $\mathbf{x}\in D$, we have
$$
\nabla (F_M\circ v)(\mathbf{x})=\left\{\begin{array}{ll}
0 &\text{if }\|v(\mathbf{x})\|> M\\
\nabla v(\mathbf{x})&\text{if }\|v(\mathbf{x})\|< M.
\end{array}\right.$$
In particular,
for almost every $\mathbf{x}\in D$, we have
\begin{equation}\label{curltrunc}\tag{A.1}
\nabla \wedge (F_M\circ v)(\mathbf{x})=\left\{\begin{array}{ll}
0 &\text{if }\|v(\mathbf{x})\|> M\\
(\nabla\wedge v)(\mathbf{x})&\text{if }\|v(\mathbf{x})\|< M.
\end{array}\right.
\end{equation}

For any $w\in L^2(D,\mathbb{R}^3)$, we define
$$A(w)=\left\{t\geq 1:|\{\mathbf{x}\in D:\ \|v(\mathbf{x})\|=t\}|>0\right\}$$
and notice that $A(w)$ is a countable subset of $\mathbb{R}$.

Let $D=\bigcup_{i\in\mathbb{N}} B_i$ where $B_i$ is an open ball compactly contained in $D$ for any $i\in\mathbb{N}$.

Let $u\in H(\mathrm{curl},D)$. We limit ourselves for simplicity, and without loss of generality, to prove the result for $u$ which has values in $\mathbb{R}^3$ instead of $\mathbb{C}^3$.

For any $i\in\mathbb{N}$, there exists $\{u^i_n\}_{n\in\mathbb{N}}\subset C^{\infty}_0(\mathbb{R}^3)$ such that the following properties are satisfied. First, as $n\to\infty$,
$u^i_n\to u$ and $(\nabla\wedge u^i_n)\to (\nabla\wedge u)$ in $L^2(B_i,\mathbb{R}^3)$ and almost everywhere in $B_i$. Second, there exists $h^i\in L^2(B_i)$ such that
$\|u_n^i\|+\|\nabla\wedge u_n^i\|\leq h^i$ almost everywhere in $B_i$ for any $n\in\mathbb{N}$.

We denote 
$$A=A(u)\cup\left(\bigcup_{i,n\in\mathbb{N}}A(u^i_n)\right),
$$
which is clearly a countable subset of $\mathbb{R}$. We pick a sequence
$\{t_m\}_{m\in\mathbb{N}}\subset \mathbb{R}$ such that, for any $m\in \mathbb{N}$, $1\leq t_m<t_{m+1}$, $t_m\not\in A$,
 and $\lim_{m\to\infty}t_m=+\infty$.

For any $m\in\mathbb{N}$, we call
$$
V_m(\mathbf{x})=\left\{\begin{array}{ll}
0 &\text{if }\|u(\mathbf{x})\|> t_m\\
(\nabla\wedge u)(\mathbf{x})&\text{if }\|u(\mathbf{x})\|< t_m.
\end{array}\right.$$

We fix $i\in\mathbb{N}$ and $m\in\mathbb{N}$. Then, for almost every $\mathbf{x}\in B_i$, we have
$$\|F_{t_m}\circ u^i_n-F_{t_m}\circ u\|(\mathbf{x})\leq h^i(\mathbf{x})+\|u\|(\mathbf{x})\quad\text{for any }n\in\mathbb{N}.$$
As $n\to\infty$, since $u_n^i(\mathbf{x})\to u(\mathbf{x})$, we also have
$(F_{t_m}\circ u^i_n)(\mathbf{x})\to(F_{t_m}\circ u)(\mathbf{x})$.
By Lebesgue theorem, we obtain that $(F_{t_m}\circ u^i_n)(\mathbf{x})\to(F_{t_m}\circ u)(\mathbf{x})$ in $L^2(B_i)$ as $n\to\infty$.

Since, as $n\to\infty$, $u^i_n$ converges to $u$ and $\nabla\wedge u^i_n$ converges to $\nabla\wedge u$ almost everywhere in $B_i$ and $|\{\mathbf{x}\in D:\ \|u(\mathbf{x})\|=t_m\}|=0$, by \eqref{curltrunc} we can immediately infer that
$\nabla\wedge (F_{t_m}\circ u^i_n)$ converges to $V_m$ almost everywhere in $B_i$ and, again by Lebesgue theorem, also in $L^2(B_i)$.

Therefore, $F_{t_m}\circ u\in H(\mathrm{curl},B_i)$, with $\nabla\wedge(F_{t_m}\circ u)=V_m|_{B_i} $ in the sense of distributions in $B_i$, for any $i\in\mathbb{N}$.

Now let $\varphi\in C^{\infty}_0(D,\mathbb{R}^3)$ and let $K\subset D$ be compact such that the support of $\varphi$ is contained in $K$. Up to changing the order, we can suppose that $K\subset \bigcup_{i=1}^lB_i$ and we let $\chi_i\in C^{\infty}_0(B_i)$, $i=1,\ldots,l$, be a partition of unity on $K$.
Then 
\begin{multline*}
\langle F_{t_m}\circ u,\nabla \wedge\varphi \rangle_{D}=
\left\langle F_{t_m}\circ u,\nabla \wedge \left(\sum_{i}^l\chi_i\varphi\right) \right\rangle_{D}\\=
\sum_{i=1}^l\langle F_{t_m}\circ u,\nabla \wedge (\chi_i\varphi) \rangle_{B_i}=
\sum_{i}^l\langle V_m,\chi_i\varphi\rangle_{B_i}=\langle V_m,\varphi\rangle_{D}.
\end{multline*}
From here it is easy to conclude that $F_{t_m}\circ u\in H(\mathrm{curl},D)$, with $\nabla\wedge(F_{t_m}\circ u)=V_m$ in the sense of distributions in $D$.
Finally, it is an obvious remark that $F_{t_m}\circ u\in H(\mathrm{curl},D)\cap
L^{\infty}(D,\mathbb{R}^3)$ and that $F_{t_m}\circ u$ converges to $u$, as $m\to \infty$, in the $H(\mathrm{curl})$ norm.
\end{proof}

\begin{proof}[Proof of Proposition~\textnormal{\ref{suffMCP}}]
For any $\mathbf{x}\in \partial D$, let $U_{\mathbf{x}}$ be as in the assumptions.
We call $\tilde{U}_{\mathbf{x}}$ the union of the connected components of $
U_{\mathbf{x}}\cap D$ such that $\mathbf{x}$ belongs to their boundaries.
We call $V_{\mathbf{x}}=(U_{\mathbf{x}}\cap D)\backslash \tilde{U}_{\mathbf{x}}$
and, finally, $\hat{U}_{\mathbf{x}}=U_{\mathbf{x}}\backslash \overline{V_{\mathbf{x}}}$.
We have that $\hat{U}_{\mathbf{x}}$ is an open set containing $\mathbf{x}$ and such that $\hat{U}_{\mathbf{x}}\cap D=\tilde{U}_{\mathbf{x}}$.
Therefore $\hat{U}_{\mathbf{x}}\cap D$ has a finite number of connected components, and each of them may be mapped onto a Lipschitz domain by a
bi-$W^{1,\infty}$ mapping.

By compactness,
we consider $\partial D\subset \bigcup_{i=1}^n\hat{U}_{\mathbf{x}_i}$. Then $\overline{D}\subset \bigcup_{i=0}^nU_i$ with $U_0$ a smooth open set compactly contained in $D$
and $U_i=\hat{U}_{\mathbf{x}_i}$ for any $i=1,\ldots,n$.
Let $\chi_i\in C^{\infty}_0(U_i)$, $i=0,\ldots,n$, be a partition of unity on $\overline{D}$.

By our previous reasoning, for any $i=1,\ldots,n$,
we have $U_i\cap D=\bigcup_{j=1}^{m_i}D_{i,j}$ where each $D_{i,j}$ is a domain and there exists $T_{i,j}:D_{i,j}\to \tilde{D}_{i,j}$ with
$T_{i,j}$ a bi-$W^{1,\infty}$ mapping and $\tilde{D}_{i,j}$ a Lipschitz domain.

By Proposition~\ref{Druprop}, through Corollary~\ref{changecor},
we infer that $D_0=U_0$ and $D_{i,j}$, for any $i=1,\ldots,n$ and $j=1,\ldots,m_i$, satisfy
the MCP (and, with an analogous reasoning, also the RCP). It is easy to prove that
$D_0$ and $D_i=U_i\cap D$,  for any $i=1,\ldots,n$, satisfy the MCP and RCP as well.

Now, let us consider $\{u^m\}_{m\in\mathbb{N}}$ bounded in $H_0(\mathrm{curl},D)\cap H(\mathrm{div},D)$ (or bounded in $H(\mathrm{curl},D)\cap H_0(\mathrm{div},D)$ respectively). We define, for any $m\in\mathbb{N}$ and any $i=0,\ldots,n$,
$u_i^m=\chi_iu^m$.
It is easy to show that, for any $i=0,\ldots,n$,
$\{u^m_i\}_{m\in\mathbb{N}}$ is bounded in $H_0(\mathrm{curl},D_i)\cap H(\mathrm{div},D_i)$ (or bounded in $H(\mathrm{curl},D_i)\cap H_0(\mathrm{div},D_i)$ respectively). Therefore the proof may be easily concluded.\end{proof}

\section*{Acknowledgement}

The work of Hongyu Liu was supported by Hong Kong Baptist University (FRG fund), by Hong Kong RGC grants (projects No. 12302415 and 405513), and by NSF of China (grant No. 11371115). Luca Rondi was partly supported by Universit\`a degli Studi di Trieste (FRA 2014 grants), and by GNAMPA, INdAM.


\end{document}